\documentclass{amsart}
\usepackage{amsmath,amssymb,amsthm}
\usepackage[all]{xy}
\newtheorem{theorem}{Theorem}[section]
\newtheorem{lemma}[theorem]{Lemma}
\newtheorem{proposition}[theorem]{Proposition}
\newtheorem{corollary}[theorem]{Corollary}
\theoremstyle{definition}

\theoremstyle{remark}
\newtheorem{remark}[theorem]{Remark}
\numberwithin{equation}{section}
\title[Biharmonic homogeneous submanifolds]
{Biharmonic homogeneous submanifolds in compact symmetric spaces and compact Lie groups}

\author{Shinji Ohno}
\address{
Osaka City University Advanced Mathematical Institute (OCAMI) 
3-3-138, Sugimoto, Sumiyoshi-ku 
Osaka, 558-8585 Japan }
\email{kudamono.shinji@gmail.com}

\author{Takashi Sakai}
\address{
Department of Mathematics and Information Sciences, 
Tokyo Metropolitan University, 
Minami-Osawa 1-1, Hachioji, Tokyo, 192-0397, Japan.}
\email{sakai-t@tmu.ac.jp}

\author{Hajime Urakawa}
\address{
Institute for International Education, 
Tohoku University, 
Kawauchi 41, Sendai, 980-8576, Japan.}
\email{urakawa@math.is.tohoku.ac.jp}


\date{\today}


\begin{document}

\begin{abstract}
We give a necessary and sufficient condition for
orbits of commutative Hermann actions and actions of the direct product of two symmetric subgroups on compact Lie groups
to be biharmonic in terms of symmetric triad with multiplicities.
By this criterion,
we determine all the proper biharmonic submanifolds in irreducible symmetric spaces of compact type 
which are singular orbits of commutative Hermann actions of cohomogeneity two.
Also, in compact simple Lie groups, we determine all the biharmonic hypersurfaces
which are regular orbits of actions of the direct product of two symmetric subgroups
which are associated to commutative Hermann actions of cohomogeneity one.
\end{abstract}

\maketitle

\tableofcontents

\section{Introduction}
\label{sect:Introduction}

Study of harmonic maps which are critical points of the energy functional
is one of the central problems in differential geometry including  minimal submanifolds.
The Euler-Lagrange equation is given by the vanishing of the tension field. 

In 1983, Eells and Lemaire \cite{EL} proposed 
to study biharmonic maps which are the critical points of the bienergy, 
by definition,
half of the integral of square of the norm of tension 
field $\tau(\varphi)$ 
for a smooth map $\varphi$ of a Riemannian 
manifold $(M,g)$ into another Riemannian manifold $(N,h)$. 
After a pioneering work of G.Y. Jiang \cite{J}, several geometers have studied biharmonic maps 
(see \cite{CMP}, \cite{IIU1}, \cite{IIU2}, \cite{II}, 
\cite{LO}, \cite{MO}, \cite{OT2}, \cite{S}, etc.). 
Notice that harmonic maps are always biharmonic. 
One of central problems is to ask whether the converse is true. 
The {\em B.Y. Chen's conjecture} is 
to ask whether every biharmonic submanifold of the Euclidean space 
${\mathbb R}^n$ must be harmonic, i.e., minimal (\cite{C}).  
It was solved affirmatively in the case surfaces in the three dimensional Euclidean space,
and the case of hypersurfaces of the four dimensional Euclidean space (\cite{D}, \cite{HV}),
and for the case of generic hypersurfaces in the Euclidean space (\cite{KU}). 
Furthermore, 
Akutagawa and Maeta 
showed  (\cite{AM}) that every 
complete properly immersed biharmonic submanifold in the Euclidean space 
${\mathbb R}^n$ must be minimal.  

Moreover, 
R. Caddeo, S. Montaldo, P. Piu and C. Oniciuc 
raised (\cite{CMP}, \cite{Oniciuc}) the {\em generalized B.Y. Chen's conjecture} to ask
whether each biharmonic submanifold in a Riemannian manifold $(N,h)$
of non-positive sectional curvature must be harmonic (minimal).
For the generalized Chen's conjecture,
Ou and Tang gave (\cite{OT1}, \cite{OT2}) a counter example in some Riemannian manifold
of negative sectional curvature. 
However, it is also known (cf. \cite{NU1}, \cite{NU2}, \cite{NUG})
that every biharmonic map of a complete Riemannian manifold
into another Riemannian manifold of non-positive sectional curvature
with finite energy and finite bienergy must be harmonic. 

On the contrary, for the target Riemannian manifold $(N,h)$ of non-negative sectional curvature, 
theories of biharmonic maps and/or biharmonic immersions seems to be quite different
from the case $(N,h)$ of non-positive sectional curvature. 

In 2015, we characterized the biharmonic property of isometric immersions into Einstein manifolds
whose tension field is parallel with respect to the normal connection
in terms of second fundamental form and the curvature tensor,
and determined all the biharmonic hypersurfaces in irreducible symmetric spaces of compact type
which are regular orbits of commutative Hermann actions of cohomogeneity one (cf. \cite{OSU}).
For this purpose,
we used the description of second fundamental forms of orbits of commutative Hermann actions
in terms of symmetric triad with multiplicities, which is given by Ikawa (\cite{I1}).
Recently, the first author (\cite{Ohno}) applied the method of symmetric triad
to study the geometry of orbits of actions of the direct product of two symmetric subgroups
on compact Lie groups, which are associated to commutative Hermann actions.

In this paper, 
we characterize the biharmonic property of orbits of commutative Hermann actions
and the actions of the direct product of two symmetric subgroups on compact Lie groups
in terms of symmetric triad with multiplicities
(cf. Theorems~\ref{thm: Bih. orbits of commutative Hermann actions} and \ref{Thm:charec:Liegrp}).
By this characterization, in Section~\ref{sect:Biharmonic homogeneous submanifolds in compact symmetric spaces},
we determine all the proper biharmonic submanifolds in irreducible symmetric spaces of compact type 
which are singular orbits of commutative Hermann actions of cohomogeneity two.
In the list in Section~\ref{Tables_of_proper_biharmonic_orbits_of_commutative_Hermann_actions},
we obtain a great many examples of proper biharmonic submanifolds in compact symmetric spaces
of higher codimension.
Furthermore, in Section~\ref{sect:Biharmonic homogeneous hypersurfaces in compact Lie groups},
in compact simple Lie groups, 
we determine all the biharmonic hypersurfaces 
which are regular orbits of actions of the direct product of two symmetric subgroups
which are associated to commutative Hermann actions of cohomogeneity one
(cf. Theorem~\ref{thm:list_of_biharmonic_orbits_in_Lie_group}).
We note that recently Inoguchi and Sasahara (\cite{IS}) also investigated biharmonic homogeneous hypersurfaces
in compact symmetric spaces.

\section{Biharmonic isometric immersions}
\label{sect:Biharmonic isometric immersions}

We first recall the definition and fundamentals of harmonic maps and biharmonic maps.
Let $\varphi: (M,g) \rightarrow (N,h)$ be a smooth map from an $m$-dimensional compact Riemannian manifold $(M,g)$
into an $n$-dimensional Riemannian manifold $(N,h)$.
Then $\varphi$ is said to be {\em harmonic}
if it is a critical point of the {\em energy functional} defined by 
$$
E(\varphi) = \frac{1}{2} \int_M \vert d\varphi \vert^2 v_g.
$$
That is, for any variation $\{\varphi_t\}$ of $\varphi$ with $\varphi_0 = \varphi$,
\begin{equation} \label{eq:2.1}
\frac{d}{dt} \bigg\vert_{t=0} E(\varphi_t) = -\int_M h(\tau(\varphi), V) v_g = 0.
\end{equation}
Here $V\in \Gamma(\varphi^{-1}TN)$ is a variation vector field along $\varphi$ which is given by 
$V(x)=\frac{d}{dt}\big\vert_{t=0}\varphi_t(x)\in T_{\varphi(x)}N \ (x\in M)$, 
and $\tau(\varphi)$ is the {\em tension field} of $\varphi$ which is given by 
$\tau(\varphi)=\sum_{i=1}^m B_\varphi(e_i,e_i)\in \Gamma(\varphi^{-1}TN)$, 
where 
$\{e_i\}_{i=1}^m$ is a locally defined orthonormal frame field on $(M,g)$, 
and $B_\varphi$ is the second fundamental form of $\varphi$ defined by 
\begin{align*}
B_\varphi(X,Y)
&=(\widetilde{\nabla}d\varphi)(X,Y) \\
&=(\widetilde{\nabla}_Xd\varphi)(Y) \\
&=\overline{\nabla}_X(d\varphi(Y))-d\varphi(\nabla_XY),
\end{align*}
for all vector fields $X, Y\in \mathfrak{X}(M)$. 
Here we denote by $\nabla$ and $\nabla^h$ the Levi-Civita connections on $TM$, $TN$ of $(M,g)$, $(N,h)$,
and by $\overline{\nabla}$ and $\widetilde{\nabla}$ the induced connections
on $\varphi^{-1}TN$ and $T^{\ast}M\otimes \varphi^{-1}TN$, respectively.
By (\ref{eq:2.1}), $\varphi$ is harmonic if and only if $\tau(\varphi)=0$. 
We note that if $\varphi : M \to N$ is an isometric immersion,
then the tension field $\tau (\varphi)$ coincides with the mean curvature vector field of $\varphi$,
hence $\varphi$ is harmonic if and only if $\varphi$ is a minimal immersion.

J. Eells and L. Lemaire \cite{EL} proposed the notion of biharmonic maps,
and Jiang \cite{J} studied the first and second variation formulas of biharmonic maps.
Let us consider the {\em bienergy functional} defined by 
$$
E_2(\varphi) = \frac{1}{2} \int_M \vert\tau(\varphi)\vert^2 v_g,
$$
where $\vert V\vert^2=h(V,V)$ for $V\in \Gamma(\varphi^{-1}TN)$.
The first variation formula of the bienergy functional is given by
$$
\frac{d}{dt} \bigg\vert_{t=0} E_2(\varphi_t) = -\int_M h(\tau_2(\varphi), V) v_g,
$$
where $\tau_2(\varphi)$ is called the {\em bitension field} of $\varphi$ which is defined by 
$$
\tau_2(\varphi) := J(\tau(\varphi)) = \overline{\Delta}(\tau(\varphi))-\mathcal{R}(\tau(\varphi)).
$$
Here $J$ is the {\em Jacobi operator} acting on $\Gamma(\varphi^{-1}TN)$ given by 
$$
J(V)=\overline{\Delta}V-\mathcal{R}(V),
$$
where 
$\overline{\Delta}V
=\overline{\nabla}^\ast \overline{\nabla}V
=-\sum_{i=1}^m\{\overline{\nabla}_{e_i}\overline{\nabla}_{e_i}V-\overline{\nabla}_{\nabla_{e_i}e_i}V\}$ 
is the rough Laplacian and 
$\mathcal{R}$ is a linear operator on $\Gamma(\varphi^{-1}TN)$ defined by 
$\mathcal{R}(V)=\sum_{i=1}^mR^h(V,d\varphi(e_i))d\varphi(e_i)$,
where $R^h$ is the curvature tensor of $(N,h)$ given by 
$R^h(U,V)W=\nabla^h_U(\nabla^h_VW)-\nabla^h_V(\nabla^h_UW)-\nabla^h_{[U,V]}W$ for $U,V,W\in  \mathfrak{X}(N)$. 
A smooth map $\varphi$ of $(M,g)$ into $(N,h)$ is said to be 
{\em biharmonic} if $\tau_2(\varphi)=0$.
By definition, every harmonic map is biharmonic. 
We say that a smooth map $\varphi:\,(M,g)\rightarrow (N,h)$ is proper biharmonic
if it is biharmonic but not harmonic.

Now we give a characterization theorem for an isometric immersion $\varphi$ 
of a Riemannian manifold $(M,g)$ into another Riemannian manifold $(N,h)$ 
whose tension field $\tau(\varphi)$ satisfies 
$\overline{\nabla}^\perp_X\tau(\varphi)=0$ for all $X \in \mathfrak{X}(M)$ to be biharmonic,
where $\overline{\nabla}^\perp$ is the normal connection on the normal bundle $T^\perp M$.
From Jiang's theorem (\cite{J}), we showed the following theorem. 

\begin{theorem}[\cite{OSU}] \label{theorem1}
Let $\varphi : (M,g) \rightarrow (N,h)$ be an isometric immersion 
which satisfies that $\overline{\nabla}^\perp_X \tau(\varphi) = 0$ for all $X \in \mathfrak{X}(M)$.
Then $\varphi$ is biharmonic if and only if
\begin{equation} \label{eq:2.2}
\sum_{k=1}^m R^h(\tau(\varphi), d\varphi(e_k)) d\varphi(e_k)
= \sum_{j,k=1}^m h(\tau(\varphi), B_{\varphi}(e_j,e_k)) B_{\varphi}(e_j,e_k)
\end{equation}
holds.
\end{theorem}

The condition (\ref{eq:2.2}) is equivalent to the following equation.

\begin{equation} \label{eq:2.3}
\sum_{i=1}^mR^h \big( \tau(\varphi), d\varphi(e_i) \big) d\varphi(e_i)
= \sum_{i=1}^m B_\varphi \big( A_{\tau(\varphi)}e_i, e_i \big).
\end{equation}

\section{Hermann actions and associated $(K_2 \times K_1)$-actions}
\label{sect:Hermann actions}

\subsection{Hermann actions and symmetric triads}

Ikawa (\cite{I1}) introduced the notion of symmetric triad as a generalization of irreducible root system.
He described the second fundamental forms of orbits of commutative Hermann actions in terms of symmetric triads
with multiplicities,
and studied geometric properties of the orbits as submanifolds in compact symmetric spaces.
In this section, we review Ikawa's method, and we show that his method can be also applied to study
geometric properties of orbits of actions of the direct product of two symmetric subgroups on compact Lie groups,
which are associated to Hermann actions.

Let $G$ be a compact connected semisimple Lie group,
and $K_1$, $K_2$ closed subgroups of $G$.
For each $i=1, 2$, we assume that there exists an involutive automorphism $\theta_i$ of $G$
which satisfies $(G_{\theta_i})_0 \subset K_i \subset G_{\theta_i}$,
where $G_{\theta_i}$ is the set of fixed points of $\theta_i$ 
and $(G_i)_0$ is the identity component of $G_{\theta_i}$.
Then $(G, K_1)$ and $(G, K_2)$ are compact symmetric pairs,
and the triple $(G, K_1, K_2)$ is called a {\it compact symmetric triad}.
We denote the Lie algebras of $G$, $K_1$ and $K_2$ by $\mathfrak{g}$, $\mathfrak{k}_1$ and $\mathfrak{k}_2$,
respectively.
The involutive automorphism of $\mathfrak{g}$ induced from $\theta_i$ will be also denoted by $\theta_i$.
Take an $\mathrm{Ad}(G)$-invariant inner product $\langle \cdot, \cdot \rangle$ on $\mathfrak{g}$.
Then the inner product $\langle \cdot, \cdot \rangle$ induces a bi-invariant Riemannian metric on $G$
and $G$-invariant Riemannian metrics on the left coset manifold $N_{1}:=G/K_{1}$
and on the right coset manifold $N_{2}:=K_{2}\backslash G$.
We denote these Riemannian metrics on $G$, $N_{1}$ and $N_{2}$ by the same symbol $\langle \cdot, \cdot \rangle$.
These Riemannian manifolds $G$, $N_{1}$ and $N_{2}$ are Riemannian symmetric spaces
with respect to $\langle \cdot, \cdot \rangle$.
The isometric action of $K_2$ on $G/K_1$ and that of $K_1$ on $K_2 \backslash G$ defined by
\begin{itemize}
\item $K_2 \curvearrowright N_1 :\quad k_2 \cdot \pi_1(x) = \pi_1(k_2 x) \quad (k_2 \in K_2, x \in G)$
\item $K_1 \curvearrowright N_2 :\quad k_1 \cdot \pi_2(x) = \pi_2(x k_1^{-1}) \quad (k_1 \in K_1, x \in G)$
\end{itemize}
are called {\em Hermann actions},
where $\pi_{i}$ denotes the natural projection from $G$ onto $N_{i}\ (i=1, 2)$.
Under this setting, we can also consider the isometric action of $K_2 \times K_1$ on $G$ defined by
\begin{itemize}
\item $K_2 \times K_1 \curvearrowright G :\quad (k_2, k_1) \cdot x = k_2 x k_1^{-1} \quad (k_2 \in K_2, k_1 \in K_1, x \in G)$.
\end{itemize}
The three actions have the same orbit space, 
and in fact the following diagram is commutative:
\[
\xymatrix{
                                    & G \ar[dl]_{\pi_{1}}\ar[dr]^{\pi_{2}} & \\
N_{1} \ar[dr]^{\widetilde{\pi}_{1}} &                                      & N_{2} \ar[dl]_{\widetilde{\pi}_{2}}\\
                                    & K_{2}\backslash G/K_{1}              & \\
}
\]
where $\tilde{\pi}_{i}$ is the natural projection from $N_{i}$ onto the orbit space $K_{2}\backslash G/K_{1}$.
For $x \in G$, we denote the left (resp. right) transformation of $G$ by $L_{x}$ (resp. $R_{x}$).
The isometry on $N_{1}$ (resp. $N_{2}$) induced by $L_{x}$ (resp. $R_{x}$) will be also denoted
by the same symbol $L_{x}$ (resp. $R_{x}$).

For $i=1, 2$, we set
$$
\mathfrak{m}_{i} = \{X \in \mathfrak{g} \mid \theta_{i}(X)=-X \}.
$$ 
Then we have two orthogonal direct sum decompositions of $\mathfrak{g}$, that is the canonical decompositions:
$$
\mathfrak{g} = \mathfrak{k}_1 \oplus \mathfrak{m}_1 = \mathfrak{k}_2 \oplus \mathfrak{m}_2.
$$
The tangent space $T_{\pi_{i}(e)}N_{i}$ of $N_{i}$ at the origin $\pi_{i}(e)$ is identified
with $\mathfrak{m}_{i}$ in a natural way.
We define a closed subgroup $G_{12}$ of $G$ by 
$$
G_{12} = \{x \in G \mid \theta_{1}(x) = \theta_{2}(x)\}.
$$
Then $\theta_1$ induces an involutive automorphism of the identity component $(G_{12})_0$ of $G_{12}$, 
hence $((G_{12})_{0},K_{12})$ is a compact symmetric pair,
where $K_{12}$ is a closed subgroup of $(G_{12})_{0}$ defined by 
$$
K_{12}=\{k \in (G_{12})_{0} \mid \theta_{1}(k)=k\}.
$$
The canonical decomposition of $\mathfrak{g}_{12} = \mathrm{Lie} (G_{12})_0$ with respect to $\theta_1$ is given by 
$$
\mathfrak{g}_{12}=(\mathfrak{k}_{1} \cap \mathfrak{k}_{2}) \oplus (\mathfrak{m}_{1}\cap \mathfrak{m}_{2}).
$$

In general, an isometric action of a compact Lie group on a Riemannian manifold is said to be {\em hyperpolar}
if there exists a closed connected submanifold which is flat in the induced metric and meets all orbits orthogonally.
Such a submanifold is called a {\em section} of the Lie group action.
In our setting, fix a maximal abelian subspace $\mathfrak{a}$ in $\mathfrak{m}_{1} \cap \mathfrak{m}_{2}$.
Then $\exp \mathfrak{a}$ is a torus subgroup in $(G_{12})_{0}$.
Then $\exp \mathfrak{a}$, $\pi_{1}(\exp \mathfrak{a})$ and $\pi_{2}(\exp \mathfrak{a})$
are sections of the $(K_{2}\times K_{1})$-action on $G$, 
the $K_{2}$-action on $N_1$, and the $K_{1}$-action on $N_2$, respectively.
Hence these three actions are hyperpolar,
and their cohomogeneities are equal to $\dim \mathfrak{a}$.
A. Kollross (\cite{K}) classified hyperpolar actions on compact irreducible symmetric spaces.
By the classification, we can see that a hyperpolar action on a compact irreducible symmetric space
whose cohomogeneity is greater than or equal to two is orbit-equivalent to some Hermann action.

In order to describe the orbit spaces of the three actions,
we consider an equivalent relation $\sim$ on $\mathfrak{a}$ defined as follows:
for $H_{1}, H_{2} \in \mathfrak{a}$, 
we define $H_{1}\sim H_{2}$ if $(K_{2}\times K_{1})\cdot \exp (H_{1}) = (K_{2}\times K_{1})\cdot \exp (H_{2})$. 
Clearly, we have $H_{1}\sim H_{2}$ if and only if $K_{2}\cdot \pi_{1}(\exp (H_{1})) = K_{2}\cdot \pi_{1}(\exp (H_{2}))$,
and similarly, $H_{1}\sim H_{2}$ if and only if $K_{1}\cdot \pi_{2}(\exp (H_{1})) = K_{1}\cdot \pi_{2}(\exp (H_{2}))$.
This implies that $\mathfrak{a}/\hspace{-5pt} \sim \cong K_{2}\backslash G /K_{1}$.
For a subgroup $L$ of $G$, we define the normalizer $N_{L}(\mathfrak{a})$ and the centralizer $Z_{L}(\mathfrak{a})$
of $\mathfrak{a}$ in $L$ by
\begin{align*}
N_{L}(\mathfrak{a}) &= \{ k \in L \mid \mathrm{Ad}(k)\mathfrak{a} =\mathfrak{a}\},\\
Z_{L}(\mathfrak{a}) &= \{ k \in L \mid \mathrm{Ad}(k)H =H \ (H \in \mathfrak{a})\}.
\end{align*}
Then $Z_{L}(\mathfrak{a})$ is a normal subgroup of $N_{L}(\mathfrak{a})$.
We define a group $\tilde{J}$ by 
$$
\tilde{J}=\{([s], Y) \in N_{K_{2}}(\mathfrak{a})/Z_{K_{1}\cap K_{2}}(\mathfrak{a}) \ltimes \mathfrak{a} \mid \exp (-Y)s \in K_{1}\}.
$$
The group $\tilde{J}$ naturally acts on $\mathfrak{a}$ 
by the following manner: 
$$
([s], Y)\cdot H= \mathrm{Ad}(s)H+Y \quad \big( ([s], Y) \in \tilde{J},\ H \in \mathfrak{a} \big).
$$
Matsuki (\cite{M}) proved that 
$$
K_{2}\backslash G /K_{1} \cong\mathfrak{a}/\tilde{J}.
$$

Hereafter, we suppose $\theta_{1}\theta_{2}=\theta_{2}\theta_{1}$.
In such a case, $(G, K_1, K_2)$ is called a {\em commutative} compact symmetric triad,
and the $K_2$-action on $N_1$ and the $K_1$-action on $N_2$ are called {\em commutative} Hermann actions.
Then we have an orthogonal direct sum decomposition of $\mathfrak{g}$:
$$
\mathfrak{g}=
(\mathfrak{k}_{1}\cap \mathfrak{k}_{2})\oplus
(\mathfrak{m}_{1}\cap \mathfrak{m}_{2})\oplus
(\mathfrak{k}_{1}\cap \mathfrak{m}_{2})\oplus
(\mathfrak{m}_{1}\cap \mathfrak{k}_{2}).
$$
We define subspaces of $\mathfrak{g}$ as follows:
\begin{align*}
\mathfrak{k}_{0} &= \{ X \in \mathfrak{k}_{1}\cap \mathfrak{k}_{2} \mid [\mathfrak{a}, X] =\{0\}\},\\
V(\mathfrak{k}_{1} \cap \mathfrak{m}_{2}) &= \{ X \in \mathfrak{k}_{1} \cap \mathfrak{m}_{2} \mid [\mathfrak{a}, X] =\{0\}\},\\  
V(\mathfrak{m}_{1} \cap \mathfrak{k}_{2}) &= \{ X \in \mathfrak{m}_{1} \cap \mathfrak{k}_{2} \mid [\mathfrak{a}, X] =\{0\}\}.
\end{align*}
For $\lambda \in \mathfrak{a}$, 
\begin{align*}
\mathfrak{k}_{\lambda} &= \{X \in \mathfrak{k}_{1}\cap \mathfrak{k}_{2} \mid [H,[H,X]] =-\langle \lambda , H \rangle^{2} X \ (H \in \mathfrak{a})\},\\
\mathfrak{m}_{\lambda} &= \{X \in \mathfrak{m}_{1}\cap \mathfrak{m}_{2} \mid [H,[H,X]] =-\langle \lambda , H \rangle^{2} X \ (H \in \mathfrak{a})\},\\
V^{\perp}_{\lambda}(\mathfrak{k}_{1} \cap \mathfrak{m}_{2})
&= \{ X \in \mathfrak{k}_{1} \cap \mathfrak{m}_{2} \mid [H,[H,X]] =-\langle \lambda, H \rangle^{2} X \ (H \in \mathfrak{a})\},\\
V^{\perp}_{\lambda}(\mathfrak{m}_{1} \cap \mathfrak{k}_{2})
&= \{ X \in \mathfrak{m}_{1} \cap \mathfrak{k}_{2} \mid [H,[H,X]] =-\langle \lambda, H \rangle^{2} X \ (H \in \mathfrak{a})\}.
\end{align*}
We set 
\begin{align*}
\Sigma &= \{\lambda \in \mathfrak{a} \setminus \{0\} \mid \mathfrak{k}_{\lambda} \neq \{0\}\},\\ 
W &= \{\alpha \in \mathfrak{a} \setminus \{0\} \mid V^{\perp}_{\alpha}(\mathfrak{k}_{1} \cap \mathfrak{m}_{2}) \neq \{0\}\},\\ 
\tilde{\Sigma} &= \Sigma \cup W.
\end{align*}
It is known that $\dim \mathfrak{k}_{\lambda}= \dim \mathfrak{m}_{\lambda}$
and 
$\dim V^{\perp}_{\lambda}(\mathfrak{k}_{1} \cap \mathfrak{m}_{2}) = \dim V^{\perp}_{\lambda}(\mathfrak{m}_{1} \cap \mathfrak{k}_{2})$
for each $\lambda \in \tilde{\Sigma}$.
Thus we set 
$m(\lambda):=\dim \mathfrak{k}_{\lambda},\ n(\lambda):=\dim V^{\perp}_{\lambda}(\mathfrak{k}_{1} \cap \mathfrak{m}_{2}).$
Notice that $\Sigma$ is the restricted root system of the symmetric pair $((G_{12})_{0}, K_{12})$, 
and $\tilde{\Sigma}$ becomes a root system of $\mathfrak{a}$ (see \cite{I1}).

We define an open subset $\mathfrak{a}_r$ of $\mathfrak{a}$ by
$$
\mathfrak{a}_{r} = \bigcap_{\lambda \in \Sigma, \alpha \in W} \left\{H \in \mathfrak{a} \ \Big| \
\langle \lambda, H \rangle \not\in \pi \mathbb{Z},\ \langle \alpha, H \rangle \not\in \frac{\pi}{2} + \pi \mathbb{Z} \right\}.
$$
A connected component of $\mathfrak{a}_{r}$ is called a {\it cell}.
The {\it affine Weyl group} $\tilde{W}(\tilde{\Sigma}, \Sigma, W)$ of $(\tilde{\Sigma}, \Sigma, W)$ is a subgroup of $\mathrm{Aff}(\mathfrak{a})$ 
generated by 
$$
\left\{\left(s_{\lambda}, \frac{2n\pi}{\langle  \lambda, \lambda \rangle }\lambda\right) \ \bigg| \ \lambda \in \Sigma, n \in \mathbb{Z}\right\}
\cup \left\{\left(s_{\alpha}, \frac{(2n+1)\pi}{\langle  \alpha, \alpha \rangle }\alpha\right) \ \bigg| \ \alpha \in W, n \in \mathbb{Z}\right\}, 
$$
where $\mathrm{Aff}(\mathfrak{a})$ is the affine group of $\mathfrak{a}$ 
which is expressed as the semidirect product $\mathrm{O}(\mathfrak{a}) \ltimes \mathfrak{a}$. 
The action of $\left(s_{\lambda}, (2n\pi/ \langle \lambda , \lambda\rangle )\lambda\right)$ on $\mathfrak{a}$
is the reflection with respect to the hyperplane $\{ H \in \mathfrak{a} \mid \langle \lambda, H \rangle = n\pi \}$, 
and the action of $\left(s_{\alpha}, ((2n+1)\pi/ \langle \alpha,  \alpha\rangle )\alpha\right)$ on $\mathfrak{a}$
is the reflection with respect to the hyperplane $\{ H \in \mathfrak{a} \mid \langle \alpha, H \rangle = (n+1/2)\pi\}$.
The affine Weyl group $\tilde{W}(\tilde{\Sigma}, \Sigma, W)$ acts transitively on the set of all cells.
More precisely, for each cell $P$, it holds that
$$
\mathfrak{a} = \bigcup_{s \in \tilde{W}(\tilde{\Sigma}, \Sigma, W)} s \overline{P}.
$$

For $x = \exp H \ (H \in \mathfrak{a})$,
the orbit $K_2 \cdot \pi_1(x)$ in $N_1$ is regular, so $(K_2 \times K_1) \cdot x$ in $G$ is, if and only if $H \in \mathfrak{a}_{r}$.
Here we call an orbit {\em regular} if it is an orbit of the maximal dimension.
In \cite{I1}, it is proved that $\tilde{W}(\tilde{\Sigma}, \Sigma, W)$ is a subgroup of $\tilde{J}$.
Moreover, if $N_1$ and $N_2$ are simply-connected,
then $\tilde{W}(\tilde{\Sigma}, \Sigma, W) = \tilde{J}$,
hence the orbit space $K_2 \backslash G/K_1$ of the actions can be identified with
$\mathfrak{a}/\tilde{J} = \mathfrak{a}/\tilde{W}(\tilde{\Sigma}, \Sigma, W) \cong \overline{P}$.
Indeed, for each orbit $K_2 \cdot \pi_1(x)$ in $N_1$ and $(K_2 \times K_1) \cdot x$ in $G$,
there exists $H \in \overline{P}$ uniquely so that $x = \exp H$.
An interior point $H$ in $\overline{P}$ corresponds to a regular orbit,
and a point $H$ in the boundary of $\overline{P}$ corresponds to a singular orbit.
In fact, $\overline{P}$ is a closed region in $\mathfrak{a}$, which is a direct product of some simplexes.
Then the cell decomposition of $\overline{P}$ gives a stratification of orbit types of the action.
We should note that, in general,
the cell decomposition of $\overline{P}$ gives a stratification of {\em local} orbit types of the action.
In this paper, we study the biharmonicity of the orbits, that is a local property of a submanifold.
Therefore, without loss of generality,
we may assume that $N_1$ and $N_2$ are simply-connected, hence $K_2 \backslash G/K_1 \cong \overline{P}$.

Ikawa (\cite{I1}) introduced the notion of {\em symmetric triad} with multiplicities
as a generalization of irreducible root system.
For the precise definition of symmetric triad with multiplicities,
we refer to Ikawa's papers (\cite{I1, I2, I3}).
In general, the triad $(\tilde{\Sigma}, \Sigma, W)$,
which is obtained from a compact symmetric triad $(G, K_1, K_2)$,
is not a symmetric triad with multiplicities in the sense of Ikawa.
However, we know the following theorem.
 
\begin{theorem}[\cite{I2} Theorem 3.1, \cite{I3} Theorem 1.14] \label{Theorem:I2 AB}
Let $(G, K_{1}, K_{2})$ be a compact symmetric triad which satisfies one of the following conditions.
\begin{enumerate}
\item[(A)] $G$ is simple and $\theta_1 \not\sim \theta_2$,
i.e. $\theta_{1}$ and $\theta_{2}$ can not be transformed each other by an inner automorphism of $\mathfrak{g}$. 
\item[(B)] There exist a compact connected simple Lie group $U$ and a symmetric subgroup $\overline{K}$ of $U$ such that
$$
G=U\times U, \quad K_{1}=\Delta G= \{(u, u) \mid u \in U\}, \quad K_{2}=\overline{K}\times \overline{K}.
$$
\item[(C)] There exist a compact connected simple Lie group $U$ and an involutive outer automorphism $\sigma$ such that
\begin{align*}
G=U\times U, \quad K_{1}=\Delta G= \{(u, u) \mid u \in U\},\\
 \quad K_{2}=\{(u_{1}, u_{2}) \mid  (\sigma (u_{2}),\ \sigma (u_{1}))=(u_{1}, u_{2})\}.
\end{align*}
\end{enumerate}
Then the triple $(\tilde{\Sigma}, \Sigma, W)$ defined as above is a symmetric triad of $\mathfrak{a}$,
moreover $m(\lambda)$ and $n(\alpha)$ are multiplicities of $\lambda \in \Sigma$ and $\alpha \in W$.
Conversely every symmetric triad is obtained in this way.
\end{theorem}

When $(G, K_{1}, K_{2})$ satisfies (A), (B) or (C) in Theorem~\ref{Theorem:I2 AB},
hence $(\tilde{\Sigma}, \Sigma, W)$ is a symmetric triad,
we take a fundamental system $\tilde{\Pi}$ of $\tilde{\Sigma}$.
We denote by $\tilde{\Sigma}^{+}$ the set of positive roots in $\tilde{\Sigma}$.
Set $\Sigma^{+} = \tilde{\Sigma}^{+} \cap \Sigma$ and $W^{+} = \tilde{\Sigma}^{+} \cap W$.
Denote by $\Pi$ the set of simple roots of $\Sigma$.
We set 
$$
W_{0} = \{\alpha \in W^{+} \mid \alpha + \lambda \not\in W \ (\lambda \in \Pi) \}.
$$
From the classification of symmetric triads,
we have that $W_{0}$ consists of the only one element, denoted by $\tilde{\alpha}$. 
We define an open subset $P_0$ of $\mathfrak{a}$ by
\begin{equation} \label{eq:cell_of_Hermann_action}
P_{0} = \left\{H \in \mathfrak{a} \ \Big| \ \langle \tilde{\alpha}, H \rangle < \frac{\pi}{2},\
\langle \lambda, H \rangle >0 \ (\lambda \in \Pi) \right\}.
\end{equation}
Then $P_{0}$ is a cell. 
For a nonempty subset $\Delta \subset \Pi \cup \{\tilde{\alpha}\}$, 
set
\begin{align*}
P_{0}^{\Delta}=\left\{ H \in \overline{P}_{0} \left|
\begin{array}{c}
\langle \lambda , H \rangle > 0 \ (\lambda \in \Delta \cap \Pi )\\
\langle \mu , H \rangle = 0 \ (\mu \in \Pi \setminus \Delta ) \\
\langle \tilde{\alpha}, H \rangle 
\begin{cases}
 < (\pi/2) \ (\text{if}\ \tilde{\alpha} \in \Delta )\\
 = (\pi/2) \ (\text{if}\ \tilde{\alpha} \not\in \Delta )
\end{cases}
\end{array}
\right. \right\}.
\end{align*}
Then 
\begin{equation} \label{cell_decomposition_of_Hermann_action}
\overline{P}_{0}=\bigcup_{\Delta \subset \Pi \cup \{\tilde{\alpha}\}} P_{0}^{\Delta} \ (\text{disjoint union}).
\end{equation}

When $(G, K_1, K_2)$ satisfies $\theta_1 \sim \theta_2$,
i.e. $\theta_{1}$ and $\theta_{2}$ are transformed each other by an inner automorphism of $\mathfrak{g}$,
the Hermann action of $K_2$ on the compact symmetric space $G/K_1$ is equivalent of
the action of the isotropy group $K_1$ on $G/K_1$.
Hence we may assume that $\theta_1 = \theta_2$, and so $K_1 = K_2$.
Since $W = \emptyset$, then $(\tilde{\Sigma}, \Sigma, W)$ is not a symmetric triad,
and $\tilde{\Sigma} = \Sigma$ is the restricted root system of $(G, K_1)$.
In this case, we can describe the orbit space of $K_1$-action on $G/K_1$
in terms of the restricted root system $\Sigma$ of $(G, K_1)$.
For simplicity, here we assume that $G$ is simple.
We take a fundamental system $\Pi$ of $\Sigma$,
and denote the set of positive roots by $\Sigma^+$ and the highest root by $\delta$.
We define an open subset $P_0$ of $\mathfrak{a}$ by
\begin{equation} \label{eq:cell_of_isotropy_action}
P_{0} = \left\{H \in \mathfrak{a} \ \Big| \ \langle \delta, H \rangle < \pi,\
\langle \lambda, H \rangle >0 \ (\lambda \in \Pi) \right\}.
\end{equation}
Then $P_{0}$ is a cell. 
For a nonempty subset $\Delta \subset \Pi \cup \{\delta\}$, 
set
\begin{align*}
P_{0}^{\Delta}=\left\{ H \in \overline{P}_{0} \left|
\begin{array}{c}
\langle \lambda, H \rangle > 0 \ (\lambda \in \Delta \cap \Pi)\\
\langle \mu, H \rangle = 0 \ (\mu \in \Pi \setminus \Delta) \\
\langle \delta, H \rangle 
\begin{cases}
 < \pi \ (\text{if}\ \delta \in \Delta )\\
 = \pi \ (\text{if}\ \delta \not\in \Delta)
\end{cases}
\end{array}
\right. \right\}.
\end{align*}
Then 
\begin{equation} \label{eq:cell_decomposition_of_isotropy_action}
\overline{P}_{0}=\bigcup_{\Delta \subset \Pi \cup \{\delta\}} P_{0}^{\Delta} \ (\text{disjoint union}).
\end{equation}

\subsection{Second fundamental forms of orbits }

We express the second fundamental forms and mean curvature vector fields
of orbits of commutative Hermann actions
and their associated $K_2 \times K_1$-actions.

Here, let $G$ be a compact connected semisimple Lie group and $(G, K_1, K_2)$ a commutative compact symmetric triad.
Fix a maximal abelian subspace $\mathfrak{a}$ in $\mathfrak{m}_1 \cap \mathfrak{m}_2$,
then we have a triad $(\tilde{\Sigma}, \Sigma, W)$.
We take a fundamental system $\tilde{\Pi}$ of $\tilde{\Sigma}$,
and set 
$$
\tilde{\Sigma}^{+}=\{\lambda \in \tilde{\Sigma} \mid \lambda >0\},\quad
\Sigma^{+}=\Sigma \cap \tilde{\Sigma}^{+},\quad
W^{+}=W\cap \tilde{\Sigma}^{+}.
$$
Then we have an orthogonal direct sum decomposition of $\mathfrak{g}$:
\begin{align*}
\mathfrak{g} = \mathfrak{k}_{0} \oplus \sum_{\lambda \in \Sigma^{+}} \mathfrak{k}_{\lambda}
\oplus \mathfrak{a} \oplus \sum_{\lambda \in \Sigma^{+}} \mathfrak{m}_{\lambda}
&\oplus V(\mathfrak{k}_{1} \cap \mathfrak{m}_{2}) \oplus \sum_{\alpha \in W^{+}} V^{\perp}_{\alpha}(\mathfrak{k}_{1} \cap \mathfrak{m}_{2})\\
&\oplus V(\mathfrak{m}_{1} \cap \mathfrak{k}_{2}) \oplus \sum_{\alpha \in W^{+}} V^{\perp}_{\alpha}(\mathfrak{m}_{1} \cap \mathfrak{k}_{2}).
\end{align*}

According to the above orthogonal direct sum decomposition of $\mathfrak{g}$,
we have the following lemma.

\begin{lemma}[\cite{I1} Lemmas 4.3 and 4.16] \label{onb} \quad
{\rm (1)} For each $\lambda \in \Sigma^{+}$,
there exist orthonormal bases $\{S_{\lambda, i}\}_{i=1}^{m(\lambda)}$
and $\{T_{\lambda, i}\}_{i=1}^{m(\lambda)}$ of $\mathfrak{k}_{\lambda}$ and $\mathfrak{m}_{\lambda}$ respectively such that 
for any $H \in \mathfrak{a}$, 
$$
[H, S_{\lambda, i}] = \langle \lambda, H \rangle T_{\lambda, i},\quad
[H, T_{\lambda, i}] = -\langle \lambda, H \rangle S_{\lambda , i},\quad
[S_{\lambda, i}, T_{\lambda, i}] = \lambda,
$$
\begin{align*}
\mathrm{Ad}(\exp H) S_{\lambda, i} &= \cos \langle \lambda, H \rangle S_{\lambda, i} + \sin \langle \lambda, H \rangle T_{\lambda, i},\\
\mathrm{Ad}(\exp H) T_{\lambda, i} &= -\sin \langle \lambda, H \rangle S_{\lambda, i} + \cos \langle \lambda, H \rangle T_{\lambda, i}.
\end{align*}
{\rm (2)} For each $\alpha \in W^{+}$,
there exist orthonormal bases $\{X_{\alpha, j}\}_{j=1}^{n(\alpha)}$
and $\{Y_{\alpha, j}\}_{j=1}^{n(\alpha)}$ of $V^{\perp}_{\alpha}(\mathfrak{k}_{1} \cap \mathfrak{m}_{2})$
and $V^{\perp}_{\alpha}(\mathfrak{m}_{1} \cap \mathfrak{k}_{2})$ respectively such that for any $H \in \mathfrak{a}$
$$
[H, X_{\alpha, j}] = \langle \alpha, H \rangle Y_{\alpha, j},\quad
[H, Y_{\alpha, j}] = -\langle \alpha, H \rangle X_{\alpha, j},\quad
[X_{\alpha, j}, Y_{\alpha, j}] = \alpha,
$$
\begin{align*}
\mathrm{Ad}(\exp H) X_{\alpha, j} &= \cos \langle \alpha, H \rangle X_{\alpha, j} + \sin \langle \alpha, H \rangle Y_{\alpha, j},\\
\mathrm{Ad}(\exp H) Y_{\alpha, j} &= -\sin \langle \alpha, H \rangle X_{\alpha, j} + \cos \langle \alpha, H \rangle Y_{\alpha, j}.
\end{align*}
\end{lemma}

First, for $x \in G$, we consider an orbit $K_2 \cdot \pi_1(x)$ of the commutative Hermann action of $K_2$ on $N_1$.
Without loss of generality we can assume that $x = \exp H $ where $H \in \mathfrak{a}$,
since $\pi_1(\exp \mathfrak{a})$ is a section of the action.
For $H \in \mathfrak{a}$, we set
\begin{align*}
\Sigma_{H}&=\{ \lambda \in \Sigma \mid \langle \lambda, H \rangle \in \pi\mathbb{Z} \}, \quad 
W_{H}=\{\alpha \in W \mid \langle \alpha , H \rangle \in (\pi/2) +\pi \mathbb{Z} \}, \\
\tilde{\Sigma}_{H}&=\Sigma_{H}\cup W_{H}, \quad
\Sigma_{H}^{+}= \Sigma^{+}\cap \Sigma_{H}, \quad
W_{H}^{+}=W^{+}\cap W_{H}, \quad
\tilde{\Sigma}_{H}^{+}=\Sigma_{H}^{+}\cup W_{H}^{+}.
\end{align*}
Then the tangent space
\begin{align*}
T_{\pi_1(x)}(N_1)
&= dL_{x}(\mathfrak{m}_1) = dL_{x} \big( (\mathfrak{m}_1 \cap \mathfrak{m}_2 ) \oplus (\mathfrak{m}_1 \cap \mathfrak{k}_2 ) \big) \\
&= dL_{x} \left( \mathfrak{a} \oplus \sum_{\lambda \in \Sigma^{+}} \mathfrak{m}_{\lambda}
\oplus V(\mathfrak{m}_{1} \cap \mathfrak{k}_{2})
\oplus \sum_{\alpha \in W^{+}} V^{\perp}_{\alpha}(\mathfrak{m}_{1} \cap \mathfrak{k}_{2}) \right)
\end{align*}
of $N_1$ at $\pi_1(x)$ is decomposed to
the tangent space $T_{\pi_{1}(x)}(K_{2}\cdot \pi_{1}(x))$
and the normal space $T^{\perp}_{\pi_{1}(x)}(K_{2}\cdot \pi_{1}(x))$
of the orbit $K_{2}\cdot \pi_{1}(x)$ as follows.
\begin{align*}
T_{\pi_{1}(x)}(K_{2}\cdot \pi_{1}(x))
&= \left\{ \left. \left. \frac{d}{dt} \exp(tX_{2})\cdot \pi_{1}(x) \right|_{t=0} \ \right| X_{2}\in \mathfrak{k}_{2}\right\} \\ 
&= dL_{x}(d\pi_{1}(\mathrm{Ad}(x)^{-1} \mathfrak{k}_{2})) \\
&= dL_{x}\left( 
\sum_{\lambda \in \Sigma^{+}\setminus \Sigma_{H}} \mathfrak{m}_{\lambda} 
\oplus V(\mathfrak{m}_{1}\cap \mathfrak{k}_{2}) 
\oplus \sum_{\alpha  \in W^{+}     \setminus W_{H}}      V_{\alpha}^{\perp}(\mathfrak{m}_{1}\cap \mathfrak{k}_{2})  \right),\\
T^{\perp}_{\pi_{1}(x)}(K_{2}\cdot \pi_{1}(x))
&= dL_{x}((\mathrm{Ad}(x)^{-1}\mathfrak{m}_{2}) \cap \mathfrak{m}_{1})\\
&= dL_{x}
\left( 
\mathfrak{a} \oplus 
\sum_{\lambda \in \Sigma_{H}^{+}} \mathfrak{m}_{\lambda} \oplus 
\sum_{\alpha  \in W_{H}^{+}} V_{\alpha}^{\perp}(\mathfrak{m}_{1}\cap \mathfrak{k}_{2}) \right).
\end{align*}
From Lemma \ref{onb}, Ikawa proved the following theorems.
\begin{theorem}[\cite{I1} Lemma 4.22] \label{2nd}
For $x = \exp H \ (H \in \mathfrak{a})$, 
we denote the second fundamental forms of the orbits $K_{2}\cdot \pi_{1}(x)$ in $N_{1}$ by $B_{H}^1$.
Then we have
\begin{enumerate}
\item $dL_{x}^{-1}B_{H}^1 (dL_{x}(T_{\lambda, i}), dL_{x}(T_{\mu, j})) = \cot(\langle \mu, H \rangle) [T_{\lambda, i}, S_{\mu, j}]^{\perp}$,
\item $dL_{x}^{-1}B_{H}^1 (dL_{x}(Y_{\alpha, i}), dL_{x}(Y_{\beta, j})) = -\tan(\langle \beta, H \rangle) [Y_{\alpha, i}, X_{\beta, j}]^{\perp}$,
\item $B_{H}^1 (dL_{x}(Y_{1}), dL_{x}(Y_{2}))=0$,
\item $B_{H}^1 (dL_{x}(T_{\lambda, i}), dL_{x}(Y_2)) = 0$,
\item $B_{H}^1 (dL_{x}(Y_{\alpha, i}), dL_{x}(Y_2)) = 0$,
\item $dL_{x}^{-1}B_{H}^1 (dL_{x}(T_{\lambda, i}), dL_{x}(Y_{\beta, j})) = -\tan(\langle \beta, H \rangle) [T_{\lambda, i}, X_{\beta, j}]^{\perp}$,
\end{enumerate}
for
\begin{align*}
&\lambda, \mu \in \Sigma^{+} \ \text{with} \ \langle \lambda, H \rangle, \langle \mu, H \rangle \not\in \pi \mathbb{Z},\ 1 \leq i \leq m(\lambda),\ 1 \leq j \leq m(\mu),\\
&\alpha, \beta \in W^{+} \ \text{with} \ \langle \alpha, H \rangle, \langle \beta, H \rangle \not\in \frac{\pi}{2} + \pi \mathbb{Z},\ 1 \leq i \leq n(\alpha),\ 1 \leq j \leq n(\beta),\\
&Y_{1}, Y_{2} \in V(\mathfrak{m}_{1} \cap \mathfrak{k}_{2}).
\end{align*}
Here $X^{\perp}$ is the normal component, i.e. $(\mathrm{Ad}(x^{-1}) \mathfrak{m}_{2}) \cap \mathfrak{m}_{1}$-component,
of a tangent vector $X \in \mathfrak{m}_{1}$. 
\end{theorem}

\begin{theorem}[\cite{I1} Corollaries 4.23, 4.29, 4.24, and \cite{GT} Theorem 5.3] \label{Theorem3.2}
For $x = \exp H \ (H \in \mathfrak{a})$,
we denote the mean curvature vector field of $K_{2}\cdot \pi_{1}(x)$ in $N_{1}$ by $\tau^{1}_{H}$.
Then we have
$$
dL_{x}^{-1}(\tau^{1}_{H})_{\pi_{1}(x)}
= -\sum_{\lambda \in \Sigma^{+} \setminus \Sigma_{H}} m(\lambda ) \cot \langle \lambda, H \rangle \lambda 
+ \sum_{\alpha \in W^{+} \setminus W_{H}} n(\alpha) \tan \langle \alpha, H \rangle \alpha. 
$$
\end{theorem} 
We can also apply Theorem~\ref{Theorem3.2} for the orbit $K_{1}\cdot \pi_{2}(x)$ in $N_{2}$. 
Thus we have the following corollary.
\begin{corollary}[\cite{I1} Corollary 4.30]
The orbit $K_{2}\cdot \pi_{1}(x)$ is minimal if and only if 
$K_{1}\cdot \pi_{2}(x)$ is minimal. 
\end{corollary}

Next we consider the second fundamental forms of orbits of the $(K_{2}\times K_{1})$-action on $G$. 
For $x=\exp H \ (H \in \mathfrak{a})$,
the tangent space $T_{x}((K_{2}\times K_{1})\cdot x)$ and the normal space $T^{\perp}_{x}((K_{2}\times K_{1})\cdot x)$
of the orbit $(K_{2}\times K_{1})\cdot x$ at $x$ are given as follows.
\begin{align}
\lefteqn{T_{x}((K_{2}\times K_{1})\cdot x)} \hspace{10mm} \nonumber \\
&= \left\{ \left. \left. \frac{d}{dt} \exp(tX_{2}) x \exp(-tX_{1}) \right|_{t=0} \ \right| \
X_{1}\in \mathfrak{k}_{1},\ X_{2}\in \mathfrak{k}_{2}\right\} \nonumber \\
&= dL_{x}((\mathrm{Ad}(x)^{-1} \mathfrak{k}_{2}) +\mathfrak{k}_{1}) \nonumber \\
&= dL_{x} \left( 
\mathfrak{k}_{0}
\oplus \sum_{\lambda \in \Sigma^{+}} \mathfrak{k}_{\lambda} 
\oplus \sum_{\lambda \in \Sigma^{+}\setminus \Sigma_{H}} \mathfrak{m}_{\lambda} 
\oplus V(\mathfrak{k}_{1}\cap \mathfrak{m}_{2}) \right. \label{tangent} \\
&\left. \hspace{10mm} \oplus \sum_{\alpha \in W^{+}} V_{\alpha}^{\perp}(\mathfrak{k}_{1}\cap \mathfrak{m}_{2})
\oplus V(\mathfrak{m}_{1}\cap \mathfrak{k}_{2}) 
\oplus \sum_{\alpha \in W^{+} \setminus W_{H}} V_{\alpha}^{\perp}(\mathfrak{m}_{1}\cap \mathfrak{k}_{2})
\right), \nonumber \\
\lefteqn{T^{\perp}_{x}((K_{2}\times K_{1})\cdot x)} \hspace{10mm} \nonumber \\
&= dL_{x}((\mathrm{Ad}(x)^{-1}\mathfrak{m}_{2}) \cap \mathfrak{m}_{1}) \nonumber \\
&=dL_{x} \left( 
\mathfrak{a} \oplus 
\sum_{\lambda \in \Sigma_{H}^{+}} \mathfrak{m}_{\lambda} \oplus 
\sum_{\alpha  \in W_{H}^{+}} V_{\alpha}^{\perp}(\mathfrak{m}_{1}\cap \mathfrak{k}_{2}) \right).
\end{align}

For $X=(X_{2}, X_{1}) \in \mathfrak{g}\times \mathfrak{g}$, 
we define a Killing vector field $X^{\ast}$ on $G$ by
$$
(X^{\ast})_{y}=\left. \frac{d}{dt} \exp (tX_{2})y\exp(-tX_{1})\right|_{t=0} \quad (y \in G).
$$
Then  
$$
(X^{\ast})_{y}=(dL_{y})(\mathrm{Ad}(y)^{-1}X_{2}-X_{1})
$$
holds. 
If $X_{2}=0$, then $X^{\ast}$ is a left invariant vector field. 
Denote by $\nabla$ the Levi-Civita connection on $G$. 
By Koszul's formula, we have the following. 
\begin{lemma}[\cite{Ohno} Lemma 3] \label{Lemma:connection of G}
Let $y \in G, \  X=(X_{2},X_{1}), \ Y=(Y_{2},Y_{1}) \in \mathfrak{g}\times \mathfrak{g}$. 
Then we have
\begin{align*}
\left( \nabla_{X^{\ast}} Y^{\ast}\right)_{y} =-\frac{1}{2}dL_{y} [\mathrm{Ad}(y)^{-1}X_{2}-X_{1}, \mathrm{Ad}(y)^{-1}Y_{2}+Y_{1}].
\end{align*}
\end{lemma}

By Lemma \ref{Lemma:connection of G}, 
the first author proved the following theorems.
\begin{theorem}[\cite{Ohno} Theorem 3]\label{Lemma:2nd fand. form of of orbits of proper action}
For $x = \exp H \ (H \in \mathfrak{a})$, 
we denote the second fundamental form of the orbit $(K_{2}\times K_{1})\cdot x$ in $G$ by $B_{H}$.
We set
\begin{align*}
V_{1}&=\sum_{\lambda \in \Sigma^{+} \setminus \Sigma_{H} } \mathfrak{m}_{\lambda} 
\oplus  \sum_{\alpha  \in W^{+} \setminus W_{H}} V_{\alpha}^{\perp}(\mathfrak{m}_{1}\cap \mathfrak{k}_{2}), \\ 
V_{2}&=\sum_{\lambda \in \Sigma^{+} } \mathfrak{k}_{\lambda} 
  \oplus \sum_{\alpha  \in W^{+} } V_{\alpha}^{\perp}(\mathfrak{k}_{1}\cap \mathfrak{m}_{2}).
\end{align*} 
Then we have
\begin{enumerate}
\item For $X \in \mathfrak{k}_{0}$, \label{2nd fand. form 1} $B_{H}(dL_{x}(X),Y)=0 \quad $ where $Y \in T_{x}((K_{2}\times K_{1})\cdot x).$
\item For $X \in V(\mathfrak{k}_{1}\cap \mathfrak{m}_{2})$, \label{2nd fand. form 2}
\begin{align*}
&dL_{x}^{-1}B_{H}(dL_{x}(X),dL_{x}(Y))
=
\begin{cases}
0 & (Y \in \mathfrak{k}_{1} \oplus V(\mathfrak{m}_{1}\cap \mathfrak{k}_{2}))\\
\displaystyle -\frac{1}{2}[X,Y]^{\perp}
& \left( Y \in V_{1} \right).
\end{cases}
\end{align*}
\item For $X \in V(\mathfrak{m}_{1}\cap \mathfrak{k}_{2})$, \label{2nd fand. form 3}
\begin{align*}
&dL_{x}^{-1}B_{H}(dL_{x}(X),dL_{x}(Y))
=
\begin{cases}
0 
&\left(  Y\in V(\mathfrak{m}_{1}\cap \mathfrak{k}_{2}) \oplus V_{1} \right) \\
\displaystyle \frac{1}{2}[X,Y]^{\perp} 
&  \left( Y \in V_{2} \right).
\end{cases}
\end{align*}
\item For $S_{\lambda ,i}\ (\lambda \in \Sigma^{+},\ 1\leq i \leq m(\lambda ))$, \label{2nd fand. form 4}
\begin{align*}
&dL_{x}^{-1}B_{H}(dL_{x}(S_{\lambda ,i}),dL_{x}(Y))
=
\begin{cases}
0 
&\left( Y \in V_{2} \right) \\
\displaystyle -\frac{1}{2}[S_{\lambda ,i},Y]^{\perp} 
&\left(Y\in V_{1} \right).
\end{cases}
\end{align*}
\item For $X_{\alpha ,i}\ (\alpha \in W^{+},\ 1\leq i \leq n(\alpha ))$, \label{2nd fand. form 5}
\begin{align*}
&dL_{x}^{-1}B_{H}(dL_{x}(X_{\alpha ,i}),dL_{x}(Y))
=
\begin{cases}
0 
& \left( Y \in V_{2} \right)\\
\displaystyle -\frac{1}{2}[X_{\alpha ,i} ,Y]^{\perp} 
& \left(Y\in V_{1}\right).
\end{cases}
\end{align*}
\item For $T_{\lambda ,i}\ (\lambda \in \Sigma^{+}\setminus \Sigma_{H},\ 1\leq i \leq m(\lambda ))$, \label{2nd fand. form 6}
	\begin{itemize}
	\item $dL_{x}^{-1}B_{H}(dL_{x}(T_{\lambda ,i}), dL_{x}(T_{\mu , j}))=\cot \langle \mu , H \rangle [T_{\lambda ,i}, S_{\mu , j}]^{\perp}$  \\ 
where $\mu \in \Sigma^{+}\setminus \Sigma_{H},\ 1\leq j \leq m(\mu).$
	\item $dL_{x}^{-1}B_{H}(dL_{x}(T_{\lambda ,i}), dL_{x}(Y_{\beta , j}))=-\tan \langle \beta , H \rangle [T_{\lambda ,i}, X_{\beta , j}]^{\perp}$ \\
where $\beta \in W^{+}\setminus W_{H},\ 1\leq j \leq n(\beta).$
	\end{itemize}
\item For $Y_{\alpha ,i}\ (\alpha \in W^{+}\setminus W_{H},\ 1\leq i \leq n(\alpha ))$, \label{2nd fand. form 7}
$$dL_{x}^{-1}B_{H}(dL_{x}(Y_{\alpha ,i}), dL_{x}(Y_{\beta , j}))=-\tan \langle \beta , H \rangle [Y_{\alpha ,i}, X_{\beta , j}]^{\perp}$$
where $\beta \in W^{+}\setminus W_{H},\ 1\leq j \leq n(\beta).$
\end{enumerate}
Here, 
$X^{\perp}$ is the normal component, i.e. the $((\mathrm{Ad}(x)^{-1}\mathfrak{m}_{2}) \cap \mathfrak{m}_{1})$-component, of a tangent vector $X \in \mathfrak{g}$.
\end{theorem}

By Theorems~\ref{Theorem3.2} and \ref{Lemma:2nd fand. form of of orbits of proper action}, 
we obtain the following corollary.

\begin{corollary}[\cite{Ohno} Corollary 2] \label{cor:2}
For $x = \exp H \ (H \in \mathfrak{a})$,
we denote the mean curvature vector field of the orbit $(K_{2}\times K_{1})\cdot x$ in $G$ by $\tau_{H}$.
Then, 
$$
dL_{x}^{-1}(\tau_{H})_{x}= -\sum_{\lambda \in \Sigma^{+} \setminus \Sigma_{H}} m(\lambda )\cot \langle \lambda , H\rangle \lambda 
+\sum_{\alpha \in W^{+} \setminus W_{H}} n(\alpha )\tan \langle \alpha , H\rangle \alpha . 
$$
Moreover, $dL_{x}^{-1}(\tau_{H})_{x}=dL_{x}^{-1}(\tau^{1}_{H})_{\pi_{1}(x)}$ holds. 
Hence, the orbit $(K_{2}\times K_{1})\cdot x$ in $G$ is minimal if and only if $K_{2}\cdot \pi_{1}(x)$ in $N_{1}$ is minimal. 
\end{corollary}

We show the following properties of the mean curvature vector field $\tau_H$ of $(K_{2}\times K_{1})\cdot x$ in $G$
and $\tau_H^1$ of $K_{2} \cdot \pi_1(x)$ in $N_1$.
\begin{proposition}\label{prop2.14new}
For $H \in \mathfrak{a}$ and $\sigma =([s], Y) \in \tilde{J}$, 
we set $H'=\sigma\cdot H \in \mathfrak{a}$, $x=\exp(H)$ and $x'=\exp(H')$.
Then $(K_{2}\times K_{1})\cdot x=(K_{2}\times K_{1})\cdot x'$
and 
$$
dL_{x'}^{-1}(\tau_{H})_{x'}=[s]\cdot dL_{x}^{-1}(\tau_{H})_{x}.
$$
\end{proposition}

\begin{proof}
By the definition of $\tilde{J}$, 
there exists $s \in N_{K_{2}}(\mathfrak{a})$ such that $\mathrm{Ad}(s)|_{\mathfrak{a}}=[s]$ and 
$(s, \exp(-Y)s) \in K_{2}\times K_{1}$.
Then we have
$$
(s, \exp(-Y)s)\cdot \exp(H)=s\exp(H)s^{-1}\exp(Y) 
=\exp(\mathrm{Ad}(s)H+Y)=\exp(H').
$$
Thus, $(K_{2}\times K_{1})\cdot x=(K_{2}\times K_{1})\cdot x'$.
Since $L_{s}\circ R_{s^{-1}\exp(Y)}$ is an isometry, we have 
\begin{align*}
(\tau_{H})_{x'}
&=(\tau_{H})_{L_{s}\circ R_{s^{-1}\exp(Y)}(x)}\\
&=dL_{s}\circ dR_{s^{-1}\exp(Y)} ((\tau_{H})_{x})\\
&=\left.\frac{d}{dt} s\exp(H) \exp(t dL_{x}^{-1}(\tau_{H})_{x})s^{-1}\exp(Y) \right|_{t=0}\\
&=\left.\frac{d}{dt} \exp\left( \mathrm{Ad}(s)(t dL_{x}^{-1}(\tau_{H})_{x}+H)\right) \exp(Y) \right|_{t=0}\\
&=\left.\frac{d}{dt} \exp\left( \mathrm{Ad}(s)H+Y\right) \exp \big( t\mathrm{Ad}(s)( dL_{x}^{-1}(\tau_{H})_{x}) \big) \right|_{t=0}\\
&=dL_{x'}(\mathrm{Ad}(s) dL_{x}^{-1}(\tau_{H})_{x})\\
&=dL_{x'}([s] \cdot dL_{x}^{-1}(\tau_{H})_{x}).
\end{align*}
\end{proof}

By Lemmas~4.4 and 4.21 in \cite{I1},
we have that
$\tilde{W}(\tilde{\Sigma}, \Sigma, W)$ is a subgroup of $\tilde{J}$.
Then we have the following Lemma.
\begin{lemma}\label{lemma:perpnew}
For $x = \exp H \ (H \in \mathfrak{a})$, we have 
$$
\langle \lambda, dL_{x}^{-1}(\tau_{H})_{x}\rangle=0 \quad (\lambda \in \tilde{\Sigma}_{H}).
$$
\end{lemma}

\begin{proof}
When $(\tau_{H})_{x}=0$, it is trivial. 
Thus we assume $(\tau_{H})_{x}\neq 0$.
Since $\lambda \in \tilde{\Sigma}_{H}$, we have
$$
\left( s_{\lambda }, 2\frac{\langle \lambda , H \rangle}{\langle \lambda , \lambda \rangle}\lambda  \right) \in \tilde{W}(\tilde{\Sigma}, \Sigma, W).
$$  
Then, 
\begin{eqnarray*}
\left( s_{\lambda }, 2\frac{\langle \lambda , H \rangle}{\langle \lambda , \lambda \rangle}\lambda  \right) H 
=s_{\lambda }(H)+2\frac{\langle \lambda , H \rangle}{\langle \lambda , \lambda \rangle}\lambda =H.
\end{eqnarray*}
By Proposition \ref{prop2.14new}, we have
$$
dL_{x}^{-1}(\tau_{H})_{x}
=s_{\lambda}(dL_{x}^{-1}(\tau_{H})_{x} )
= dL_{x}^{-1}(\tau_{H})_{x}-2\frac{\langle \lambda, dL_{x}^{-1}(\tau_{H})_{x} \rangle}{\langle \lambda, \lambda \rangle}\lambda.
$$
Therefore, we obtain $\langle \lambda , dL_{x}^{-1}(\tau_{H})_{x}\rangle =0$.
\end{proof}

\begin{proposition}\label{prop:parallelmeancurvatuer}
For $x = \exp H \ (H \in \mathfrak{a})$, we have 
$$
\nabla^{\perp}_{X}\tau_{H}=0 \quad \big(X \in \mathfrak{X}((K_{2}\times K_{1})\cdot x) \big).
$$
\end{proposition}

\begin{proof}
Since the orbit $(K_{2}\times K_{1})\cdot x$ is a homogeneous submanifold in $G$,
it is sufficient to prove that $\nabla^{\perp}_{X}\tau_{H}=0$ at one point $x$ in $G$.

Let $X\in T_{x}((K_{2}\times K_{1})\cdot x)$. 
Then there exists $(X_{2}, X_{1}) \in \mathfrak{k}_{2}\times \mathfrak{k}_{1}$ such that 
$X=(X_{2}, X_{1})_{x}^{\ast}$.
For $k_{2} \in K_{2}$, we have 
\begin{align*}
(0, -dL_{x}^{-1}(\tau_{H})_{x})^{\ast}_{k_{2}x}
&= \left. \frac{d}{dt} k_{2}x\exp (tdL_{x}^{-1}(\tau_{H})_{x}) \right|_{t=0} \\
&= dL_{k_{2}}dL_{x}
\left. \frac{d}{dt} \exp (tdL_{x}^{-1}(\tau_{H})_{x}) \right|_{t=0} \\
&= dL_{k_{2}}(\tau_{H})_{x} \\
&= (\tau_{H})_{k_{2}x}.
\end{align*}
Since $H$ and $dL_{x}^{-1}(\tau_{H})_{x}$ are in $\mathfrak{a}$ from Corollary~\ref{cor:2},
we have, for $k_{1} \in K_{1}$,
\begin{align*}
(dL_{x}^{-1}(\tau_{H})_{x}, 0)^{\ast}_{xk_{1}^{-1}}
&= \left. \frac{d}{dt} \exp (tdL_{x}^{-1}(\tau_{H})_{x}) x k_{1}^{-1} \right|_{t=0} \\
&= \left. \frac{d}{dt} x \exp (tdL_{x}^{-1}(\tau_{H})_{x}) k_{1}^{-1} \right|_{t=0} \\
&= dR_{k_{1}}^{-1}dL_{x}
\left. \frac{d}{dt} \exp (tdL_{x}^{-1}(\tau_{H})_{x}) \right|_{t=0} \\
&= dR_{k_{1}}^{-1}(\tau_{H})_{x} \\
&= (\tau_{H})_{xk_{1}^{-1}}.
\end{align*}
In particular,
$\tau_{H}=(0, -dL_{x}^{-1}(\tau_{H})_{x})^{\ast}$
on the curve $\exp(tX_{2})x$ for $X_{2}\in \mathfrak{k}_{2}$
and 
$\tau_{H}=(dL_{x}^{-1}(\tau_{H})_{x}, 0)^{\ast}$
on the curve $x\exp(tX_{1})$ for $X_{1}\in \mathfrak{k}_{1}$.

Since $H$ and $dL_{x}^{-1}(\tau_{H})_{x}$ are in $\mathfrak{a}$,
$$
\mathrm{Ad}(x)^{-1} dL_{x}^{-1}(\tau_{H})_{x} = dL_{x}^{-1}(\tau_{H})_{x}.
$$
Hence, by Lemma~\ref{Lemma:connection of G}, we have 
\begin{align*}
\left( \nabla^{\perp}_{X} \tau_{H} \right)_{x}
&= \left( \nabla_{(X_{2}, X_{1})^{\ast}} \tau_{H} \right)^{\perp}_{x} \\
&= \left( \nabla_{(X_{2}, 0)^{\ast}} \tau_{H} \right)^{\perp}_{x}
+ \left( \nabla_{(0, X_{1})^{\ast}} \tau_{H} \right)^{\perp}_{x}\\
&= \left( \nabla_{(X_{2}, 0)^{\ast}} (0, -dL_{x}^{-1}(\tau_{H})_{x})^{\ast} \right)^{\perp}_{x}
+ \left( \nabla_{(0, X_{1})^{\ast}} (dL_{x}^{-1}(\tau_{H})_{x}, 0)^{\ast} \right)^{\perp}_{x}\\
&= \left( -\frac{1}{2}dL_{x}[\mathrm{Ad}(x)^{-1}X_{2}, -dL_{x}^{-1}(\tau_{H})_{x}]\right)^{\perp}
+\left( -\frac{1}{2}dL_{x}[-X_{1}, dL_{x}^{-1}(\tau_{H})_{x}]\right)^{\perp}\\
&= \left( \frac{1}{2}dL_{x}[\mathrm{Ad}(x)^{-1}X_{2}+X_{1}, dL_{x}^{-1}(\tau_{H})_{x}]\right)^{\perp}.
\end{align*}
Therefore, in order to prove $\nabla^{\perp}_{X}\tau_{H}=0$, 
it is sufficient to show that 
$$
[(\mathrm{Ad}(x)^{-1}\mathfrak{k}_{2})+\mathfrak{k}_{1}, dL_{x}^{-1}(\tau_{H})_{x}] \subset \mathrm{(Ad}(x)^{-1}\mathfrak{k}_{2})+\mathfrak{k}_{1}.
$$
From (\ref{tangent}), we have
\begin{align*}
\lefteqn{(\mathrm{Ad}(x)^{-1} \mathfrak{k}_{2}) + \mathfrak{k}_{1}} \hspace{5mm}\\
&=\left( 
\mathfrak{k}_{0}
\oplus \sum_{\lambda \in \Sigma^{+}} \mathfrak{k}_{\lambda}
\oplus \sum_{\lambda \in \Sigma^{+}\setminus \Sigma_{H}} \mathfrak{m}_{\lambda}
\oplus V(\mathfrak{k}_{1}\cap \mathfrak{m}_{2})
\oplus \sum_{\alpha \in W^{+}} V_{\alpha}^{\perp}(\mathfrak{k}_{1}\cap \mathfrak{m}_{2})
\right. \\
& \hspace{10mm}\left.
\oplus V(\mathfrak{m}_{1}\cap \mathfrak{k}_{2}) 
\oplus \sum_{\alpha \in W^{+} \setminus W_{H}} V_{\alpha}^{\perp}(\mathfrak{m}_{1}\cap \mathfrak{k}_{2})  
\right).
\end{align*}
Since $dL_{x}^{-1}(\tau_{H})_{x} \in \mathfrak{a}$ and Lemma~\ref{onb}, we have
\begin{align*}
& [\mathfrak{k}_{0}\oplus V(\mathfrak{m}_{1}\cap \mathfrak{k}_{2})
\oplus V(\mathfrak{k}_{1}\cap \mathfrak{m}_{2}), dL_{x}^{-1}(\tau_{H})_{x}] = \{0\}, \\
& [\mathfrak{k}_{\lambda}\oplus \mathfrak{m}_{\lambda}, dL_{x}^{-1}(\tau_{H})_{x}]
\subset \mathfrak{k}_{\lambda} \oplus \mathfrak{m}_{\lambda}, \\
& [V_{\alpha}^{\perp}(\mathfrak{m}_{1}\cap \mathfrak{k}_{2})
\oplus V_{\alpha}^{\perp}(\mathfrak{k}_{1}\cap \mathfrak{m}_{2}), dL_{x}^{-1}(\tau_{H})_{x}]
\subset V_{\alpha}^{\perp}(\mathfrak{m}_{1}\cap \mathfrak{k}_{2})
\oplus V_{\alpha}^{\perp}(\mathfrak{k}_{1}\cap \mathfrak{m}_{2})
\end{align*}
for $\lambda \in \Sigma^{+} \setminus \Sigma_{H}$ and $\alpha  \in W^{+} \setminus W_{H}$.
By Lemma~\ref{onb} and Lemma~\ref{lemma:perpnew}, we also have
\begin{align*}
[\mathfrak{k}_{\lambda }, dL_{x}(\tau_{H})_{x}]=\{ 0\}, \quad
[V_{\alpha}^{\perp}(\mathfrak{k}_{1}\cap \mathfrak{m}_{2}),  dL_{x}(\tau_{H})_{x}] =\{0\}
\end{align*} 
for $\lambda \in \Sigma_{H}^{+}$ and $\alpha  \in W_{H}^{+}$.
Therefore we have the consequence.
\end{proof}

In Theorem~\ref{Theorem3.2},
we described the tension field of the orbit of commutative Hermann actions of $K_2$ on $N_1$.
From this expression, we can verify the existence of minimal orbits.
In the case of isotropy actions of compact symmetric spaces,
the orbit space can be identified with a cell $P_0$ as in (\ref{eq:cell_of_isotropy_action}).
Hirohashi, Tasaki, Song and Takagi proved that,
according to the stratification of orbit types,
there exists a unique minimal orbit in each orbit type (cf. \cite[Theorem 3.1]{HTST}).
For commutative Hermann actions which satisfy one of (A), (B) and (C) in Theorem~\ref{Theorem:I2 AB},
Ikawa obtained the same result (cf. \cite[Theorem 2.24]{I1}).
Moreover, the first author proved Proposition~\ref{Lemma:2nd fand. form of of orbits of proper action},
i.e. $dL_{x}^{-1}(\tau_{H})_{x}=dL_{x}^{-1}(\tau^{1}_{H})_{\pi_{1}(x)}$,
which implies that there exists a unique minimal orbit in each orbit type
for the $(K_2 \times K_1)$-action of $G$.

\section{Characterizations of biharmonic orbits}
\label{sect:Characterizations of biharmonic orbits}

In the previous section, 
we described the second fundamental forms of orbits of the Hermann action of $K_2$ on $N_1$
and the $(K_{2}\times K_{1})$-action on $G$.
In this section, we give a necessary and sufficient condition 
for an orbit to be a biharmonic submanifold.

\subsection{Characterization of biharmonic orbits of commutative Hermann actions}

First, we consider orbits of commutative Hermann actions.
Since all orbits of Hermann actions satisfy $\nabla^{\perp}_{X}\tau^{1}_{H}=0$ (see \cite{IST1}),
we can apply Theorem~\ref{theorem1}.

\begin{theorem}\label{thm: Bih. orbits of commutative Hermann actions}
Let $(G, K_{1}, K_{2})$ be a commutative compact symmetric triad.
For $x = \exp H \ (H \in \mathfrak{a})$,
the orbit $K_{2}\cdot \pi_{1}(x)$ is biharmonic in $N_1$ if and only if 
\begin{align}\label{bih.eq in thm 4.1}
&\sum_{\lambda \in \Sigma^{+} \setminus \Sigma_{H}} m(\lambda)
\langle dL_{x}^{-1}(\tau^{1}_{H})_{\pi_{1}(x)}, \lambda \rangle
\left( 1 - (\cot \langle \lambda, H\rangle )^{2}\right) \lambda \\
+ &\sum_{\alpha \in W^{+} \setminus W_{H}} n(\alpha)
\langle dL_{x}^{-1}(\tau^{1}_{H})_{\pi_{1}(x)}, \alpha \rangle
\left( 1 - (\tan \langle \alpha, H\rangle )^{2}\right) \alpha = 0 \nonumber
\end{align}
holds.
\end{theorem}

\begin{proof}
The curvature tensor $R^{\langle, \rangle}$ of a Riemannian symmetric space $(N_{1}, \langle, \rangle)$
is given by
$$
R^{\langle, \rangle } (dL_{x}(X), dL_{x}(Y)) dL_{x}(Z) = -dL_{x} [[X,Y],Z] \quad (X, Y, Z \in \mathfrak{m}_{1}).
$$
Since $dL_x^{-1}(\tau^{1}_{H})_{\pi_{1}(x)} \in \mathfrak{a}$, we have
\begin{itemize}
\item for $\lambda \in \Sigma^{+} \setminus \Sigma_{H} , 1\leq i\leq m(\lambda )$,
\begin{align*}
R^{\langle, \rangle} ((\tau^{1}_{H})_{\pi_{1}(x)}, dL_{x}(T_{\lambda, i})) dL_{x}(T_{\lambda, i})
&= -dL_x [[dL_x^{-1}(\tau^{1}_{H})_{\pi_{1}(x)}, T_{\lambda, i}], T_{\lambda, i}] \\
&= \langle dL_{x}^{-1}(\tau^{1}_{H})_{\pi_{1}(x)}, \lambda \rangle dL_{x} [S_{\lambda, i}, T_{\lambda, i}] \\
&= \langle dL_{x}^{-1}(\tau^{1}_{H})_{\pi_{1}(x)}, \lambda \rangle dL_{x}(\lambda), 
\end{align*}
\item for $\alpha \in W^{+} \setminus W_{H}, 1 \leq j \leq m(\alpha)$,
$$
R^{\langle, \rangle} ((\tau^{1}_{H})_{\pi_{1}(x)}, dL_{x} (Y_{\alpha, j})) dL_{x} (Y_{\alpha, j})
= \langle dL_{x}^{-1}(\tau^{1}_{H})_{\pi_{1}(x)}, \alpha \rangle dL_{x}(\alpha),
$$
\item for $X \in V(\mathfrak{m}_{1} \cap \mathfrak{k}_{2})$,
$$
R^{\langle, \rangle} ((\tau^{1}_{H})_{\pi_{1}(x)}, dL_{x} (X)) dL_{x} (X) = 0.
$$
\end{itemize} 
On the other hand, 
by Lemma \ref{2nd}, we have 
\begin{itemize}
\item for $\lambda  \in \Sigma^{+} \setminus \Sigma_{H}, 1 \leq i \leq m(\lambda)$, 
$$
A_{(\tau^{1}_{H})_{\pi_{1}(x)}} dL_{x}(T_{\lambda, i})
= -\langle dL_{x}^{-1} (\tau^{1}_{H})_{\pi_{1}(x)}, \lambda \rangle (\cot \langle \lambda, H\rangle) T_{\lambda, i},
$$
\begin{align*}
\lefteqn{B_{H}^1 (A_{(\tau^{1}_{H})_{\pi_{1}(x)}} dL_{x}(T_{\lambda , i}), dL_{x}(T_{\lambda , i}))} \hspace{10mm} \\ 
&= -\langle dL_{x}^{-1}(\tau^{1}_{H})_{\pi_{1}(x)}, \lambda \rangle (\cot \langle \lambda, H\rangle)
B_{H}^1 (dL_{x}(T_{\lambda, i}), dL_{x}(T_{\lambda, i}))\\
&= \langle dL_{x}^{-1}(\tau^{1}_{H})_{\pi_{1}(x)}, \lambda \rangle (\cot \langle \lambda, H \rangle)^{2} dL_{x}(\lambda),
\end{align*} 
\item for $\alpha  \in W^{+} \setminus W_{H} , 1\leq j\leq n(\alpha )$, 
\begin{align*}
\lefteqn{B_{H}^1 (A_{(\tau^{1}_{H})_{\pi_{1}(x)}} dL_{x}(Y_{\alpha, j}), dL_{x}(Y_{\alpha, j}))} \hspace{10mm} \\  
&= -\langle dL_{x}^{-1}(\tau^{1}_{H})_{\pi_{1}(x)}, \alpha \rangle (\tan \langle \alpha, H \rangle)
B_{H}^1 (dL_{x}(Y_{\alpha, j}), dL_{x}(Y_{\alpha, j}))\\
&= \langle dL_{x}^{-1}(\tau^{1}_{H})_{\pi_{1}(x)}, \alpha \rangle (\tan \langle \alpha, H \rangle)^{2} dL_{x}(\alpha),
\end{align*} 
\item for $X\in V(\mathfrak{m}_{1}\cap \mathfrak{k}_{2})$,
$$
B_{H}^1 (A_{(\tau^{1}_{H})_{\pi_{1}(x)}} dL_{x}(X), dL_{x}(X)) = 0.
$$
\end{itemize}
Therefore, by Theorem~\ref{theorem1}, we have the consequence.
\end{proof}

\begin{corollary}\label{cohom1hermann}
Let $(G, K_{1}, K_{2})$ be a commutative compact symmetric triad
which satisfies $\dim \mathfrak{a} =1$, i.e. $\tilde\Sigma \subset \{ \alpha, 2\alpha \}$.
For $x = \exp H \ (H \in \mathfrak{a})$,
suppose that the orbit $K_{2}\cdot \pi_{1}(x)$ is a regular orbit.
Then $K_{2}\cdot \pi_{1}(x)$ is biharmonic in $N_{1}$ if and only if 
\begin{align*}
\langle dL_{x}^{-1}(\tau^{1}_{H})_{\pi_{1}(x)} , \alpha \rangle 
&\left( m(\alpha ) \left\{ 1 - (\cot \langle \alpha , H \rangle )^{2} \right\}
+4m(2\alpha ) \left\{ 1 - (\cot \langle 2\alpha , H \rangle )^{2} \right\} \right. \\
&\left. +n(\alpha ) \left\{ 1 - (\tan \langle \alpha , H \rangle )^{2} \right\}
+4n(2\alpha ) \left\{ 1 - (\tan \langle 2\alpha , H \rangle )^{2} \right\}
\right)
=0
\end{align*}
holds.
Here, for $\lambda \in \mathfrak{a}$,
if $\lambda \notin \Sigma$ (resp. $\lambda \notin W$),
then $m(\lambda)=0$ (resp. $n(\lambda )=0$). 
\end{corollary}

\subsection{Characterization of biharmonic orbits of $(K_2 \times K_1)$-actions}

Next, we consider orbits of the $(K_{2}\times K_{1})$-action on $G$.
By Proposition~\ref{prop:parallelmeancurvatuer}, we can apply Theorem~\ref{theorem1}.

\begin{theorem}\label{Thm:charec:Liegrp}
Let $(G, K_{1}, K_{2})$ be a commutative compact symmetric triad.
For $x = \exp H \ (H \in \mathfrak{a})$,
the orbit $(K_{2}\times K_{1})\cdot x$ is biharmonic in $G$ if and only if 
\begin{align*}
&\sum_{\lambda \in \Sigma^{+} \setminus \Sigma_{H}} m(\lambda ) \langle dL_{x}^{-1}(\tau_{H})_{x}, \lambda \rangle
\left( \frac{3}{2} - (\cot \langle \lambda , H\rangle )^{2}\right) \lambda \\
+&\sum_{\alpha \in W^{+} \setminus W_{H}} n(\alpha) \langle dL_{x}^{-1}(\tau_{H})_{x}, \alpha \rangle
\left( \frac{3}{2} - (\tan \langle \alpha , H\rangle )^{2}\right) \alpha \\ 
+&\sum_{\mu \in  \Sigma_{H}^{+}} m(\mu ) \langle dL_{x}^{-1}(\tau_{H})_{x} , \mu \rangle \mu 
+\sum_{\beta \in  W_{H}^{+}} n(\beta) \langle dL_{x}^{-1}(\tau_{H})_{x} , \beta \rangle  \beta =0
\end{align*}
holds.
\end{theorem}

\begin{proof}
Since $(G, \langle, \rangle)$ is a Riemannian symmetric space,
the curvature tensor $R^{\langle, \rangle}$ of $(G, \langle, \rangle)$
is given by
$$
R^{\langle, \rangle} (dL_{x}(X), dL_{x}(Y)) dL_{x}(Z) = -dL_{x}[[X,Y],Z] \quad (X, Y, Z \in \mathfrak{g}).
$$
Hence, we have
\begin{itemize}
\item for $\lambda \in \Sigma^{+} \setminus \Sigma_{H},\ 1 \leq i \leq m(\lambda)$,
$$
R^{\langle, \rangle} ((\tau_{H})_{x}, dL_{x}(T_{\lambda, i})) dL_{x}(T_{\lambda, i})
= \langle dL_{x}^{-1}(\tau_{H})_{x}, \lambda \rangle dL_{x}(\lambda),
$$
\item for $\lambda \in \Sigma^{+},\ 1 \leq i \leq m(\lambda)$,
$$
R^{\langle, \rangle} ((\tau_{H})_{x}, dL_{x}(S_{\lambda, i})) dL_{x}(S_{\lambda, i})
= \langle dL_{x}^{-1}(\tau_{H})_{x}, \lambda \rangle dL_{x}(\lambda),
$$
\item for $\alpha \in W^{+} \setminus W_{H},\ 1 \leq j \leq n(\alpha)$,
$$
R^{\langle, \rangle} ((\tau_{H})_{x}, dL_{x}(Y_{\alpha, j})) dL_{x}(Y_{\alpha, j})
= \langle dL_{x}^{-1}(\tau_{H})_{x}, \alpha \rangle dL_{x}(\alpha),
$$
\item for $\alpha \in W^{+},\ 1 \leq j \leq n(\alpha)$,
$$
R^{\langle, \rangle} ((\tau_{H})_{x}, dL_{x}(X_{\alpha, j})) dL_{x}(X_{\alpha, j})
= \langle dL_{x}^{-1}(\tau_{H})_{x}, \alpha \rangle dL_{x}(\alpha),
$$
\item for $X \in \mathfrak{k}_{0}\oplus V(\mathfrak{k}_{1}\cap \mathfrak{m}_{2})
\oplus V(\mathfrak{m}_{1}\cap \mathfrak{k}_{2}))$,
$$
R^{\langle, \rangle} ((\tau_{H})_{x}, dL_{x}(X)) dL_{x}(X) = 0.
$$
\end{itemize} 
On the other hand, 
for each $\lambda \in \Sigma^{+}\setminus \Sigma_{H},\ 1\leq i \leq m(\lambda )$
and $X\in \mathrm{Ad}(x)^{-1}(\mathfrak{k}_{2})+\mathfrak{k}_{1}$,
by Theorem~\ref{Lemma:2nd fand. form of of orbits of proper action},
we have 
\begin{align*}
\lefteqn{\langle A_{(\tau_{H})_{x}}dL_{x}(T_{\lambda , i}), dL_{x}(X)\rangle
=\langle B_{H}(dL_{x}(T_{\lambda , i}), dL_{x}(X)), (\tau_{H})_{x}\rangle} \\
&=
\begin{cases}
0 & (X \in \mathfrak{k}_{0}\oplus V(\mathfrak{m}_{1}\cap \mathfrak{k}_{2}) )\\
-\frac{1}{2}\langle dL_{x}[X, T_{\lambda, i}]^{\perp}, (\tau_{H})_{x}\rangle 
& (X \in V(\mathfrak{k}_{1}\cap \mathfrak{m}_{2}) \oplus \sum_{\mu \in \Sigma^{+}}\mathfrak{k}_{\mu}
\oplus \sum_{\alpha \in W^{+}}V^{\perp}_{\alpha}(\mathfrak{k}_{1} \cap \mathfrak{m}_{2})) \\
\cot \langle \mu, H \rangle \langle dL_{x}[T_{\lambda, i}, S_{\mu, j}]^{\perp}, (\tau_{H})_{x} \rangle 
&(X= T_{\mu, j} \ \text{for}\ \mu \in \Sigma^{+}\setminus \Sigma_{H}, 1\leq j\leq m(\lambda) )\\
-\tan \langle \alpha, H \rangle \langle dL_x[T_{\lambda, i}, X_{\alpha, j}]^{\perp}, (\tau_{H})_{x} \rangle 
&(X= Y_{\alpha, j} \ \text{for}\ \alpha \in W^{+}\setminus W_{H}, 1 \leq j \leq n(\alpha) ) 
\end{cases}\\
&=
\begin{cases}
-\frac{1}{2}\langle dL_{x}^{-1}(\tau_{H})_{x}, \lambda \rangle & (X=S_{\lambda, i})\\
-\cot \langle \lambda, H \rangle \langle dL_{x}^{-1}(\tau_{H})_{x}, \lambda \rangle & (X=T_{\lambda, i})\\
0& (\text{if } \langle X, S_{\lambda, i} \rangle = \langle X, T_{\lambda, i} \rangle = 0).
\end{cases}
\end{align*}
Thus, we have 
$$
A_{(\tau_{H})_{x}}dL_{x}(T_{\lambda , i})
= -\frac{1}{2}\langle dL_{x}^{-1}(\tau_{H})_{x}, \lambda \rangle dL_x S_{\lambda, i}
-\cot \langle \lambda, H \rangle \langle dL_{x}^{-1}(\tau_{H})_{x}, \lambda \rangle dL_x T_{\lambda, i}.
$$
Therefore, we obtain 
\begin{align*}
\lefteqn{B_{H}(A_{(\tau_{H})_{x}} dL_{x}(T_{\lambda, i}), dL_{x}(T_{\lambda, i}))} \hspace{10mm} \\
&= -\frac{1}{2} \langle dL_{x}^{-1}(\tau_{H})_{x}, \lambda \rangle
B_{H}(dL_{x}(S_{\lambda, i}), dL_{x}(T_{\lambda, i}))\\ 
& \hspace{10mm} -\langle dL_{x}^{-1}(\tau_{H})_{x}, \lambda \rangle \cot (\langle \lambda, H \rangle)
B_{H}(dL_{x}(T_{\lambda, i}), dL_{x}(T_{\lambda, i}))\\
&= \langle dL_{x}^{-1}(\tau_{H})_{x}, \lambda \rangle
\left( \frac{1}{4} + (\cot \langle \lambda, H \rangle)^{2} \right) dL_{x}(\lambda) 
\end{align*}
for $\lambda \in \Sigma^{+} \setminus \Sigma_{H}, 1 \leq i \leq m(\lambda )$.
Similarly, we have
\begin{itemize}
\item for $\lambda \in \Sigma^{+} \setminus \Sigma_{H}, 1 \leq i \leq m(\lambda )$,
\begin{align*}
\lefteqn{B_{H}(A_{(\tau_{H})_{x}} dL_{x}(S_{\lambda, i}), dL_{x}(S_{\lambda, i}))} \hspace{10mm} \\
&=-\frac{1}{2} \langle dL_{x}^{-1}(\tau_{H})_{x}, \lambda \rangle
B_{H}(dL_{x}(T_{\lambda, i}), dL_{x}(S_{\lambda, i}))\\
&= \frac{1}{4} \langle dL_{x}^{-1}(\tau_{H})_{x}, \lambda \rangle dL_{x}(\lambda),
\end{align*}
\item for $\alpha \in W^{+} \setminus W_{H}, 1 \leq j \leq n(\alpha)$,
\begin{align*}
\lefteqn{B_{H}(A_{(\tau_{H})_{x}} dL_{x}(Y_{\alpha, j}), dL_{x}(Y_{\alpha, j}))} \hspace{10mm} \\
&= -\frac{1}{2} \langle dL_{x}^{-1}(\tau_{H})_{x}, \alpha \rangle
B_{H}(dL_{x}(X_{\alpha, j}), dL_{x}(Y_{\alpha, j})) \\
& \hspace{10mm} -\langle dL_{x}^{-1}(\tau_{H})_{x}, \alpha \rangle \tan (\langle \alpha, H \rangle)
B_{H}(dL_{x}(Y_{\alpha, j}), dL_{x}(Y_{\alpha, j}) )\\
&= \langle dL_{x}^{-1}(\tau_{H})_{x}, \alpha \rangle
\left( \frac{1}{4} + (\tan \langle \alpha, H \rangle )^{2} \right) dL_{x}(\alpha),
\end{align*}
\item for $\alpha \in W^{+} \setminus W_{H} , 1\leq j \leq n(\alpha)$,
\begin{align*}
\lefteqn{B_{H}(A_{(\tau_{H})_{x}} dL_{x}(X_{\alpha, j}), dL_{x}(X_{\alpha, j}))} \hspace{10mm} \\
&= -\frac{1}{2} \langle dL_{x}^{-1}(\tau_{H})_{x}, \alpha \rangle
B_{H}(dL_{x}(Y_{\alpha, j}), dL_{x}(X_{\alpha, j}))\\
&= \frac{1}{4} \langle dL_{x}^{-1}(\tau_{H})_{x}, \alpha \rangle dL_{x}(\alpha),
\end{align*}
\item for $X \in \mathfrak{k}_{0} 
\oplus V(\mathfrak{k}_{1}\cap \mathfrak{m}_{2}) 
\oplus V(\mathfrak{m}_{1} \cap \mathfrak{k}_{2}) 
\oplus \sum_{ \lambda \in \Sigma_{H}^{+}} \mathfrak{k}_{\lambda}
\oplus \sum_{ \alpha \in W_{H}^{+}}  V_{\alpha}^{\perp} (\mathfrak{k}_{1} \cap \mathfrak{m}_{2})$,
$$
B_{H}(A_{(\tau_{H})_{x}} dL_{x}(X), dL_{x}(X)) = 0.
$$

\end{itemize}
Therefore, by Theorem \ref{theorem1}, we have the consequence.
\end{proof}

When $\dim \mathfrak{a} =1$, we have the following corollary.
\begin{corollary}\label{cohom1assoc}
Let $(G, K_{1}, K_{2})$ be a commutative compact symmetric triad
which satisfies $\dim \mathfrak{a} = 1$, i.e. $\tilde\Sigma \subset \{ \alpha, 2\alpha\}$.
For $x = \exp H \ (H \in \mathfrak{a})$, suppose that $(K_{2}\times K_{1})\cdot x$ is a regular orbit.
Then the orbit $(K_{2}\times K_{1})\cdot x$ is biharmonic in $G$ if and only if 
\begin{align*}
\langle dL_{x}^{-1}(\tau_{H})_{x} , \alpha \rangle 
&\left( 
  m(\alpha ) \left\{ \frac{3}{2} - (\cot \langle \alpha , H \rangle )^{2} \right\}
+4m(2\alpha ) \left\{ \frac{3}{2} - (\cot \langle 2\alpha , H \rangle )^{2} \right\} \right. \\
&\left.
 +n(\alpha ) \left\{ \frac{3}{2} - (\tan \langle \alpha , H \rangle )^{2} \right\}
+4n(2\alpha ) \left\{ \frac{3}{2} - (\tan \langle 2\alpha , H \rangle )^{2} \right\}
\right)
=0
\end{align*}
holds.
Here, for $\lambda \in \mathfrak{a}$,
if $\lambda \notin \Sigma$ (resp. $\lambda \notin W$),
then $m(\lambda)=0$ (resp. $n(\lambda )=0$). 
\end{corollary}

\section{Biharmonic homogeneous submanifolds in compact symmetric spaces}
\label{sect:Biharmonic homogeneous submanifolds in compact symmetric spaces}
In the previous section, 
we characterized the biharmonic property of orbits of commutative Hermann actions and the actions of the direct product of two symmetric subgroups on compact Lie groups in terms of symmetric triad with multiplicities.
In this section, 
we study proper biharmonic orbits of commutative Hermann actions by using Theorems \ref{Theorem3.2} and \ref{thm: Bih. orbits of commutative Hermann actions}.
When $\dim \mathfrak{a}=1$, we classified proper biharmonic orbits in \cite{OSU}.
So our interest is in the cases of $\dim \mathfrak{a} \geq 2$.
Here we determine the biharmonic properties of singular orbits of commutative Hermann actions in the cases of $\dim \mathfrak{a} = 2$
Then all cohomogeneity two isotropy actions and commutative Hermann actions satisfying the condition (A), (B), or (C) in Theorem~\ref{Theorem:I2 AB} are classified as follows:
\vspace{6pt}\\
{\bf When $\theta_{1}=\theta_{2}$ (isotropy actions).}
Since $K_{1}=K_{2}$, we show the list of irreducible symmetric pairs of compact type of rank two.
\begin{itemize}
\item Type $\rm{A}_{2}$
	\begin{itemize}
	\item $(\mathrm{SU}(3), \mathrm{SO}(3))$,
	\item $(\mathrm{SU}(3)\times \mathrm{SU}(3), \mathrm{SU}(3))$,
	\item $(\mathrm{SU}(6), \mathrm{Sp}(3))$,
	\item $(E_{6}, F_{4})$,
	\end{itemize}
\item Type $\rm{B}_{2}$
	\begin{itemize}
	\item $(\mathrm{SO}(3)\times \mathrm{SO}(3), \mathrm{SO}(3))$,
	\item $(\mathrm{SO}(4+n), \mathrm{SO}(2)\times \mathrm{SO}(2+n))$,
	\end{itemize}
\item Type $\rm{C}_{2}$
	\begin{itemize}
	\item $(\mathrm{Sp}(2), \mathrm{U}(2))$,
	\item $(\mathrm{Sp}(2)\times \mathrm{Sp}(2), \mathrm{Sp}(2))$,
	\item $(\mathrm{Sp}(4), \mathrm{Sp}(2)\times \mathrm{Sp}(2))$,
	\item $(\mathrm{SU}(4), \mathrm{S}(\mathrm{U}(2)\times\mathrm{U}(2)))$,
	\item $(\mathrm{SO}(8), \mathrm{U}(4))$,
	\end{itemize}
\item Type $\rm{BC}_{2}$
	\begin{itemize}
	\item $(\mathrm{SU}(4+n), \mathrm{S}(\mathrm{U}(2)\times \mathrm{U}(2+n)))$,
	\item $(\mathrm{SO}(10), \mathrm{U}(5))$,
	\item $(\mathrm{Sp}(4+n), \mathrm{Sp}(2)\times \mathrm{Sp}(2+n))$,
	\item $(E_{6}, \mathrm{T}^{1}\cdot \mathrm{Spin}(10))$,
	\end{itemize}
\item Type $\rm{G}_{2}$
	\begin{itemize}
	\item $(G_{2}, \mathrm{SO}(4))$,
	\item $(G_{2}\times G_{2}, G_{2})$,
	\end{itemize}
\end{itemize}
\vspace{6pt}
{\bf When $(\theta_{1}\not\sim \theta_{2})$.}
The following classification is due to Ikawa \cite{I2}.
\begin{itemize}
\item Type $\rm{I\mathchar`-B}_{2}$
	\begin{itemize}
	\item $(\mathrm{SO}(2+s+t), \mathrm{SO}(2+s)\times \mathrm{SO}(t), \mathrm{SO}(2)\times \mathrm{SO}(s+t)) \ (2<t, 1\leq s)$,
	\item $(\mathrm{SO}(6)\times \mathrm{SO}(6), \Delta(\mathrm{SO}(6)\times \mathrm{SO}(6)), K_{2})$\ (condition (C)).
	
	Here $K_{2}=\{ (u_{1}, u_{2})\in  \mathrm{SO}(6)\times \mathrm{SO}(6)\mid  (\sigma(u_{2}), \sigma (u_{1})) =(u_{1}, u_{2})\}$ 
and $\sigma$ is an involutive outer automorphism on $\mathrm{SO}(6)$.
Then $(G_{\sigma})_{0} \cong \mathrm{SO}(3)\times \mathrm{SO}(3)$.
	\end{itemize}
\item Type $\rm{I\mathchar`-C}_{2}$
	\begin{itemize}
	\item $(\mathrm{SO}(8), \mathrm{SO}(4)\times \mathrm{SO}(4), \mathrm{U}(4)) $,
	\item $(\mathrm{SU}(4), \mathrm{SO}(4), \mathrm{S}(\mathrm{U}(2)\times \mathrm{U}(2)) ) $,\
	\item $(\mathrm{SU}(4)\times \mathrm{SU}(4), \Delta(\mathrm{SU}(4)\times \mathrm{SU}(4)), K_{2})$\ 
	(condition (C)).

	Here $K_{2}=\{ (u_{1}, u_{2})\in \mathrm{SU}(4)\times \mathrm{SU}(4)\mid  (\sigma(u_{2}), \sigma (u_{1})) =(u_{1}, u_{2})\}$ 
and $\sigma$ is an involutive outer automorphism on $\mathrm{SU}(4)$.
Then $(G_{\sigma})_{0} \cong \mathrm{SO}(4)$.
	\item $(\mathrm{SU}(4)\times \mathrm{SU}(4), \Delta(\mathrm{SU}(4)\times \mathrm{SU}(4)), K_{2})$\ 
	(condition (C)).

	Here $K_{2}=\{ (u_{1}, u_{2})\in \mathrm{SU}(4)\times \mathrm{SU}(4)\mid  (\sigma(u_{2}), \sigma (u_{1})) =(u_{1}, u_{2})\}$ 
and $\sigma$ is an involutive outer automorphism on $\mathrm{SU}(4)$.
Then $(G_{\sigma})_{0} \cong \mathrm{Sp}(2)$.
	\end{itemize}
\item Type $\rm{I\mathchar`-BC}_{2}\rm{\mathchar`-A}_{1}^{2}$
	\begin{itemize}
	\item $(\mathrm{SU}(2+s+t), \mathrm{S}(\mathrm{U}(2+s)\times \mathrm{U}(t)), \mathrm{S}(\mathrm{U}(2)\times \mathrm{U}(s+t) ) ) \ (2<t, 1\leq s)$,\
	\item $(\mathrm{Sp}(2+s+t), \mathrm{Sp}(2+s)\times \mathrm{Sp}(t), \mathrm{Sp}(2)\times \mathrm{Sp}(s+t)  ) \ (2<t, 1\leq s)$,\
	\item $(\mathrm{SO}(12), \mathrm{U}(6), \mathrm{U}(6)')$.\
	Here, we define 
$\mathrm{U}(6)'=\{g\in \mathrm{SO}(12) \mid JgJ^{-1}=g \}$, 
where
$$
J=\left[
\begin{array}{cc|cc}
& &I_{5} &\\
& & &-1\\ \hline
-I_{5}& & &\\
& 1& &
\end{array}
\right]
$$ 
and $I_l$ denotes the identity matrix of $l \times l$.
	\end{itemize}
\item Type $\rm{I\mathchar`-BC}_{2}\rm{\mathchar`-B}_{2}$
	\begin{itemize}
	\item $(\mathrm{SO}(4+2s), \mathrm{SO}(4)\times \mathrm{SO}(2s), \mathrm{U}(2+s))\ (2<s), $
	\item $(E_{6}, \mathrm{SU}(6)\cdot \mathrm{SU}(2), \mathrm{SO}(10)\cdot \mathrm{U}(1))$,\
	\item $(E_{7}, \mathrm{SO}(12)\cdot \mathrm{SU}(2), E_{6}\cdot \mathrm{U}(1))$,
	\end{itemize}
\item Type $\rm{II\mathchar`-BC}_{2}$
	\begin{itemize}
	\item $(\mathrm{SU}(2+s), \mathrm{SO}(2+s), \mathrm{S}(\mathrm{U}(2)\times \mathrm{U}(s)))\ (2<s)$,\ 
	\item $(\mathrm{SO}(10), \mathrm{SO}(5)\times \mathrm{SO}(5), \mathrm{U}(5))$,\
	\item $(E_{6}, \mathrm{Sp}(4), \mathrm{SO}(10)\cdot \mathrm{U}(1))$,
	\end{itemize}
\item Type $\rm{III\mathchar`-A}_{2}$
	\begin{itemize}
	\item $(\mathrm{SU}(6), \mathrm{Sp}(3), \mathrm{SO}(6))$,\ 
	\item $(E_{6}, \mathrm{Sp}(4),F_{4})$,\
	\item $(U\times U, \Delta(U\times U), \overline{K}\times \overline{K})$\ (condition (B)).
	
Here $(U, \overline{K})$ is a compact symmetric pair of type $\rm{A}_{2}$.
	\end{itemize}
\item Type $\rm{III\mathchar`-B}_{2}$
	\begin{itemize}
	\item $(U\times U, \Delta(U\times U), \overline{K}\times \overline{K})$\ (condition (B)).

Here $(U, \overline{K})$ is a compact symmetric pair of type $\rm{B}_{2}$.
	\end{itemize}
\item Type $\rm{III\mathchar`-C}_{2}$
	\begin{itemize}
	\item $(\mathrm{SU}(8), \mathrm{S}(\mathrm{U}(4)\times \mathrm{U}(4)), \mathrm{Sp}(4))$,\
	\item $(\mathrm{Sp}(4), \mathrm{U}(4), \mathrm{Sp}(2)\times \mathrm{Sp}(2))$,\
	\item $(U\times U, \Delta(U\times U), \overline{K}\times \overline{K})$ (condition (B)).

Here $(U, \overline{K})$ is a compact symmetric pair of type $\rm{C}_{2}$.
	\end{itemize}
\item Type $\rm{III\mathchar`-BC}_{2}$
	\begin{itemize}
	\item $(\mathrm{SU}(4+2s), \mathrm{S}(\mathrm{U}(4)\times \mathrm{U}(2s)), \mathrm{Sp}(2+s)) \ (2<s)$, 
	\item $(\mathrm{SU}(10), \mathrm{S}(\mathrm{U}(5)\times \mathrm{U}(5)) , \mathrm{Sp}(5))$,\
	\item $(U\times U, \Delta(U\times U), \overline{K}\times \overline{K})$ (condition (B)). 

Here $(U, \overline{K})$ is a compact symmetric pair of type $\rm{BC}_{2}$.
	\end{itemize}
\item Type $\rm{III\mathchar`-G}_{2}$
	\begin{itemize}
	\item $(U\times U, \Delta(U\times U), \overline{K}\times \overline{K})$ (condition (B)).

Here $(U, \overline{K})$ is a compact symmetric pair of type $\rm{G}_{2}$.
	\end{itemize}
\end{itemize}

In the following, 
we consider the biharmonic properties of singular orbits of Hermann actions for each compact symmetric triad in the above list.
For $H \in \mathfrak{a}$, we set $x= \exp(H)$ and consider the orbit $K_{2}\cdot \pi_{1}(x)$ of the $K_{2}$-action on $N_{1}$ through $\pi_{1}(x)$.
For simplicity, we denote the tension field $dL_{x}^{-1}(\tau_{H}^{1})_{\pi_{1}(x)}$ by $\tau_{H}$.

\subsection*{Cases of $\theta_{1}=\theta_{2}$}
First, we examine isotropy actions of compact symmetric spaces.
When 
$\theta_{1}\sim \theta_{2}$, Hermann actions are orbit equivalent to isotropy actions of compact symmetric spaces. 
We set a basis $\{H_{\alpha}\}_{\alpha \in \Pi}$ of $\mathfrak{a}$ as follows;
$$
\langle H_{\alpha}, \beta \rangle=0 \ (\alpha \neq \beta ,\ \alpha , \beta \in \Pi),\quad 
\langle H_{\alpha}, \delta \rangle=\pi , 
$$
where $\delta$ is the highest root of $\Sigma$.
Then we have
$$
P_{0}=\left\{ \left. \sum_{\alpha \in \Pi} t_{\alpha} H_{\alpha} \ \right| \ t_{\alpha} >0 \ (\alpha \in \Pi ),\ \sum_{\alpha \in \Pi}t_{\alpha} <1 \right\} .
$$
From (\ref{eq:cell_of_isotropy_action}) and (\ref{eq:cell_decomposition_of_isotropy_action}), 
the orbit space of an isotropy action is described as 
$\overline{P_{0}}=\bigcup_{\Delta \subset \Pi \cup \{\delta \}} P_{0}^{\Delta}.$
Since $\dim \mathfrak{a}=2$, $\Pi=\{ \alpha_{1}, \alpha_{2}\}$ and $P_{0}$ is a triangle region in $\mathfrak{a}$.
We apply Theorem \ref{thm: Bih. orbits of commutative Hermann actions} 
to the following three cases; 
\begin{enumerate}
\item $H \in P_{0}^{\{ \alpha_{1}, \delta\} }=\{tH_{\alpha_{1}} \mid 0<t<1\}$,
\item $H \in P_{0}^{\{ \alpha_{2}, \delta\} }=\{tH_{\alpha_{2}} \mid 0<t<1\}$,
\item $H \in P_{0}^{\{ \alpha_{1}, \alpha_{2}\} }=\{tH_{\alpha_{1}}+(1-t)H_{\alpha_{2}} \mid 0<t<1\}$.
\end{enumerate}
These three cases correspond to three edges of $\overline{P_{0}}$.
In these cases, we can solve the equation (\ref{bih.eq in thm 4.1}) in Theorem \ref{thm: Bih. orbits of commutative Hermann actions} 
concretely in most cases. 
In the following, we compute the equation (\ref{bih.eq in thm 4.1}) in Theorem \ref{thm: Bih. orbits of commutative Hermann actions} 
for each root type.
\subsection{Type $\rm{A}_{2}$}
We set $$\mathfrak{a}=\{\xi_{1} e_{1}+ \xi_{2} e_{2}+ \xi_{3} e_{3} \mid \xi_{1}+ \xi_{3}+\xi_{3}=0 \}.$$
Then, we have 
\begin{align*}
&\Sigma^{+}=\{\alpha_{1}=e_{1}-e_{2} , \alpha_{2}=e_{2}-e_{3}, \alpha_{1}+\alpha_{2}\},\ W^{+}=\emptyset, \\
&m=m(\alpha) \qquad (\alpha \in \Sigma).
\end{align*}

(1) When $H \in P_{0}^{\{ \alpha_{1}, \delta\} }=\{tH_{\alpha_{1}} \mid 0<t<1\}$, we have $\Sigma_{H}^{+}=\{\alpha_{2} \}$. 
Hence we have 
$$
\tau_{H}=m\cot \langle \alpha_{1}, H\rangle \alpha_{1} + m\cot \langle \alpha_{1} + \alpha_{2} , H\rangle (\alpha_{1}+\alpha_{2})= m \cot\langle \alpha_{1}, H \rangle (2\alpha_{1}+\alpha_{2}).  
$$
Thus the orbit $K_{2}\cdot \pi_{1}(x)$ is harmonic if and only if 
$\langle \alpha_{1}, H\rangle= \pi /2$.
By Theorem \ref{thm: Bih. orbits of commutative Hermann actions}, the orbit $K_{2}\cdot \pi_{1}(x)$ is biharmonic if and only if
\begin{align*}
0=&m \langle \tau_{H}, \alpha_{1}             \rangle (1- (\cot \langle \alpha_{1} , H            \rangle )^{2}) \alpha_{1} \\
&+m \langle \tau_{H}, \alpha_{1}+\alpha_{2} \rangle (1- (\cot \langle \alpha_{1}+\alpha_{2} , H \rangle )^{2})(\alpha_{1}+\alpha_{2})\\
=&m \langle \tau_{H}, \alpha_{1}             \rangle (1- (\cot \langle \alpha_{1} , H            \rangle )^{2})(2\alpha_{1}+\alpha_{2}).
\end{align*}
Thus we have 
$\tau_{H}=0$ or $\langle \alpha_{1}, H \rangle =(1/4)\pi,\ (3/4)\pi.$
Therefore, the orbit $K_{2}\cdot \pi_{1}(x)$ is proper biharmonic if and only if $\langle \alpha_{1}, H \rangle =(1/4)\pi,\ (3/4)\pi.$
{\em In this case, there exist exactly
two proper biharmonic orbits.} 
By the same argument, we have the followings:\vspace{6pt}

(2) The orbit $K_{2}\cdot \pi_{1}(x)$ is proper biharmonic if and only if $\langle \alpha_{2}, H \rangle =(1/4)\pi,\ (3/4)\pi $ for $H=t H_{\alpha_{2}} \ (0<t<1).$\vspace{6pt}

(3) The orbit $K_{2}\cdot \pi_{1}(x)$ is proper biharmonic if and only if $\langle \alpha_{1}, H \rangle =(1/4)\pi,\ (3/4)\pi $ for $H=tH_{\alpha_{1}}+ (1-t)H_{\alpha_{2}} \ (0<t<1).$

\subsection{Type $\rm{B}_{2}$ and $\rm{C}_{2}$}
We set
\begin{align*}
\Sigma^{+}&=\{\alpha_{1}=e_{1}-e_{2} , \alpha_{2}=e_{2}, \alpha_{1}+\alpha_{2}, \alpha_{1}+2\alpha_{2} \},\
W^{+}=\emptyset, \\
\delta &=\alpha_{1}+2\alpha_{2}=e_{1}+e_{2},
\end{align*}
and 
$$
m_{1}=m(e_{1}),\ m_{2}=m(e_{1}-e_{2}).
$$

\vspace{12pt}
(1)\ When $H \in P_{0}^{\{ \alpha_{1}, \delta\} }=\{tH_{\alpha_{1}} \mid 0<t<1\}$, we have 
$\Sigma^{+}_{H}=\{e_{2}\}$.
By Theorem \ref{Theorem3.2}, 
we have 
\begin{align*}
\tau_{H}=
&-m_{2}\cot \langle \alpha_{1}, H \rangle \alpha_{1}
-m_{1}\cot \langle \alpha_{1}+\alpha_{2}, H \rangle (\alpha_{1}+\alpha_{2})\\
&-m_{2}\cot \langle \alpha_{1}+2\alpha_{2}, H \rangle (\alpha_{1}+2\alpha_{2})\\
=&-(2m_{2} +m_{1})\cot \langle \alpha_{1}, H \rangle (\alpha_{1}+\alpha_{2}).
\end{align*}
Hence, $\tau_{H}=0$ if and only if $\langle \alpha_{1}, H \rangle =\pi/2$.
By Theorem \ref{thm: Bih. orbits of commutative Hermann actions}, $K_{2}\cdot \pi_{1}(x)$ is biharmonic if and only if 
\begin{align*}
0=
&m_{2}\langle \tau_{H}, \alpha_{1} \rangle (1-(\cot \langle \alpha_{1}, H \rangle)^{2}) \alpha_{1}\\
&+m_{1}\langle \tau_{H}, \alpha_{1}+\alpha_{2} \rangle (1-(\cot \langle \alpha_{1}, H \rangle)^{2}) (\alpha_{1}+\alpha_{2})\\
&+m_{2}\langle \tau_{H}, \alpha_{1}+2\alpha_{2} \rangle (1-(\cot \langle \alpha_{1}, H \rangle)^{2}) (\alpha_{1}+2\alpha_{2})\\
=&\langle \tau_{H}, \alpha_{1} \rangle (2m_{2}+m_{1}) (1-(\cot \langle \alpha_{1}, H \rangle)^{2}) (\alpha_{1}+\alpha_{2}).
\end{align*}
Therefore, the orbit $K_{2}\cdot \pi_{1}(x)$ is biharmonic if and only if
$\tau_{H}=0$ or $\langle \alpha_{1}, H\rangle= \pi/4,\ (3/4)\pi$.
In particular, $K_{2}\cdot \pi_{1}(x)$ is proper biharmonic if and only if
$\langle \alpha_{1}, H\rangle= \pi/4, (3/4)\pi$.
{\em In this case, there exist exactly
two proper biharmonic orbits.}

\vspace{12pt}
(2)\ When $H \in P_{0}^{\{ \alpha_{2}, \delta\} }=\{tH_{\alpha_{2}} \mid 0<t<1\}$, we have 
$\Sigma^{+}_{H}=\{e_{1}-e_{2}\}$.
By Theorem \ref{Theorem3.2}, 
we have 
\begin{align*}
\tau_{H}=
&-m_{1}\cot \langle \alpha_{2}, H \rangle \alpha_{2}
-m_{1}\cot \langle \alpha_{1}+\alpha_{2}, H \rangle (\alpha_{1}+\alpha_{2})\\
&-m_{2}\cot \langle \alpha_{1}+2\alpha_{2}, H \rangle (\alpha_{1}+2\alpha_{2})\\
=&
-\frac{1}{2}\{ (2m_{1}+m_{2})\cot \langle \alpha_{2}, H \rangle -m_{2}\tan \langle \alpha_{2}, H \rangle\} (\alpha_{1}+2\alpha_{2}). 
\end{align*}
Hence, $\tau_{H}=0$ if and only if 
$$
(\cot\langle \alpha_{2}, H \rangle)^{2} =\frac{m_{2}}{2m_{1}+m_{2}}.
$$
By Theorem \ref{thm: Bih. orbits of commutative Hermann actions}, the orbit $K_{2}\cdot \pi_{1}(x)$ is biharmonic if and only if 
\begin{align*}
0=
&m_{1}\langle \tau_{H}, \alpha_{2} \rangle (1-(\cot \langle \alpha_{2}, H \rangle)^{2}) \alpha_{2}
+m_{1}\langle \tau_{H}, \alpha_{1}+\alpha_{2} \rangle (1-(\cot \langle \alpha_{2}, H \rangle)^{2}) (\alpha_{1}+\alpha_{2})\\
&+m_{2}\langle \tau_{H}, \alpha_{1}+2\alpha_{2} \rangle (1-(\cot \langle 2\alpha_{2}, H \rangle)^{2}) (\alpha_{1}+2\alpha_{2})\\
=
&\frac{1}{2}\langle \tau_{H}, \alpha_{2} \rangle
\{ (2m_{1}+m_{2})(1-(\cot \langle \alpha_{2}, H \rangle)^{2}) \\
&+m_{2}(1-(\tan \langle \alpha_{2}, H \rangle )^{2}) +4m_{2}\}
(\alpha_{1}+2\alpha_{2}).
\end{align*}
Therefore, the orbit $K_{2}\cdot \pi_{1}(x)$ is biharmonic if and only if
$\tau_{H}=0$ or 
$$
(2m_{1}+m_{2})(1-(\cot \langle \alpha_{2}, H \rangle)^{2}) +m_{2}(1-(\tan \langle \alpha_{2}, H \rangle )^{2}) +4m_{2}=0
$$
holds. 
The last equation is equivalent to
$$
\big( (2m_{1}+m_{2}) (\cot \langle \alpha_{2}, H \rangle)^{2}-m_{2} \big) ((\cot \langle \alpha_{2}, H \rangle)^{2}-1 )=4m_{2}(\cot \langle \alpha_{2}, H \rangle)^{2}.
$$
Since $m_{2}>0$, the solutions of the equation are not harmonic.
Hence the orbit $K_{2}\cdot \pi_{1}(x)$ is proper biharmonic if and only if
$$
(\cot\langle \alpha_{2}, H\rangle )^{2}=\frac{m_{1}+3m_{2}\pm \sqrt{m_{1}^{2} + 4m_{1}m_{2}+8m_{2}^{2}}}{2m_{1}+m_{2}}.
$$
{\em In this case, there exist exactly
two proper biharmonic orbits.}

\vspace{12pt}
(3) When $H \in P_{0}^{\{ \alpha_{1}, \alpha_{2}\} }=\{ tH_{\alpha_{1}}+(1-t)H_{\alpha_{2}} \mid 0<t<1\}$, 
we have 
$\Sigma^{+}_{H}=\{e_{1}+e_{2}\}$ and 
$\langle \alpha_{2}, H\rangle=(\pi/2)-\langle \alpha_{1}, H\rangle /2 $.
By Theorem \ref{Theorem3.2}, 
we have 
\begin{align*}
\tau_{H}=
&-m_{2}\cot \langle \alpha_{1}, H \rangle \alpha_{1}
-m_{1}\cot \langle \alpha_{2}, H \rangle \alpha_{2}
-m_{1}\cot \langle \alpha_{1}+\alpha_{2}, H \rangle (\alpha_{1}+\alpha_{2})\\
=
&\frac{1}{2}\left\{ -m_{2}\cot \left( \frac{\langle \alpha_{1}, H \rangle}{2}\right) 
+(2m_{1}+m_{2})\tan \left( \frac{\langle \alpha_{1}, H \rangle}{2}\right) \right\}\alpha_{1}.
\end{align*}
Hence, $\tau_{H}=0$ if and only if 
$$
\left( \cot \frac{\langle \alpha_{1}, H \rangle}{2}\right)^{2} =\frac{2m_{1}+m_{2}}{m_{2}}.
$$
By Theorem \ref{thm: Bih. orbits of commutative Hermann actions}, the orbit $K_{2}\cdot \pi_{1}(x)$ is biharmonic if and only if 
\begin{align*}
0=
&m_{2}\langle \tau_{H}, \alpha_{1} \rangle (1-(\cot \langle \alpha_{1}, H \rangle)^{2}) \alpha_{1}
+m_{1}\langle \tau_{H},\alpha_{2} \rangle (1-(\cot \langle \alpha_{2}, H \rangle)^{2}) \alpha_{2}\\
&+m_{1}\langle \tau_{H}, \alpha_{1}+\alpha_{2} \rangle (1-(\cot \langle \alpha_{1}+\alpha_{2}, H \rangle)^{2}) (\alpha_{1}+\alpha_{2})\\
=&
\frac{1}{4}\langle \tau_{H}, \alpha_{1} \rangle
\left\{
4m_{2}+m_{2}\left( 1-\left( \cot \frac{\langle \alpha_{1}, H \rangle }{2}\right)^{2}\right) \right.\\
&\left. \qquad + (2m_{1}+m_{2})\left( 1-\left( \tan \frac{\langle \alpha_{1}, H \rangle }{2}\right)^{2}\right)
\right\}
\alpha_{1}.
\end{align*}
Therefore, the orbit $K_{2}\cdot \pi_{1}(x)$ is biharmonic if and only if
$\tau_{H}=0$ or 
$$
4m_{2}+m_{2}\left( 1-\left( \cot \frac{\langle \alpha_{1}, H \rangle }{2}\right)^{2}\right) 
+ (2m_{1}+m_{2})\left( 1-\left( \tan \frac{\langle \alpha_{1}, H \rangle }{2}\right)^{2}\right)=0
$$
holds. 
The last equation is equivalent to
$$
\left( m_{2}\left( \cot \frac{\langle \alpha_{1}, H \rangle }{2}\right)^{2} - (2m_{1}+m_{2})\right)
\left( \left( \cot \frac{\langle \alpha_{1}, H \rangle }{2}\right)^{2}-1\right)
=4m_{2}\left( \cot \frac{\langle \alpha_{1}, H \rangle }{2}\right)^{2}.
$$
Since $m_{2}>0$, the solutions of the equation are not harmonic.
Hence the orbit $K_{2}\cdot \pi_{1}(x)$ is proper biharmonic if and only if
$$
\left( \cot \frac{\langle \alpha_{1}, H \rangle }{2}\right)^{2}=
\frac{m_{1}+3m_{2} \pm \sqrt{m_{1}^{2}+4m_{1}m_{2}+8m_{2}^{2} } }{m_{2}}
$$
holds.
{\em In this case, there exist exactly
two proper biharmonic orbits.} 

\subsection{Type $\rm{BC}_{2}$}
We set
\begin{align*}
\Sigma^{+}=&\{\alpha_{1}=e_{1}-e_{2} , \alpha_{2}=e_{2}, \alpha_{1}+\alpha_{2}, \alpha_{1}+2\alpha_{2}, 2\alpha_{2}, 2\alpha_{1}+2\alpha_{2} \},\\
W^{+}=&\emptyset, \
\delta =2\alpha_{1}+2\alpha_{2},
\end{align*}
and 
$$
m_{1}=m(e_{1}),\ m_{2}=m(e_{1}-e_{2}),\ m_{3}=m(2e_{1}).
$$

\vspace{12pt}
(1)\ When $H \in P_{0}^{\{ \alpha_{1}, \delta\} }=\{tH_{\alpha_{1}} \mid 0<t<1\}$, we have 
$\Sigma^{+}_{H}=\{e_{2}, 2e_{2}\}$.
By Theorem \ref{Theorem3.2}, 
we have 
\begin{align*}
\tau_{H}=
\{ -(m_{1}+2m_{2} +m_{3})\cot \langle \alpha_{1}, H \rangle + m_{3} \tan \langle \alpha_{1}, H \rangle \}(\alpha_{1}+\alpha_{2}).
\end{align*}
Hence, $\tau_{H}=0$ if and only if 
$$
(\cot \langle \alpha_{1}, H \rangle)^{2} =\frac{m_{3}}{m_{1}+2m_{2}+m_{3}}
$$
holds.
By Theorem \ref{thm: Bih. orbits of commutative Hermann actions}, $K_{2}\cdot \pi_{1}(x)$ is biharmonic if and only if 
\begin{align*}
0=
&\langle \tau_{H}, \alpha_{1} \rangle 
\{
(2m_{2}+m_{1}+m_{3}) (1-(\cot \langle \alpha_{1}, H \rangle)^{2}) \\
& +m_{3}(1-(\tan \langle \alpha_{1}, H \rangle)^{2})
+4m_{3} 
\}
(\alpha_{1}+\alpha_{2}).
\end{align*}
Therefore, the orbit $K_{2}\cdot \pi_{1}(x)$ is biharmonic if and only if
$\tau_{H}=0$ or 
$$
(2m_{2}+m_{1}+m_{3}) (1-(\cot \langle \alpha_{1}, H \rangle)^{2}) 
+m_{3}(1-(\tan \langle \alpha_{1}, H \rangle)^{2})
+4m_{3}
=0
$$
holds.
The last equation is equivalent to
$$
\big( (2m_{2}+m_{1}+m_{3})(\cot \langle \alpha_{1}, H \rangle)^{2}-m_{3}\big) ((\cot \langle \alpha_{1}, H \rangle)^{2}-1)=4m_{3}(\cot \langle \alpha_{1}, H \rangle)^{2}.
$$
Since $m_{3}>0$, the solutions of the equation are not harmonic.
Hence the orbit $K_{2}\cdot \pi_{1}(x)$ is proper biharmonic if and only if
\begin{align*}
&(\cot\langle \alpha_{1}, H\rangle )^{2}\\
=&\frac{m_{1}+2m_{2}+6m_{3}\pm \sqrt{(m_{1}+2m_{2}+6m_{3})-4(m_{1}+2m_{2}+m_{3})m_{3}  }}{m_{1}+2m_{2}+m_{3}}.
\end{align*}
{\em In this case, there exist exactly
two proper biharmonic orbits.}

\vspace{12pt}
(2) When $H \in P_{0}^{\{ \alpha_{2}, \delta\} }=\{tH_{\alpha_{2}} \mid 0<t<1\}$, we have 
$\Sigma^{+}_{H}=\{e_{1}-e_{2}\}$.
By Theorem \ref{Theorem3.2}, 
we have 
\begin{align*}
\tau_{H}=
-\frac{1}{2}\{ (2m_{1}+m_{2}+2m_{3})\cot \langle \alpha_{2}, H \rangle -(m_{2}+2m_{3})\tan \langle \alpha_{2}, H \rangle\} (\alpha_{1}+2\alpha_{2}) .
\end{align*}
Hence, $\tau_{H}=0$ if and only if 
$$
(\cot\langle \alpha_{2}, H \rangle)^{2} =\frac{m_{2}+2m_{3}}{2m_{1}+m_{2}+2m_{3}}.
$$
By Theorem \ref{thm: Bih. orbits of commutative Hermann actions}, the orbit $K_{2}\cdot \pi_{1}(x)$ is biharmonic if and only if 
\begin{align*}
0=
&\frac{1}{2} \langle \tau_{H}, \alpha_{2} \rangle
\{ (2m_{1}+m_{2}+2m_{3})(1-(\cot \langle \alpha_{2}, H \rangle)^{2}) \\
&+(m_{2}+2m_{3})(1-(\tan \langle \alpha_{2}, H \rangle )^{2}) +4(m_{2}+2m_{3})\}
(\alpha_{1}+2\alpha_{2}).
\end{align*}
Therefore, the orbit $K_{2}\cdot \pi_{1}(x)$ is biharmonic if and only if
$\tau_{H}=0$ or 
\begin{align*}
&(2m_{1}+m_{2}+2m_{3})(1-(\cot \langle \alpha_{2}, H \rangle)^{2}) \\
+&(m_{2}+2m_{3})(1-(\tan \langle \alpha_{2}, H \rangle )^{2}) +4(m_{2}+2m_{3})=0
\end{align*}
holds. 
The last equation is equivalent to
\begin{align*}
&\big( (2m_{1}+m_{2}+2m_{3}) (\cot \langle \alpha_{2}, H \rangle)^{2}-(m_{2}+2m_{3}) \big) ((\cot \langle \alpha_{2}, H \rangle)^{2}-1 )\\
=&4(m_{2}+2m_{3})(\cot \langle \alpha_{2}, H \rangle)^{2}.
\end{align*}
Since $m_{2}+2m_{3}>0$, the solutions of the equation are not harmonic.
Hence the orbit $K_{2}\cdot \pi_{1}(x)$ is proper biharmonic if and only if
$$
(\cot\langle \alpha_{2}, H\rangle )^{2}=\frac{m_{1}+3(m_{2}+2m_{3})\pm 
\sqrt{ m_{1}^{2} + 4m_{1}(m_{2}+2m_{3})+8(m_{2}+2m_{3})^{2} } }{2m_{1}+m_{2}+2m_{3}}.
$$
{\em In this case, there exist exactly
two proper biharmonic orbits.}

\vspace{12pt}
(3) When $H \in P_{0}^{\{ \alpha_{1}, \alpha_{2}\} }=\{ tH_{\alpha_{1}}+(1-t)H_{\alpha_{2}} \mid 0<t<1\}$, 
 we have 
$\Sigma^{+}_{H}=\{2e_{1}\}$ and $\langle \alpha_{2}, H \rangle =(\pi/2)-\langle \alpha_{1}, H\rangle $.
By Theorem \ref{Theorem3.2}, 
we have 
\begin{align*}
\tau_{H}=
-(m_{1}+m_{3})\tan \langle \alpha_{1}, H \rangle \alpha_{2}
+(2m_{2}+m_{3})\cot \langle \alpha_{1}, H \rangle \alpha_{2}.
\end{align*}
Hence, $\tau_{H}=0$ if and only if 
$$
(\cot\langle \alpha_{1}, H \rangle)^{2} =\frac{m_{1}+m_{3}}{2m_{2}+m_{3}}.
$$
By Theorem \ref{thm: Bih. orbits of commutative Hermann actions}, the orbit $K_{2}\cdot \pi_{1}(x)$ is biharmonic if and only if 
\begin{align*}
0=
&-\langle \tau_{H}, \alpha_{1} \rangle
\{
 (m_{3}+2m_{2})(1-(\cot \langle \alpha_{1}, H \rangle)^{2})\\
& +(m_{1}+m_{3})(1-(\tan \langle \alpha_{1}, H \rangle)^{2})
+4m_{3}
\} \alpha_{2}.
\end{align*}
Therefore, the orbit $K_{2}\cdot \pi_{1}(x)$ is biharmonic if and only if
$\tau_{H}=0$ or 
$$
(m_{3}+2m_{2})(1-(\cot \langle \alpha_{1}, H \rangle)^{2})
+(m_{1}+m_{3})(1-(\tan \langle \alpha_{1}, H \rangle)^{2})
+4m_{3}=0
$$
holds. 
The last equation is equivalent to
$$
\big( (m_{3}+2m_{2}) (\cot \langle \alpha_{1}, H \rangle)^{2} -(m_{1}+m_{3})\big) ( (\cot \langle \alpha_{1}, H \rangle)^{2}-1)=4m_{3}(\cot \langle \alpha_{1}, H \rangle)^{2}.
$$
Since $m_{3}>0$, the solutions of the equation are not harmonic.
Hence the orbit $K_{2}\cdot \pi_{1}(x)$ is proper biharmonic if and only if
$$
(\cot \langle \alpha_{1}, H \rangle)^{2}=
\frac{m_{1}+2m_{2}+6m_{3} \pm \sqrt{ (m_{1}-2m_{2})^{2}+8m_{3}(m_{1}+2m_{2}+4m_{3})} }{2(2m_{2}+m_{3})}
$$
holds.
{\em In this case, there exist exactly
two proper biharmonic orbits.}

\subsection{Type $\rm{G}_{2}$}
We set
\begin{align*}
\Sigma^{+}&=\{\alpha_{1}, \alpha_{2}, \alpha_{1}+\alpha_{2}, 2\alpha_{1}+\alpha_{2}, 3\alpha_{1}+\alpha_{2}, 3\alpha_{1}+2\alpha_{2} \},\
W^{+}=\emptyset, \\
&\langle \alpha_{1}, \alpha_{1} \rangle=1, \ \langle \alpha_{1}, \alpha_{2} \rangle =-\frac{3}{2}, \ \langle \alpha_{2}, \alpha_{2} \rangle=3,\\
\delta &=3\alpha_{1}+2\alpha_{2},
\end{align*}
and 
$$
m=m(\alpha_{1} )=m(\alpha_{2}).
$$

\vspace{12pt}
(1)\ When $H \in P_{0}^{\{ \alpha_{1}, \delta\} }=\{tH_{\alpha_{1}} \mid 0<t<1\}$, 
we have $\Sigma^{+}_{H}=\{\alpha_{2}\}, W^{+}_{H}=\emptyset$.
By Theorem \ref{Theorem3.2}, 
we have 
\begin{align*}
\tau_{H}=
&-m\left\{
\cot\langle \alpha_{1}, H\rangle 
+\cot\langle 2\alpha_{1}, H\rangle
+3 \frac{\cot \langle \alpha_{1}, H\rangle \cot \langle 2\alpha_{1}, H\rangle -1}{\cot \langle \alpha_{1}, H\rangle +\cot \langle 2\alpha_{1}, H\rangle}
\right\}(2\alpha_{1}+\alpha_{2}).
\end{align*}
Thus, $\tau_{H}=0$ if and only if
\begin{align*}
0=&\left\{
\cot\langle \alpha_{1}, H\rangle 
+\cot\langle 2\alpha_{1}, H\rangle
+3 \frac{\cot \langle \alpha_{1}, H\rangle \cot \langle 2\alpha_{1}, H\rangle -1}{\cot \langle \alpha_{1}, H\rangle +\cot \langle 2\alpha_{1}, H\rangle}
\right\}\\
=&\frac{1}{4}
\big(
 15(\cot \langle \alpha_{1}, H\rangle)^{2}-24+ (\tan \langle \alpha_{1}, H\rangle)^{2}
\big).
\end{align*}
Since 
$0<\langle \alpha_{1}, H \rangle <(\pi/3)$, $\tau_{H}=0$ if and only if 
$$
(\cot \langle \alpha_{1} , H \rangle)^{2}= \frac{12+\sqrt{129}}{15}. 
$$
By Theorem \ref{thm: Bih. orbits of commutative Hermann actions}, the orbit $K_{2}\cdot \pi_{1}(x)$ is biharmonic if and only if 
\begin{align*}
0=&
m\langle \tau_{H}, \alpha_{1} \rangle\{
 (1- (\cot \langle \alpha_{1}, H\rangle )^{2})\\
& +2(1- (\cot \langle 2\alpha_{1}, H\rangle)^{2} ) 
+9(1- (\cot \langle 3\alpha_{1}, H\rangle)^{2} ) 
\}(2\alpha_{1}+\alpha_{2}).
\end{align*}
Then, we have
\begin{align*}
& (1- (\cot \langle \alpha_{1}, H\rangle )^{2})
+2(1- (\cot \langle 2\alpha_{1}, H\rangle)^{2} ) 
+9(1- (\cot \langle 3\alpha_{1}, H\rangle)^{2} ) \\
=&
12- 
\left[
(\cot \langle \alpha_{1}, H\rangle )^{2}
+2(\cot \langle 2\alpha_{1}, H\rangle )^{2}
+9\left(
\frac{ \cot \langle \alpha_{1}, H\rangle \cot \langle 2\alpha_{1}, H\rangle -1}{\cot \langle \alpha_{1}, H\rangle +\cot \langle 2\alpha_{1}, H\rangle} 
\right)^{2}
\right].
\end{align*}
Thus, 
the orbit $K_{2}\cdot \pi_{1}(x)$ is biharmonic if and only if 
\begin{align*}
0=
&\{ (\cot \langle \alpha_{1}, H\rangle )^{2}
+2(\cot \langle 2\alpha_{1}, H\rangle )^{2}\}
(\cot \langle \alpha_{1}, H\rangle +\cot \langle 2\alpha_{1}, H\rangle)^{2}\\
&+9(\cot \langle \alpha_{1}, H\rangle \cot \langle 2\alpha_{1}, H\rangle -1)^{2}
-12(\cot \langle \alpha_{1}, H\rangle +\cot \langle 2\alpha_{1}, H\rangle)^{2}\\
=&
\frac{(\tan \langle \alpha_{1}, H\rangle)^{4}}{8}
\{
45(\cot \langle \alpha_{1}, H\rangle)^{8}
-378(\cot \langle \alpha_{1}, H\rangle)^{6}\\
&+318(\cot \langle \alpha_{1}, H\rangle)^{4}
-30(\cot \langle \alpha_{1}, H\rangle)^{2}
+1
\}.
\end{align*}
We set $u=(\cot \langle \alpha_{1}, H\rangle)^{2}$ and
$$
f(u)=
45u^{4}
-378u^{3}
+318u^{2}
-30u
+1.
$$
Then, 
\begin{align*}
\frac{df}{du}(u)=&180u^{3}-1026u^{2}+636u-30=6(u-5)(30u^{2}-21u+1)\\
=&
180(u-5)\left( u- \frac{21+ \sqrt{321}}{60}\right)\left( u- \frac{21- \sqrt{321}}{60}\right).
\end{align*}
Since 
$$
f\left( \frac{1}{3}\right) = \frac{128}{9}  >0, \frac{df}{du}\left( \frac{1}{3}\right) =\frac{224}{3}>0, f(5)=-6824<0\  \text{and}\ f(7)=6112>0,
$$
the equation $f(u)=0$ has distinct two solutions for $(1/3)<u $.
Therefore, there exist $0<t_{-}, t_{+} <1$ such that 
the orbits $K_{2}\cdot \pi_{1}(\exp(t_{\pm}H_{\alpha_{1}}))$ are biharmonic.
Since 
$$
f\left( \frac{12+\sqrt{129}}{15}\right) \neq 0,
$$
the orbits $K_{2}\cdot \pi_{1}(\exp(t_{\pm}H_{\alpha_{1}}))$ are proper biharmonic.
{\em In this case, there exist exactly
two proper biharmonic orbits.}

\vspace{12pt}
(2)\ When $H \in P_{0}^{\{ \alpha_{2}, \delta\} }=\{tH_{\alpha_{2}} \mid 0<t<1\}$, 
we have $\Sigma^{+}_{H}=\{\alpha_{1}\}, W^{+}_{H}=\emptyset$.
By Theorem \ref{Theorem3.2}, 
we have 
\begin{align*}
\tau_{H}=
&-\frac{1}{2}m\{
5\cot \langle \alpha_{2}, H\rangle 
-\tan \langle \alpha_{2}, H\rangle
\} (3\alpha_{1}+2\alpha_{2}).
\end{align*}
Hence, 
$\tau_{H}=0$ if and only if 
$$
(\cot \langle \alpha_{2} , H \rangle)^{2}= \frac{1}{5}. 
$$
By Theorem \ref{thm: Bih. orbits of commutative Hermann actions}, the orbit $K_{2}\cdot \pi_{1}(x)$ is biharmonic if and only if 
\begin{align*}
0=
&\frac{1}{2}
m\langle \tau_{H}, \alpha_{2} \rangle 
\{
5(1- (\cot \langle \alpha_{2}, H\rangle )^{2}) 
+(1- (\tan \langle \alpha_{2}, H\rangle )^{2})
+4 
\}(3\alpha_{1}+2\alpha_{2}).
\end{align*}
Therefore, the orbit $K_{2}\cdot \pi_{1}(x)$ is biharmonic if and only if
$\tau_{H}=0$ or 
$$
5(1- (\cot \langle \alpha_{2}, H\rangle )^{2}) 
+(1- (\tan \langle \alpha_{2}, H\rangle )^{2})
+4 =0
$$
holds. 
The last equation is equivalent to
$$
\big( 5(\cot \langle \alpha_{2}, H \rangle)^{2} -1\big) ( (\cot \langle \alpha_{2}, H \rangle)^{2}-1)=4(\cot \langle \alpha_{2}, H \rangle)^{2}.
$$
Thus, the solutions of the equation are not harmonic.
Hence $K_{2}\cdot \pi_{1}(x)$ is proper biharmonic if and only if
$$
(\cot \langle \alpha_{2}, H \rangle )^{2}=
\frac{5 \pm 2\sqrt{ 5} }{5}
$$
holds.
{\em In this case, there exist exactly
two proper biharmonic orbits.}


\vspace{12pt}
(3)\ When $H \in P_{0}^{\{ \alpha_{1}, \alpha_{2}\} }=\{ tH_{\alpha_{1}}+(1-t)H_{\alpha_{2}} \mid 0<t<1\}$, 
we have $\Sigma^{+}_{H}=\{3\alpha_{1} +2\alpha_{2}\}, W^{+}_{H}=\emptyset$.
We set $\vartheta=(\pi/6)t$.
Then, 
$$
\langle \alpha_{1}, H \rangle =2\vartheta,\  \langle \alpha_{2}, H \rangle =\frac{\pi}{2}-3\vartheta.
$$ 
By Theorem \ref{Theorem3.2}, 
we have 
\begin{align*}
\tau_{H}=&
-m\{\cot (2\vartheta) -\tan \vartheta -\tan(3\vartheta) \}\alpha_{1}.
\end{align*}
Since 
$$
\tan (3\vartheta)=\frac{\cot \vartheta+\cot (2\vartheta)}{\cot \vartheta \cot (2\vartheta) -1},
$$
$\tau_{H}=0$ if and only if,
\begin{align*}
0=&
(\cot (2\vartheta) -\tan \vartheta ) 
(\cot \vartheta \cot (2\vartheta) -1)
-3(\cot \vartheta+\cot (2\vartheta))\\
=&
\frac{(\cot \vartheta )^{4} -24 (\cot \vartheta )^{2}+15}{\cot \vartheta}.
\end{align*}
Since $0<\vartheta <(\pi/6)$, $\cot \vartheta >\sqrt{3}$.
Hence $\tau_{H}=0$ if and only if,
$$
(\cot \vartheta)^{2}=12 +\sqrt{129}.
$$
By Theorem \ref{thm: Bih. orbits of commutative Hermann actions}, the orbit $K_{2}\cdot \pi_{1}(x)$ is biharmonic if and only if 
\begin{align*}
0=&
\frac{m}{2}\langle \tau_{H}, \alpha_{1} \rangle
\{
2(1-(\cot (2\vartheta) )^{2}) 
+ (1-(\tan  \vartheta )^{2}) 
+9(1-(\tan (3\vartheta) )^{2}) 
\}\alpha_{1}.
\end{align*}
Therefore, $K_{2}\cdot \pi_{1}(x)$ is biharmonic if and only if
$\tau_{H}=0$ or 
\begin{align*}
0=&2(1-(\cot (2\vartheta) )^{2}) 
+ (1-(\tan  \vartheta )^{2}) 
+9(1-(\tan (3\vartheta) )^{2})\\
=&
(12- 2(\cot (2\vartheta) )^{2}- (\tan \vartheta )^{2}) 
-9 \frac{(\cot (2\vartheta) +\cot \vartheta )^{2}}{((\cot (2\vartheta) )(\cot \vartheta )-1)^{2}}
\end{align*}
holds. 
Thus $K_{2}\cdot \pi_{1}(x)$ is biharmonic if and only if
$\tau_{H}=0$ or 
\begin{align*}
0=&\{12- 2(\cot (2\vartheta) )^{2}- (\tan \vartheta )^{2}\} ((\cot (2\vartheta) )(\cot \vartheta -1)^{2}
-9 (\cot (2\vartheta) +\cot \vartheta )^{2}\\
=&
-\frac{1}{8}(\tan \vartheta)^{2}\{(\cot \vartheta)^{8}-32(\cot \vartheta)^{6}+330(\cot \vartheta)^{4}-360(\cot \vartheta)^{2}+45\}.
\end{align*}

We set $u=(\cot \vartheta)^{2}$ and
$$
f(u)=
u^{4}
-32u^{3}
+330u^{2}
-360u
+45.
$$
Then, 
\begin{align*}
\frac{df}{du}(u)&=4(3u^{3}-24u^{2}+165u-90), \\
\frac{d^{2}f}{du^{2}}(u)&=12(u-5)(u-11).
\end{align*}
Since 
$$
f(3)=1152>0,\ \frac{df}{du}(3)=864>0,\ \frac{df}{du}(11)=608>0,
$$
$(df/du)(u)>0 $ and $f(u)>0$ for $3<u$.
Thus the equation $f(u)=0$ has no solution for $3<u $.
{\em Therefore,
if the orbits $K_{2}\cdot \pi_{1}(x)$ is biharmonic,
then it is harmonic.}

\subsection*{Cases of $\theta_{1}\not\sim \theta_{2}$}

Next, we consider the cases of $\theta_{1}\not\sim \theta_{2}$. 
Let $(G, K_{1}, K_{2})$ be a compact symmetric triad which satisfies the condition (A), (B) or (C) in Theorem \ref{Theorem:I2 AB}.
Then the triple $(\tilde{\Sigma}, \Sigma, W)$ is a symmetric triad of $\mathfrak{a}$ with multiplicities.
From (\ref{eq:cell_of_Hermann_action}) and (\ref{cell_decomposition_of_Hermann_action}), the orbit spaces of $K_{2}$-action on $N_{1}$ and $K_{1}$-action on $N_{2}$ are described as 
$\overline{P_{0}}=\{H \in \mathfrak{a} \mid \langle \alpha, H \rangle \geq 0, \langle \tilde{\alpha}, H \rangle \leq (\pi/2)\ (\alpha\in \Pi)\}$
where $\tilde{\alpha}$ is a unique element in $W^{+}$ which satisfies $\alpha + \lambda  \not\in W$ for all $\lambda \in \Pi$.
We set a basis $\{H_{\alpha}\}_{\alpha \in \Pi}$ of $\mathfrak{a}$ as follows;
$$
\langle H_{\alpha} , \beta \rangle=0, \quad
\langle H_{\alpha}, \tilde{\alpha}\rangle =\frac{\pi}{2}\quad   (\alpha \neq \beta, \alpha, \beta \in \Pi).  
$$
Then we have 
$$
P_{0}=\left\{\left. \sum_{\alpha \in \Pi} t_{\alpha} H_{\alpha}\  \right| \  t_{\alpha }>0 \ (\alpha \in \Pi), \sum_{\alpha \in \Pi} t_{\alpha}<1  \right\}.
$$
We apply Theorem \ref{thm: Bih. orbits of commutative Hermann actions} 
to the following three cases; 
\begin{enumerate}
\item $H \in P_{0}^{\{ \alpha_{1}, \tilde{\alpha}\} }=\{tH_{\alpha_{1}} \mid 0<t<1\}$,
\item $H \in P_{0}^{\{ \alpha_{2}, \tilde{\alpha}\} }=\{tH_{\alpha_{2}} \mid 0<t<1\}$,
\item $H \in P_{0}^{\{ \alpha_{1}, \alpha_{2}\} }=\{tH_{\alpha_{1}}+(1-t)H_{\alpha_{2}} \mid 0<t<1\}$.
\end{enumerate}
In the following, we solve the equation (\ref{bih.eq in thm 4.1}) in Theorem \ref{thm: Bih. orbits of commutative Hermann actions} 
for each symmetric triad with multiplicities which satisfies $\dim \mathfrak{a}=2$.
\subsection{Type $\rm{I\mathchar`-B}_{2}$ and $\rm{I\mathchar`-BC}_{2}\rm{\mathchar`-A}_{1}^{2}$}
We set 
\begin{align*}
\Sigma^{+}&=\{ e_{1}\pm e_{2}, e_{1}, e_{2}, 2e_{1}, 2e_{2}\},\ 
W^{+}=\{ e_{1}, e_{2} \}, \\
\Pi& =\{\alpha_{1}=e_{1}-e_{2}, \alpha_{2}= e_{2}\},\
\tilde{\alpha}=\alpha_{1}+\alpha_{2}=e_{1}
\end{align*}
and 
\begin{align*}
m_{1}=m(e_{1}),\ m_{2}=m(e_{1}+e_{2}),\ m_{3}=m(2e_{1}),\
n_{1}=n(e_{1}),
\end{align*}
where, if $(\tilde{\Sigma}, \Sigma, W)$ is type $\rm{I\mathchar`-B}_{2}$, 
then $m_{3}=0$.
\\

(1)\ When $H \in P_{0}^{\{ \alpha_{1}, \tilde{\alpha}\} }=\{tH_{\alpha_{1}} \mid 0<t<1\}$,
we have $\Sigma^{+}_{H}=\{\alpha_{2}, 2\alpha_{2}\}$ and $W^{+}_{H}=\emptyset$.
By Theorem \ref{Theorem3.2}, 
we have 
\begin{align*}
\tau_{H}=&
\{
-(2m_{2} +m_{1}+m_{3}) \cot \langle \alpha_{1},H\rangle 
+(n_{1}+m_{3})\tan\langle \alpha_{1} , H \rangle 
\}e_{1}.
\end{align*}
Hence we have that 
$\tau_{H}=0$ if and only if 
$$
(\cot \langle \alpha_{1} , H \rangle)^{2}= \frac{n_{1}+m_{3}}{m_{1}+2m_{2}+m_{3}}. 
$$
By Theorem \ref{thm: Bih. orbits of commutative Hermann actions}, the orbit $K_{2}\cdot \pi_{1}(x)$ is biharmonic if and only if 
\begin{align*}
0=&\langle \tau_{H}, \alpha_{1} \rangle
\{
(m_{1}+2m_{2}+m_{3}) (1-(\cot \langle \alpha_{1}, H\rangle )^{2}) 
+4m_{3} \\
&\qquad + (n_{1}+m_{3}) (1-(\tan \langle \alpha_{1}, H\rangle )^{2})
\} e_{1}.
\end{align*}
Therefore, 
$K_{2}\cdot \pi_{1}(x)$ is biharmonic if and only if 
$\tau_{H}=0$ or 
\begin{align} 
&(m_{1}+2m_{2}+m_{3})(\cot \langle \alpha_{1}, H\rangle)^{4}\label{bih. eq 1-B_{2}} \\
-&\{ (m_{1}+2m_{2}+m_{3}) + (n_{1}+m_{3})+4m_{3}\} (\cot \langle \alpha_{1}, H\rangle)^{2}
+n_{1}+m_{3}=0 \nonumber
\end{align}
holds.
Let $H_{+}$ and $H_{-}$ denote the solutions of the biharmonic equation (\ref{bih. eq 1-B_{2}}) such that 
$(\cot \langle \alpha_{1} , H_{-} \rangle)^{2} \leq (\cot \langle \alpha_{1} , H_{+} \rangle)^{2}$.
Since 
$\tau_{H}=0$ if and only if 
$$
(\cot \langle \alpha_{1} , H \rangle)^{2}= \frac{n_{1}+m_{3}}{m_{1}+2m_{2}+m_{3}}, 
$$
$K_{2}\cdot \pi_{1}(x)$ is proper biharmonic if and only if 
\begin{align*}
&(\cot \langle \alpha_{1} , H \rangle)^{2}\\
=&\begin{cases}
\frac{-(m_{1}+2m_{2}+6m_{3}+n_{1}) \pm \sqrt{(m_{1}+2m_{2}+6m_{3}+n_{1})^{2}-4(m_{1}+2m_{2}+m_{3}) (n_{1}+m_{3})}}{2(m_{1}+2m_{2}+m_{3})}
& (m_{3}>0)\\
1 
& (m_{3}=0).
\end{cases}
\end{align*} 
Let $H_{0}$ be a vector in $\mathfrak{a}$ satisfying $\tau_{H_{0}}=0$ and $0<\langle \alpha_{1}, H_{0} \rangle <\pi/2$.
\begin{itemize}
\item
{\em If $m_{3}=0$, then 
there exists a unique proper biharmonic orbit.}

\item {\em If $m_{3}>0$, then 
$$
\langle \alpha_{1} ,H_{-} \rangle<\langle \alpha_{1} ,H_{0} \rangle < \langle \alpha_{1} ,H_{+} \rangle,
$$
hence there exist exactly
two proper biharmonic orbits.}
\end{itemize}

\vspace{12pt}
(2)\ When $H \in P_{0}^{\{ \alpha_{2}, \tilde{\alpha}\} }=\{tH_{\alpha_{2}} \mid 0<t<1\}$, 
we have $\Sigma^{+}_{H}=\{\alpha_{1} \}, W^{+}_{H}=\emptyset$.
By Theorem \ref{Theorem3.2}, 
we have 
\begin{align*}
\tau_{H}=&
\frac{1}{2}
\{
 -(2m_{1}+m_{2}+2m_{3})\cot \langle \alpha_{2}, H \rangle \\
&\quad +(2n_{1}+m_{2}+2m_{3})\tan \langle \alpha_{2}, H \rangle
\}
(\alpha_{1}+2\alpha_{2}).
\end{align*}
Hence, $\tau_{H}=0$ if and only if 
$$
(\cot \langle \alpha_{2} , H \rangle)^{2}= \frac{2n_{1}+m_{2}+2m_{3}}{2m_{1}+m_{2}+2m_{3}}. 
$$
By Theorem \ref{thm: Bih. orbits of commutative Hermann actions}, the orbit $K_{2}\cdot \pi_{1}(x)$ is biharmonic if and only if 
\begin{align*}
0=&
\langle \tau_{H}, \alpha_{2} \rangle  
\{
m_{1}(1-(\cot \langle \alpha_{2}, H\rangle )^{2}) 
 +(2m_{2}+4m_{3})(1-(\cot \langle 2\alpha_{2} , H\rangle )^{2}) \\
&+n_{1}(1-(\tan \langle \alpha_{2}, H\rangle )^{2})   
\}(\alpha_{1}+2\alpha_{2}).
\end{align*}
Therefore, 
$K_{2}\cdot \pi_{1}(x)$ is biharmonic if and only if 
$\tau_{H}=0$ or 
\begin{align*} 
0=&m_{1}(1-(\cot \langle \alpha_{2}, H\rangle )^{2}) 
 +(2m_{2}+4m_{3})(1-(\cot \langle 2\alpha_{2} , H\rangle )^{2}) \\
&+n_{1}(1-(\tan \langle \alpha_{2}, H\rangle )^{2})
\end{align*}
holds.
The last equation is equivalent to 
\begin{align*} 
&\big( (2m_{1}+m_{2}+2m_{3}) (\cot \langle \alpha_{2} , H\rangle )^{2}-(2n_{1}+m_{2}+2m_{3})\big) ((\cot \langle \alpha_{2} , H\rangle )^{2}-1)\\
=& (2m_{2}+4m_{3})(\cot \langle \alpha_{2} , H\rangle )^{2}.
\end{align*}
Since $2m_{2}+4m_{3}>0$, the solutions of the equation are not harmonic.
Hence the orbit $K_{2}\cdot \pi_{1}(x)$ is proper biharmonic if and only if
\begin{align*}
(\cot \langle \alpha_{2}, H \rangle )^{2}
=
\frac{ l_{1}+l_{2}+2m_{2}+4m_{3}  \pm \sqrt{(l_{1}+l_{2}+2m_{2}+4m_{3})^{2} - 4l_{1}l_{2} } }{2l_{1}}
\end{align*}
holds, 
where $l_{1}=2m_{1}+m_{2}+2m_{3}, l_{2}=2n_{1}+m_{2}+2m_{3}$.
{\em In this case, there exist exactly
two proper biharmonic orbits.}

\vspace{12pt}
(3)\ When $H \in P_{0}^{\{ \alpha_{1}, \alpha_{2}\} }=\{tH_{\alpha_{1}}+(1-t)H_{\alpha_{2}} \mid 0<t<1\}$, 
we have $\Sigma^{+}_{H}=\{ 2\alpha_{2}+2\alpha_{2} \}, W^{+}_{H}=\{ \alpha_{1}+\alpha_{2}\}$.
We set $\vartheta =\langle \alpha_{1}, H \rangle $.
Then, $\langle \alpha_{2}, H \rangle=(\pi/2)-\vartheta$.
By Theorem \ref{Theorem3.2}, 
we have 
\begin{align*}
\tau_{H}=&
\{(2m_{2}+m_{3}+n_{1})\cot \vartheta - (m_{1}+m_{3})\tan \vartheta
\}\alpha_{2}.
\end{align*}
Hence, $\tau_{H}=0$ if and only if 
$$
(\cot \vartheta)^{2}= \frac{m_{1}+m_{3}}{2m_{2}+m_{3}+n_{1}}. 
$$
By Theorem \ref{thm: Bih. orbits of commutative Hermann actions}, the orbit $K_{2}\cdot \pi_{1}(x)$ is biharmonic if and only if 
\begin{align*}
0=&
-\langle \tau_{H}, \alpha_{1} \rangle 
\{
 (2m_{2}+n_{1}+m_{3})(1- (\cot \vartheta)^{2} )\\
&+(m_{1}+m_{3})(1- (\tan \vartheta)^{2} )
+4m_{3}
\}\alpha_{2}.
\end{align*}
Therefore, 
$K_{2}\cdot \pi_{1}(x)$ is biharmonic if and only if 
$\tau_{H}=0$ or 
\begin{align*} 
0=\{
 (2m_{2}+n_{1}+m_{3})(1- (\cot \vartheta)^{2} )+(m_{1}+m_{3})(1- (\tan \vartheta)^{2} )
+4m_{3}
\}
\end{align*}
holds.
The last equation is equivalent to 
\begin{align*} 
\big( (2m_{2}+m_{3}+n_{1}) (\cot \vartheta )^{2}-(m_{1}+m_{3})\big) ((\cot \vartheta )^{2}-1)
=   4m_{3}(\cot \vartheta )^{2}.
\end{align*}
Since $m_{3}>0$, the solutions of the equation are not harmonic.
Hence the orbit $K_{2}\cdot \pi_{1}(x)$ is proper biharmonic if and only if
\begin{align*}
(\cot \vartheta)^{2}
=
\frac{l_{1}+l_{2}+4m_{3}\pm \sqrt{(l_{1}+l_{2}+4m_{3})^{2} -4l_{1}l_{2}} }{2l_{1}}
\end{align*}
holds,
where $l_{1}=2m_{2}+m_{3}+n_{1}, l_{2}=m_{1}+m_{3}$. 
{\em In this case, there exist exactly
two proper biharmonic orbits.}
\\
\subsection{Type $\rm{I\mathchar`-C}_{2}$}
We set 
\begin{align*}
\Sigma^{+}&=\{ e_{1}\pm e_{2}, 2e_{1}, 2e_{2}\},\ 
W^{+}=\{ e_{1}-e_{2}, e_{1}+e_{2} \}, \\
\Pi& =\{\alpha_{1}=e_{1}-e_{2}, \alpha_{2}= 2e_{2}\},\ 
\tilde{\alpha}=\alpha_{1}+\alpha_{2} =e_{1}+e_{2},
\end{align*}
and $m_{1}=m(e_{1}+e_{2}), m_{2}=m(2e_{1}), n_{1}=n(e_{1}+e_{2})$.
{\em In this case, we have the same result as cases of Type $\rm{I\mathchar`-B}_{2}$.}
\\
\subsection{Type $\rm{I\mathchar`-BC}_{2}\rm{\mathchar`-B}_{2}$}
We set 
\begin{align*}
\Sigma^{+}&=\{ e_{1}\pm e_{2}, e_{1}, e_{2}, 2e_{1}, 2e_{2}\},\ 
W^{+}=\{ e_{1}\pm e_{2}, e_{1}, e_{2}\}, \\
\Pi& =\{\alpha_{1}=e_{1}-e_{2}, \alpha_{2}= e_{2}\},\
\tilde{\alpha}=\alpha_{1}+2\alpha_{2}=e_{1}+e_{2}
\end{align*}
and 
\begin{align*}
m_{1}=m(e_{1}),\ m_{2}=m(e_{1}+e_{2}),\ m_{3}=m(2e_{1}),\
n_{1}=n(e_{1}),\ n_{2}=n(e_{1}+e_{2}).
\end{align*}
Since 
$e_{1}\in \Sigma \cap W,\ e_{1}-e_{2} \in W$ and 
$2\langle e_{1}, e_{1}-e_{2} \rangle / \langle e_{1}-e_{2}, e_{1}-e_{2} \rangle$ is odd, 
by definition of multiplicities, we have 
$m_{1}=m(e_{1})=n(e_{1})=n_{1}$.
\\

(1)\ When $H \in P_{0}^{\{ \alpha_{1}, \tilde{\alpha}\} }=\{tH_{\alpha_{1}} \mid 0<t<1\}$,
we have $\Sigma^{+}_{H}=\{\alpha_{2}, 2\alpha_{2}\}, W^{+}_{H}=\emptyset$.
By Theorem \ref{Theorem3.2}, 
we have 
\begin{align*}
\tau_{H}=&
\{
-(m_{1}+2m_{2}+m_{3})\cot \langle \alpha_{1}, H\rangle + (m_{1}+2n_{2}+m_{3})\tan \langle \alpha_{1}, H\rangle 
\} e_{1}.
\end{align*}
Hence we have 
$\tau_{H}=0$ if and only if 
$$
(\cot \langle \alpha_{1} , H \rangle)^{2}=\frac{m_{1}+2n_{2}+m_{3}}{m_{1}+2m_{2}+m_{3}}.
$$
By Theorem \ref{thm: Bih. orbits of commutative Hermann actions}, the orbit $K_{2}\cdot \pi_{1}(x)$ is biharmonic if and only if 
\begin{align*}
0=&
\langle \tau_{H}, \alpha_{1} \rangle  
\{
  (m_{1}+2m_{2}+m_{3})(1-(\cot \langle \alpha_{1}, H\rangle )^{2}) \\
&+(m_{1}+2n_{2}+m_{3})(1-(\tan \langle \alpha_{1}, H\rangle )^{2}) 
+4m_{3} 
\}(\alpha_{1}+\alpha_{2}).
\end{align*}
Therefore, 
$K_{2}\cdot \pi_{1}(x)$ is biharmonic if and only if 
$\tau_{H}=0$ or 
\begin{align*} 
0=&  (m_{1}+2m_{2}+m_{3})(1-(\cot \langle \alpha_{1}, H\rangle )^{2}) \\
&+(m_{1}+2n_{2}+m_{3})(1-(\tan \langle \alpha_{1}, H\rangle )^{2}) 
+4m_{3} 
\end{align*}
holds.
The last equation is equivalent to 
\begin{align*} 
&\big( (m_{1}+2m_{2}+m_{3}) (\cot \langle \alpha_{2} , H\rangle )^{2}-(m_{1}+2n_{2}+m_{3})\big) ((\cot \langle \alpha_{2} , H\rangle )^{2}-1)\\
=& 4m_{3}(\cot \langle \alpha_{2} , H\rangle )^{2}.
\end{align*}
Since $m_{3}>0$, the solutions of the equation are not harmonic.
Hence the orbit $K_{2}\cdot \pi_{1}(x)$ is proper biharmonic if and only if
\begin{align*}
(\cot \langle \alpha_{1}, H \rangle)^{2}
=&
\frac{l_{1}+l_{2}+4m_{3} \pm \sqrt{(l_{1}+l_{2}+4m_{3})^{2}- 4l_{1}l_{2}  }  }{2l_{1}}
\end{align*}
holds,
where 
$l_{1}=m_{1}+2m_{2}+m_{3}, l_{2}=m_{1}+2n_{2}+m_{3}$.
{\em In this case, there exist exactly
two proper biharmonic orbits.}

\vspace{12pt}
(2)\ When $H \in P_{0}^{\{ \alpha_{2}, \tilde{\alpha}\} }=\{tH_{\alpha_{2}} \mid 0<t<1\}$,
we have $\Sigma^{+}_{H}=\{\alpha_{1}\}, W^{+}_{H}=\emptyset$.
By Theorem \ref{Theorem3.2}, 
we have 
\begin{align*}
\tau_{H}
=&
\{
-(2m_{1}+m_{2}+2m_{3})\cot\langle 2\alpha_{2}, H \rangle
+n_{2}\tan\langle 2\alpha_{2}, H \rangle
\}(\alpha_{1}+2\alpha_{2}).
\end{align*}
Hence we have 
$\tau_{H}=0$ if and only if 
$$
(\cot \langle 2\alpha_{2} , H \rangle)^{2}
=\frac{n_{2}}{2m_{1}+m_{2}+2m_{3}}.
$$
By Theorem \ref{thm: Bih. orbits of commutative Hermann actions}, the orbit $K_{2}\cdot \pi_{1}(x)$ is biharmonic if and only if 
\begin{align*}
0=&
2\langle \tau_{H}, \alpha_{2} \rangle  
\{
  (2m_{1}+m_{2}+2m_{3})(1-(\cot \langle 2\alpha_{2}, H\rangle )^{2}) \\
& +n_{2}(1-(\tan \langle 2\alpha_{2}, H\rangle )^{2}) 
-4m_{3} 
\}(\alpha_{1}+ 2\alpha_{2}).
\end{align*}
Therefore, 
$K_{2}\cdot \pi_{1}(x)$ is biharmonic if and only if 
$\tau_{H}=0$ or 
\begin{align*} 
0=&(2m_{1}+m_{2}+2m_{3})(1-(\cot \langle 2\alpha_{2}, H\rangle )^{2}) \\
& +n_{2}(1-(\tan \langle 2\alpha_{2}, H\rangle )^{2}) 
-4m_{3}  
\end{align*}
holds.
The last equation is equivalent to 
\begin{align*} 
&\big( (2m_{1}+m_{2}+2m_{3}) (\cot \langle 2\alpha_{2} , H\rangle )^{2}-2n_{2}\big) ((\cot \langle \alpha_{2} , H\rangle )^{2}-1)\\
=&  -4m_{3}(\cot \langle \alpha_{2} , H\rangle )^{2}.
\end{align*}
Since $m_{3}>0$, the solutions of the equation are not harmonic.
\begin{itemize}
\item 
{\em When 
$
(-2m_{1}+m_{2}+2m_{3}+n_{2})^{2}-4(2m_{1}+m_{2}+2m_{3})n_{2}>0
$
the orbit $K_{2}\cdot \pi_{1}(x)$ is proper biharmonic if and only if
\begin{align*}
(\cot \langle 2\alpha_{2}, H \rangle)^{2}
=
\frac{l \pm \sqrt{l^{2}-4(2m_{1}+m_{2}+2m_{3})n_{2} }  }{2(2m_{1}+m_{2}+2m_{3})}
\end{align*}
holds, 
where $l=-2m_{1}+m_{2}+2m_{3}+n_{2}$.
In this case, there exist exactly
two proper biharmonic orbits.}

\item {\em When $(-2m_{1}+m_{2}+2m_{3}+n_{2})^{2}-4(2m_{1}+m_{2}+2m_{3})n_{2}<0$,
if the orbit $K_{2}\cdot \pi_{1}(x)$ is biharmonic, then it is harmonic.}

\item {\em When $(-2m_{1}+m_{2}+2m_{3}+n_{2})^{2}-4(2m_{1}+m_{2}+2m_{3})n_{2}=0$,
there exists unique proper biharmonic orbit.}
\end{itemize}

\vspace{12pt}
(3)\ When $H \in P_{0}^{\{ \alpha_{1}, \alpha_{2}\} }=\{tH_{\alpha_{1}}+(1-t)H_{\alpha_{2}} \mid 0<t<1\}$, 
we have $\Sigma^{+}_{H}=\emptyset, W^{+}_{H}=\{\alpha_{1}+2\alpha_{2}\}$.
We set $2\vartheta=\langle \alpha_{1}, H\rangle $.
Then $\langle \alpha_{2}, H\rangle =(\pi/4)- \vartheta$.
By Theorem \ref{Theorem3.2}, 
we have 
\begin{align*}
\tau_{H}=&
\{
-m_{2}\cot (2\vartheta )
+(2m_{1}+m_{3}+n_{2})\tan (2\vartheta)
\}\alpha_{1}.
\end{align*}
Hence we have 
$\tau_{H}=0$ if and only if 
$$
(\cot (2\vartheta))^{2}
=\frac{2m_{1}+2m_{3}+n_{2}}{m_{2}}.
$$
By Theorem \ref{thm: Bih. orbits of commutative Hermann actions}, the orbit $K_{2}\cdot \pi_{1}(x)$ is biharmonic if and only if 
\begin{align*}
0=&
2\langle \tau_{H}, \alpha_{1}+\alpha_{2} \rangle  
\{
  m_{2} (1-(\cot (2\vartheta) )^{2})\\
&+(2m_{1}+2m_{3}+n_{2})(1-(\tan (2\vartheta) )^{2}) -2m_{1}
\}\alpha_{1}.
\end{align*}
Therefore, 
$K_{2}\cdot \pi_{1}(x)$ is biharmonic if and only if 
$\tau_{H}=0$ or 
\begin{align*} 
0=&(m_{2} (1-(\cot (2\vartheta) )^{2})
+(2m_{1}+2m_{3}+n_{2})(1-(\tan (2\vartheta) )^{2}) -2m_{1}
\end{align*}
holds.
The last equation is equivalent to 
\begin{align*} 
\{ m_{2}(\cot (2\vartheta) )^{2}- (2m_{1}+2m_{3}+n_{2})\} 
((\cot (2\vartheta) )^{2}-1)=-2m_{1}(\cot (2\vartheta) )^{2}.
\end{align*}
Since $m_{1}>0$, the solutions of the equation are not harmonic.
\begin{itemize}
\item {\em When 
$
(2m_{3}+m_{2}+m_{2})^{2}-4m_{2}(2m_{1}+2m_{3}+n_{2})>0,
$
the orbit $K_{2}\cdot \pi_{1}(x)$ is proper biharmonic if and only if
\begin{align*}
(\cot (2\vartheta ))^{2}
=
\frac{2m_{3}+m_{2}+m_{2} \pm \sqrt{(2m_{3}+m_{2}+m_{2})^{2}-4m_{2}(2m_{1}+2m_{3}+n_{2}) }  }{2m_{2}}
\end{align*}
holds.
In this case, there exist exactly
two proper biharmonic orbits.}
\item {\em When $(2m_{3}+m_{2}+m_{2})^{2}-4m_{2}(2m_{1}+2m_{3}+n_{2})<0$, 
if the orbit $K_{2}\cdot \pi_{1}(x)$ is biharmonic, then it is harmonic.}

\item {\em When $(2m_{3}+m_{2}+m_{2})^{2}-4m_{2}(2m_{1}+2m_{3}+n_{2})=0$,
there exists unique proper biharmonic orbit.}
\end{itemize}
\subsection{Type $\rm{II\mathchar`-BC}_{2}$} 
We set 
\begin{align*}
\Sigma^{+}&=\{ e_{1}\pm e_{2}, e_{1}, e_{2}\}, \
W^{+}=\{ e_{1}\pm e_{2}, e_{1}, e_{2}, 2e_{1}, 2e_{2}\}, \\
\Pi& =\{\alpha_{1}=e_{1}-e_{2}, \alpha_{2}= e_{2}\},\
\tilde{\alpha}=2\alpha_{1}+2\alpha_{2}=2e_{1}
\end{align*}
and 
\begin{align*}
m_{1}=m(e_{1}),\ m_{2}=m(e_{1}+e_{2}),
n_{1}=n(e_{1}),\ n_{2}=n(e_{1}+e_{2}), \ n_{3}=n(2e_{1}).
\end{align*}
Since 
$e_{1}, e_{1}+e_{2} \in \Sigma \cap W,\ 2e_{1}\in W$ and 
$2\langle e_{1}, 2e_{1}\rangle / \langle 2e_{1}, 2e_{1}\rangle=1$ and 
$2\langle e_{1}+e_{2}, 2e_{1}\rangle / \langle 2e_{1}, 2e_{1}\rangle=1$ are odd, 
by definition of multiplicities, we have 
$m_{1}=m(e_{1})=n(e_{1})=n_{1}$, 
$m_{2}=m(e_{1}+e_{2})=n(e_{1}+e_{2})=n_{2}$.
\\

(1)\ When $H \in P_{0}^{\{ \alpha_{1}, \tilde{\alpha}\} }=\{tH_{\alpha_{1}} \mid 0<t<1\}$,
we have $\Sigma^{+}_{H}=\{\alpha_{2}\}, W^{+}_{H}=\emptyset$.
By Theorem \ref{Theorem3.2}, 
we have 
\begin{align*}
\tau_{H}=&
2\{
-(m_{1}+2m_{2})\cot \langle 2\alpha_{1}, H\rangle + n_{3}\tan \langle 2\alpha_{1}, H\rangle 
\} e_{1}.
\end{align*}
Hence we have 
$\tau_{H}=0$ if and only if 
$$
(\cot \langle 2\alpha_{1} , H \rangle)^{2}=\frac{n_{3}}{m_{1}+2m_{2}}.
$$
By Theorem \ref{thm: Bih. orbits of commutative Hermann actions}, the orbit $K_{2}\cdot \pi_{1}(x)$ is biharmonic if and only if 
\begin{align*}
0=&
8\langle \tau_{H}, \alpha_{1} \rangle  
\{
  -(m_{1}+2m_{2})(\cot \langle 2\alpha_{1}, H\rangle )^{2} +n_{3}(1-(\tan \langle 2\alpha_{1}, H\rangle )^{2}) 
\}(\alpha_{1}+\alpha_{2}).
\end{align*}
Therefore, 
$K_{2}\cdot \pi_{1}(x)$ is biharmonic if and only if 
$\tau_{H}=0$ or 
\begin{align*} 
0=&  -(m_{1}+2m_{2})(\cot \langle 2\alpha_{1}, H\rangle )^{2} +n_{3}(1-(\tan \langle 2\alpha_{1}, H\rangle )^{2})  
\end{align*}
holds.
The last equation is equivalent to 
\begin{align*} 
&\{ (m_{1}+2m_{2}) (\cot \langle 2\alpha_{1} , H\rangle )^{2}-n_{3}\} ((\cot \langle 2\alpha_{1} , H\rangle )^{2}-1)\\
=&  -(m_{1}+2m_{2})(\cot \langle 2\alpha_{1} , H\rangle )^{2}.
\end{align*}
Since $m_{1}+2m_{2}>0$, the solutions of the equation are not harmonic.
\begin{itemize}
\item {\em When $n_{3}^{2}-4(m_{1}+2m_{2})n_{3}>0$, 
the orbit $K_{2}\cdot \pi_{1}(x)$ is proper biharmonic if and only if
\begin{align*}
(\cot \langle 2\alpha_{1}, H \rangle )^{2}
=&
\frac{n_{3} \pm \sqrt{n_{3}^{2}-4(m_{1}+2m_{2})n_{3} }  }{2(m_{1}+2m_{2})}
\end{align*}
holds.
In this case, there exist exactly
two proper biharmonic orbits.}

\item {\em When $n_{3}^{2}-4(m_{1}+2m_{2})n_{3}<0$,
if the orbit $K_{2}\cdot \pi_{1}(x)$ is biharmonic, then it is harmonic.}

\item {\em When $n_{3}^{2}-4(m_{1}+2m_{2})n_{3}=0$,
there exists unique proper biharmonic orbit.}
\end{itemize}

\vspace{12pt}
(2)\ When $H \in P_{0}^{\{ \alpha_{2}, \tilde{\alpha}\} }=\{tH_{\alpha_{2}} \mid 0<t<1\}$,
we have $\Sigma^{+}_{H}=\{\alpha_{1}\}, W^{+}_{H}=\emptyset$.
By Theorem \ref{Theorem3.2}, 
we have 
\begin{align*}
\tau_{H}=&
\{
-(2m_{1}+m_{2})\cot\langle 2\alpha_{2}, H \rangle
+(m_{2}+2n_{3})\tan\langle 2\alpha_{2}, H \rangle
\}(\alpha_{1}+2\alpha_{2}).
\end{align*}
Hence we have 
$\tau_{H}=0$ if and only if 
$$
(\cot \langle 2\alpha_{2} , H \rangle)^{2}
=\frac{m_{2}+2n_{3}}{2m_{1}+m_{2}}.
$$
By Theorem \ref{thm: Bih. orbits of commutative Hermann actions}, the orbit $K_{2}\cdot \pi_{1}(x)$ is biharmonic if and only if 
\begin{align*}
0=&
\langle \tau_{H}, \alpha_{2} \rangle  
\{
 -(4m_{1}+2m_{2})(\cot \langle 2\alpha_{2}, H\rangle )^{2}  \\
&  +(2m_{2}+4n_{3})(1-(\tan \langle 2\alpha_{2}, H\rangle )^{2})
-4m_{1} 
\}(\alpha_{1}+2\alpha_{2}).
\end{align*}
Therefore, 
$K_{2}\cdot \pi_{1}(x)$ is biharmonic if and only if 
$\tau_{H}=0$ or 
\begin{align*} 
0=&(2m_{1}+m_{2})(1-(\cot \langle 2\alpha_{2}, H\rangle )^{2}) \\
& +(m_{2}+2n_{3})(1-(\tan \langle 2\alpha_{2}, H\rangle )^{2})
-2m_{1}
\end{align*}
holds.
The last equation is equivalent to 
\begin{align*} 
&\big( (2m_{1}+m_{2}) (\cot \langle 2\alpha_{2} , H\rangle )^{2}-(m_{2}+2n_{3})\big) ((\cot \langle 2\alpha_{2} , H\rangle )^{2}-1)\\
=&  -2m_{1}(\cot \langle 2\alpha_{2} , H\rangle )^{2}.
\end{align*}
Since $2m_{1}>0$, the solutions of the equation are not harmonic.
\begin{itemize}
\item {\em When 
$
(m_{2}+n_{3})^{2}-(2m_{1}+m_{2})(m_{2}+2n_{3})>0
$
the orbit $K_{2}\cdot \pi_{1}(x)$ is proper biharmonic if and only if
\begin{align*}
(\cot \langle 2\alpha_{2}, H \rangle )^{2}
=
\frac{m_{2}+n_{3}  \pm \sqrt{(m_{2}+n_{3})^{2}-(2m_{1}+m_{2})(m_{2}+2n_{3}) }  }{(2m_{1}+m_{2})}
\end{align*}
holds.
In this case, there exist exactly
two proper biharmonic orbits.}
\item {\em When $(m_{2}+n_{3})^{2}-(2m_{1}+m_{2})(m_{2}+2n_{3})<0$,
if the orbit $K_{2}\cdot \pi_{1}(x)$ is biharmonic, then it is harmonic.}

\item {\em When $(m_{2}+n_{3})^{2}-(2m_{1}+m_{2})(m_{2}+2n_{3})=0$,
there exists unique proper biharmonic orbit.}
\end{itemize}

\vspace{12pt}
(3)\ When $H \in P_{0}^{\{ \alpha_{1}, \alpha_{2}\} }=\{tH_{\alpha_{1}}+(1-t)H_{\alpha_{2}} \mid 0<t<1\}$, 
we have $\Sigma^{+}_{H}=\emptyset, W^{+}_{H}=\{\tilde{\alpha}=2\alpha_{1}+2\alpha_{2}\}$.
We set $\vartheta=\langle 2\alpha_{1}, H\rangle $.
Then $\langle 2\alpha_{2}, H\rangle =(\pi/2)- \vartheta$.
By Theorem \ref{Theorem3.2}, 
we have 
\begin{align*}
\tau_{H}=&
2\{
(2m_{2}+n_{3})\cot \vartheta -m_{1}\tan \vartheta 
\}\alpha_{2}.
\end{align*}
Hence we have 
$\tau_{H}=0$ if and only if 
$$
(\cot \vartheta)^{2}
=\frac{m_{1}}{2m_{2}+n_{3}}.
$$
By Theorem \ref{thm: Bih. orbits of commutative Hermann actions}, the orbit $K_{2}\cdot \pi_{1}(x)$ is biharmonic if and only if 
\begin{align*}
0=&
4\langle \tau_{H}, \alpha_{1}\rangle
\{
(2m_{2}+n_{3})(1-(\cot \vartheta )^{2})
+m_{1}(1-(\tan \vartheta )^{2})
-(2m_{2}+m_{1})
\}\alpha_{2}
\end{align*}
Therefore, 
$K_{2}\cdot \pi_{1}(x)$ is biharmonic if and only if 
$\tau_{H}=0$ or 
\begin{align*} 
0=(2m_{2}+n_{3})(1-(\cot \vartheta )^{2})
+m_{1}(1-(\tan \vartheta )^{2})
-(2m_{2}+m_{1})
\end{align*}
holds.
The last equation is equivalent to 
\begin{align*} 
\{ (2m_{2}+n_{3})(\cot\vartheta )^{2}- m_{1}\} 
((\cot \vartheta )^{2}-1)=-(2m_{2}+m_{1})(\cot (2\vartheta) )^{2}.
\end{align*}
Since $2m_{2}+m_{1}>0$, the solutions of the equation are not harmonic.
\begin{itemize}
\item {\em When 
$
n_{3}^{2}-4(2m_{2}+n_{3})m_{1}>0,
$
the orbit $K_{2}\cdot \pi_{1}(x)$ is proper biharmonic if and only if
\begin{align*}
(\cot \vartheta)^{2}
=
\frac{n_{3} \pm \sqrt{n_{3}^{2}-4(2m_{2}+n_{3})m_{1} }  }{2(2m_{2}+n_{3})}
\end{align*}
holds.
In this case, there exist exactly
two proper biharmonic orbits.}

\item {\em When $n_{3}^{2}-4(2m_{2}+n_{3})m_{1}<0$,
if the orbit $K_{2}\cdot \pi_{1}(x)$ is biharmonic, then it is harmonic.}

\item {\em When $n_{3}^{2}-4(2m_{2}+n_{3})m_{1}=0$,
there exists unique proper biharmonic orbit.}
\end{itemize}

\vspace{12pt}
\subsection{Type $\rm{III\mathchar`-A}_{2}$}
We set 
$$
\mathfrak{a}=\{x_{1}e_{1}+x_{2}e_{2}+x_{3}e_{3} \mid x_{i}\in \mathbb{R},\ x_{1}+x_{2}+x_{3}=0\},
$$
and 
\begin{align*}
&\Sigma^{+}=W^{+}=\{e_{1}-e_{2},\ e_{2}-e_{3},\ e_{1}-e_{3}\}, \\
&\Pi=\{\alpha_{1}=e_{1}-e_{2}, \alpha_{2}=e_{2}-e_{3}\},\ \tilde{\alpha}=\alpha_{1}+\alpha_{2},\\
&m:=m(\lambda )= n(\lambda ) \quad (\lambda \in \tilde{\Sigma} ).
\end{align*}

\vspace{12pt}
(1)\ When $H \in P_{0}^{\{ \alpha_{1}, \tilde{\alpha}\} }=\{tH_{\alpha_{1}} \mid 0<t<1\}$,
we have $\Sigma^{+}_{H}=\{\alpha_{2} \}, W^{+}_{H}=\emptyset$.
By Theorem \ref{Theorem3.2},
we have
\begin{align*}
\tau_{H}
=&m\{-\cot \langle \alpha_{1} , H \rangle +\tan \langle \alpha_{1} , H \rangle \}(2\alpha_{1}+\alpha_{2}).
\end{align*} 
Hence we have that
$\tau_{H}=0$ if and only if $\langle \alpha_{1} , H \rangle=\pi/4$.
By Theorem \ref{thm: Bih. orbits of commutative Hermann actions}, the orbit $K_{2}\cdot \pi_{1}(x)$ is biharmonic if and only if 
\begin{align*}
0
=&-m\langle \tau_{H}, \alpha_{1} \rangle (\cot \langle \alpha_{1}, H\rangle -\tan \langle \alpha_{1}, H\rangle)^{2}(2\alpha_{1}+\alpha_{2}).
\end{align*}
Hence, the orbit $K_{2}\cdot \pi_{1}(x)$ is biharmonic if and only if 
$\langle \alpha_{1} , H \rangle=\pi/4$.
{\em Therefore, if the orbit $K_{2}\cdot \pi_{1}(x)$ is biharmonic, then it is harmonic.}

\vspace{12pt}

(2)\ When $H \in P_{0}^{\{ \alpha_{2}, \tilde{\alpha}\} }=\{tH_{\alpha_{2}} \mid 0<t<1\}$,
we have $\Sigma^{+}_{H}=\{\alpha_{1} \}, W^{+}_{H}=\emptyset$.
By the same calculation as (1), we have that 
the orbit $K_{2}\cdot \pi_{1}(x)$ is biharmonic if and only if 
$\langle \alpha_{2} , H \rangle=\pi/4$
and {\em if the orbit is biharmonic, then it is harmonic.}

\vspace{12pt}
(3)\ When $H \in P_{0}^{\{ \alpha_{1}, \alpha_{2}\} }=\{tH_{\alpha_{1}}+(1-t)H_{\alpha_{2}} \mid 0<t<1\}$, 
we have $\Sigma^{+}_{H}=\emptyset, W^{+}_{H}=\{ \alpha_{1}+\alpha_{2}\}$.
By the same calculation as (1), we have that 
the orbit $K_{2}\cdot \pi_{1}(x)$ is biharmonic if and only if 
$\langle \alpha_{1} , H \rangle=\pi/4$.
{\em If the orbit is biharmonic, then it is harmonic.}
\subsection{Type $\rm{III\mathchar`-B}_{2}$ and $\rm{III\mathchar`-C}_{2}$}
We set 
\begin{align*}
\Sigma^{+}&=\{ e_{1}\pm e_{2}, e_{1}, e_{2}\}, \
W^{+}=\{ e_{1}\pm e_{2}, e_{1}, e_{2}\}, \\
\Pi& =\{\alpha_{1}=e_{1}-e_{2}, \alpha_{2}= e_{2}\},\
\tilde{\alpha}=\alpha_{1}+2\alpha_{2}=e_{1}+e_{2}
\end{align*}
and 
\begin{align*}
m_{1}=m(e_{1}),\ m_{2}=m(e_{1}+e_{2}),\
n_{1}=n(e_{1}),\ n_{2}=n(e_{1}+e_{2}).
\end{align*}
Since 
$e_{1}\in \Sigma \cap W,\ e_{1}+e_{2}\in W$ and 
$2\langle e_{1}, e_{1}+e_{2}\rangle / \langle e_{1}+e_{2}, e_{1}+e_{2}\rangle=1$ is odd, 
by definition of multiplicities, we have 
$m_{1}=m(e_{1})=n(e_{1})=n_{1}$.


\vspace{12pt}
(1)\ When $H \in P_{0}^{\{ \alpha_{1}, \tilde{\alpha}\} }=\{tH_{\alpha_{1}} \mid 0<t<1\}$,
we have $\Sigma^{+}_{H}=\{\alpha_{2}\}, W^{+}_{H}=\emptyset$.
By Theorem \ref{Theorem3.2}, 
we have 
\begin{align*}
\tau_{H}=&
-(2m_{2}+m_{1}) \cot\langle \alpha_{1}, H \rangle (\alpha_{1}+\alpha_{2})+(2n_{2}+m_{1}) \tan\langle \alpha_{1}, H \rangle (\alpha_{1}+\alpha_{2}).
\end{align*}
Hence we have 
$\tau_{H}=0$ if and only if 
$$
(\cot \langle 2\alpha_{1} , H \rangle)^{2}=\frac{2m_{2}+m_{1}}{2n_{2}+m_{1}}.
$$
By Theorem \ref{thm: Bih. orbits of commutative Hermann actions}, the orbit $K_{2}\cdot \pi_{1}(x)$ is biharmonic if and only if 
\begin{align*}
0=&
\langle \tau_{H}, \alpha_{1} \rangle  
\{
(2m_{2}+m_{1})(1-(\cot \langle \alpha_{1}, H\rangle )^{2})\\
&+(2n_{2}+m_{1})(1-(\tan \langle \alpha_{1}, H\rangle )^{2})
\}(\alpha_{1}+2\alpha_{2}).
\end{align*}
Therefore, 
$K_{2}\cdot \pi_{1}(x)$ is biharmonic if and only if 
$\tau_{H}=0$ or 
\begin{align*} 
0=(2m_{2}+m_{1})(1-(\cot \langle \alpha_{1}, H\rangle )^{2})
+(2n_{2}+m_{1})(1-(\tan \langle \alpha_{1}, H\rangle )^{2})  
\end{align*}
holds.
The last equation is equivalent to 
\begin{align*} 
&\{ (2m_{2}+m_{1}) (\cot \langle \alpha_{1} , H\rangle )^{2}-(2n_{2}+m_{1})\} 
+((\cot \langle \alpha_{1} , H\rangle )^{2}-1)=0.
\end{align*}
\begin{itemize}
\item {\em When $m_{2}\neq n_{2}$, 
the orbit $K_{2}\cdot \pi_{1}(x)$ is proper biharmonic if and only if
\begin{align*}
(\cot \langle 2\alpha_{1}, H \rangle )^{2}
=1\qquad (\text{i.e.}\  \langle 2\alpha_{1}, H\rangle=(\pi/4) )
\end{align*}
holds.
In this case, there exists a unique proper biharmonic orbit.}
\item {\em When $m_{2}= n_{2}$, 
if the orbit $K_{2}\cdot \pi_{1}(x)$ is biharmonic, then it is harmonic.}
\end{itemize}

\vspace{12pt}
(2)\ When $H \in P_{0}^{\{ \alpha_{2}, \tilde{\alpha}\} }=\{tH_{\alpha_{2}} \mid 0<t<1\}$,
we have $\Sigma^{+}_{H}=\{\alpha_{1}\}, W^{+}_{H}=\emptyset$.
By Theorem \ref{Theorem3.2}, 
we have 
\begin{align*}
\tau_{H}=&
\{
-(2m_{1}+m_{2})\cot\langle 2\alpha_{2}, H \rangle
+n_{2}\tan\langle 2\alpha_{2}, H \rangle
\}(\alpha_{1}+2\alpha_{2}).
\end{align*}
Hence we have 
$\tau_{H}=0$ if and only if 
$$
(\cot \langle 2\alpha_{2} , H \rangle)^{2}
=\frac{n_{2}}{2m_{1}+m_{2}}.
$$
By Theorem \ref{thm: Bih. orbits of commutative Hermann actions}, the orbit $K_{2}\cdot \pi_{1}(x)$ is biharmonic if and only if 
\begin{align*}
0=&
\langle \tau_{H}, 2\alpha_{2} \rangle
\{
(2m_{1}+m_{2})(1-(\cot \langle 2\alpha_{2}, H\rangle )^{2})\\
&+n_{2}(1-(\tan \langle 2\alpha_{2}, H\rangle )^{2})-2m_{1}
\}(\alpha_{1}+2\alpha_{2}).
\end{align*}
Therefore, 
$K_{2}\cdot \pi_{1}(x)$ is biharmonic if and only if 
$\tau_{H}=0$ or 
\begin{align*} 
0=(2m_{1}+m_{2})(1-(\cot \langle 2\alpha_{2}, H\rangle )^{2})
+n_{2}(1-(\tan \langle 2\alpha_{2}, H\rangle )^{2})-2m_{1}  
\end{align*}
holds.
The last equation is equivalent to 
\begin{align*} 
\big( (2m_{1}+m_{2}) (\cot \langle 2\alpha_{2} , H\rangle )^{2}-n_{2}\big) ((\cot \langle 2\alpha_{2} , H\rangle )^{2}-1)
=  -2m_{1}(\cot \langle 2\alpha_{2} , H\rangle )^{2}.
\end{align*}
Since $2m_{1}>0$, the solutions of the equation are not harmonic.
\begin{itemize}
\item {\em When 
$
(m_{2}+n_{2})^{2}-4(2m_{1}+m_{2})n_{2}>0,
$
the orbit $K_{2}\cdot \pi_{1}(x)$ is proper biharmonic if and only if
\begin{align*}
(\cot  \langle 2\alpha_{2}, H \rangle )^{2}
=
\frac{m_{2}+n_{2}  \pm \sqrt{(m_{2}+n_{2})^{2}-4(2m_{1}+m_{2})n_{2} }  }{2(2m_{1}+m_{2})}
\end{align*}
holds.
In this case, there exist exactly
two proper biharmonic orbits.}

\item {\em When $(m_{2}+n_{2})^{2}-4(2m_{1}+m_{2})n_{2}<0$,
if the orbit $K_{2}\cdot \pi_{1}(x)$ is biharmonic, then it is harmonic.}

\item {\em When $(m_{2}+n_{2})^{2}-4(2m_{1}+m_{2})n_{2}=0$,
there exists unique proper biharmonic orbit.}
\end{itemize}

\vspace{12pt}
(3)\ When $H \in P_{0}^{\{ \alpha_{1}, \alpha_{2}\} }=\{tH_{\alpha_{1}}+(1-t)H_{\alpha_{2}} \mid 0<t<1\}$, 
we have $\Sigma^{+}_{H}=\emptyset, W^{+}_{H}=\{\tilde{\alpha}=\alpha_{1}+2\alpha_{2}\}$.
We set $\vartheta=\langle \alpha_{1}, H\rangle $.
Then $\langle 2\alpha_{2}, H\rangle =(\pi/2)- \vartheta$.
By Theorem \ref{Theorem3.2}, 
we have 
\begin{align*}
\tau_{H}=&
\{-m_{2}\cot (\vartheta) +(n_{2}+2m_{1})\tan (\vartheta)\}\alpha_{1}
.
\end{align*}
Hence we have 
$\tau_{H}=0$ if and only if 
$$
(\cot \vartheta)^{2}
=\frac{n_{2}+2m_{1}}{m_{2}}.
$$
By Theorem \ref{thm: Bih. orbits of commutative Hermann actions}, the orbit $K_{2}\cdot \pi_{1}(x)$ is biharmonic if and only if 
\begin{align*}
0=&
\langle \tau_{H}, \alpha_{1} \rangle
\{
m_{2}(1-(\cot \langle \alpha_{1}, H\rangle  )^{2})
+(n_{2}+2m_{1})(1-(\tan \langle \alpha_{1}, H\rangle  )^{2})
-2m_{1}
\}\alpha_{1}.
\end{align*}
Therefore, 
$K_{2}\cdot \pi_{1}(x)$ is biharmonic if and only if 
$\tau_{H}=0$ or 
\begin{align*} 
0=m_{2}(1-(\cot \langle \alpha_{1}, H\rangle  )^{2})
+(n_{2}+2m_{1})(1-(\tan \langle \alpha_{1}, H\rangle  )^{2})
-2m_{1}
\end{align*}
holds.
The last equation is equivalent to 
\begin{align*} 
\{ m_{2}(\cot\vartheta )^{2}- (n_{2}+2m_{1})\} 
((\cot (\vartheta) )^{2}-1)=-2m_{1}(\cot (\vartheta) )^{2}.
\end{align*}
Since $m_{1}>0$, the solutions of the equation are not harmonic.
\begin{itemize}
\item {\em When 
$
(m_{2}+n_{2})^{2}-4m_{2}(n_{2}+2m_{1})>0,
$
the orbit $K_{2}\cdot \pi_{1}(x)$ is proper biharmonic if and only if
\begin{align*}
(\cot \vartheta)^{2}
=
\frac{(m_{2}+n_{2}) \pm \sqrt{(m_{2}+n_{2})^{2}-4m_{2}(n_{2}+2m_{1}) }  }{2m_{2}}
\end{align*}
holds.
In this case, there exist exactly
two proper biharmonic orbits.}

\item {\em When $(m_{2}+n_{2})^{2}-4m_{2}(n_{2}+2m_{1})<0$,
if the orbit $K_{2}\cdot \pi_{1}(x)$ is biharmonic, then it is harmonic.}

\item {\em When $(m_{2}+n_{2})^{2}-4m_{2}(n_{2}+2m_{1})=0$,
there exists unique proper biharmonic orbit.}
\end{itemize}
\subsection{Type $\rm{III\mathchar`-BC}_{2}$}
We set 
\begin{align*}
\Sigma^{+}&=W^{+}=\{ e_{1}\pm e_{2}, e_{1}, e_{2}, 2e_{1}, 2e_{2}\},\\  
\Pi& =\{\alpha_{1}=e_{1}-e_{2}, \alpha_{2}= e_{2}\},\
\tilde{\alpha}=2\alpha_{1}+2\alpha_{2}=2e_{1}
\end{align*} 
and 
\begin{align*}
m_{1}&=m(e_{1}),\ m_{2}=m(e_{1}+e_{2}),\ m_{3}=(2e_{1}),\\ 
n_{1}&=n(e_{1}),\ n_{2}=n(e_{1}+e_{2}),\ n_{3}=(2e_{1}).
\end{align*}
Since 
$e_{1}, e_{1}+e_{2} \in \Sigma \cap W,\ 2e_{1}\in W$ and 
$2\langle e_{1}, 2e_{1}\rangle / \langle 2e_{1}, 2e_{1}\rangle=1$ and
$2\langle e_{1}+e_{2}, 2e_{1}\rangle / \langle 2e_{1}, 2e_{1}\rangle=1$ 
 are odd, 
by definition of multiplicities, we have 
$m_{1}=m(e_{1})=n(e_{1})=n_{1}, m_{2}=m(e_{1}+e_{2})=n(e_{1}+e_{2})=n_{2}$.

\vspace{12pt}
(1)\ When $H \in P_{0}^{\{ \alpha_{1}, \tilde{\alpha}\} }=\{tH_{\alpha_{1}} \mid 0<t<1\}$,
we have $\Sigma^{+}_{H}=\{\alpha_{2}, 2\alpha_{2}\}, W^{+}_{H}=\emptyset$.
By Theorem \ref{Theorem3.2}, 
we have 
\begin{align*}
\tau_{H}=&
2\{ -(2m_{2}+m_{1}+m_{3})\cot\langle 2\alpha_{1}, H \rangle
+n_{3}\tan\langle 2\alpha_{1}, H \rangle
\}(\alpha_{1}+\alpha_{2}).
\end{align*}
Hence we have 
$\tau_{H}=0$ if and only if 
$$
(\cot \langle 2\alpha_{1} , H \rangle)^{2}=\frac{n_{3}}{m_{1}+2m_{2}+m_{3}}.
$$
By Theorem \ref{thm: Bih. orbits of commutative Hermann actions}, the orbit $K_{2}\cdot \pi_{1}(x)$ is biharmonic if and only if 
\begin{align*}
0=&
4\langle \tau_{H}, \alpha_{1} \rangle 
\{
(2m_{2}+m_{1}+m_{3})(1-(\cot \langle 2\alpha_{1} , H\rangle )^{2})\\
&+n_{3}(1-(\tan \langle 2\alpha_{1} , H\rangle )^{2})
-(2m_{2}+m_{1})
\}(\alpha_{1}+\alpha_{2})
.
\end{align*}
Therefore, 
$K_{2}\cdot \pi_{1}(x)$ is biharmonic if and only if 
$\tau_{H}=0$ or 
\begin{align*} 
0=(2m_{2}+m_{1}+m_{3})(1-(\cot \langle 2\alpha_{1} , H\rangle )^{2})
+n_{3}(1-(\tan \langle 2\alpha_{1}, H\rangle )^{2})
-(2m_{2}+m_{1})  
\end{align*}
holds.
The last equation is equivalent to 
\begin{align*} 
&\{ (2m_{2}+m_{1}+m_{3}) (\cot \langle 2\alpha_{1} , H\rangle )^{2}-n_{3}\} ((\cot \langle 2\alpha_{1} , H\rangle )^{2}-1)\\
=&-(2m_{2}+m_{1})(\cot \langle 2\alpha_{1} , H\rangle )^{2}.
\end{align*}
Since $(2m_{2}+m_{1})>0$, the solutions of the equation are not harmonic.
\begin{itemize}
\item {\em When $(m_{3}+n_{3})^{2}-4(2m_{2}+m_{1}+m_{3})n_{3}>0$, 
the orbit $K_{2}\cdot \pi_{1}(x)$ is proper biharmonic if and only if
\begin{align*}
(\cot  \langle 2\alpha_{1}, H \rangle )^{2}
=\frac{m_{3}+n_{3} \pm \sqrt{(m_{3}+n_{3})^{2}-4(2m_{2}+m_{1}+m_{3})n_{3}} }{2(2m_{2}+m_{1}+m_{3})}
\end{align*}
holds.
In this case, there exist exactly
two proper biharmonic orbits.}
\item {\em When $(m_{3}+n_{3})^{2}-4(2m_{2}+m_{1}+m_{3})n_{3}<0$,
if the orbit {$K_{2}\cdot \pi_{1}(x)$} is biharmonic, then it is harmonic.}

\item {\em When $(m_{3}+n_{3})^{2}-4(2m_{2}+m_{1}+m_{3})n_{3}=0$,
there exists unique proper biharmonic orbit.}
\end{itemize}

\vspace{12pt}
(2)\ When $H \in P_{0}^{\{ \alpha_{2}, \tilde{\alpha}\} }=\{tH_{\alpha_{2}} \mid 0<t<1\}$,
we have $\Sigma^{+}_{H}=\{\alpha_{1}\}, W^{+}_{H}=\emptyset$.
By Theorem \ref{Theorem3.2}, 
we have 
\begin{align*}
\tau_{H}=&
\{
-(2m_{1}+m_{2}+2m_{3})\cot\langle 2\alpha_{2}, H \rangle
+(m_{2}+2n_{2})\tan\langle 2\alpha_{2}, H \rangle
\}(\alpha_{1}+2\alpha_{2}).
\end{align*}
Hence we have 
$\tau_{H}=0$ if and only if 
$$
(\cot \langle 2\alpha_{2} , H \rangle)^{2}
=\frac{m_{2}+2n_{2}}{2m_{1}+m_{2}+2m_{3}}.
$$
By Theorem \ref{thm: Bih. orbits of commutative Hermann actions}, the orbit $K_{2}\cdot \pi_{1}(x)$ is biharmonic if and only if 
\begin{align*}
0=&
\langle \tau_{H}, 2\alpha_{2} \rangle
\{
 (2m_{1}+m_{2}+2m_{3})(1-(\cot \langle 2\alpha_{2}, H\rangle )^{2})\\
&+(n_{2}+2n_{3})(1-(\tan \langle 2\alpha_{2}, H\rangle )^{2})
-2m_{1}
\}(\alpha_{1}+2\alpha_{2}).
\end{align*}
Therefore, 
$K_{2}\cdot \pi_{1}(x)$ is biharmonic if and only if 
$\tau_{H}=0$ or 
\begin{align*} 
0=& (2m_{1}+m_{2}+2m_{3})(1-(\cot \langle 2\alpha_{2}, H\rangle )^{2})\\
&+(n_{2}+2n_{3})(1-(\tan \langle 2\alpha_{2}, H\rangle )^{2})
-2m_{1} 
\end{align*}
holds.
The last equation is equivalent to 
\begin{align*} 
&\big( (2m_{1}+m_{2}+2m_{3}) (\cot \langle 2\alpha_{2} , H\rangle )^{2}-(m_{2}+2n_{3})\big) ((\cot \langle 2\alpha_{2} , H\rangle )^{2}-1)\\
=&  -2m_{1}(\cot \langle 2\alpha_{2} , H\rangle )^{2}.
\end{align*}
Since $2m_{1}>0$, the solutions of the equation are not harmonic.
\begin{itemize}
\item {\em When 
$
(m_{2}+m_{3}+n_{3})^{2}-(2m_{1}+m_{2}+2m_{3})(m_{2}+2n_{2})>0
$
the orbit $K_{2}\cdot \pi_{1}(x)$ is proper biharmonic if and only if
\begin{align*}
&(\cot  \langle 2\alpha_{2}, H \rangle )^{2}\\
=&
\frac{m_{2}+m_{3}+n_{3}  \pm \sqrt{(m_{2}+m_{3}+n_{3})^{2}-(2m_{1}+m_{2}+2m_{3})(m_{2}+2n_{2}) }  }{2m_{1}+m_{2}+2m_{3}}
\end{align*}
holds.
In this case, there exist exactly
two proper biharmonic orbits.}

\item {\em When $(m_{2}+m_{3}+n_{3})^{2}-(2m_{1}+m_{2}+2m_{3})(m_{2}+2n_{2})<0$,
if the orbit {$K_{2}\cdot \pi_{1}(x)$} is biharmonic, then it is harmonic.}

\item {\em When $(m_{2}+m_{3}+n_{3})^{2}-(2m_{1}+m_{2}+2m_{3})(m_{2}+2n_{2})=0$,
there exists unique proper biharmonic orbit.}
\end{itemize}

\vspace{12pt}
(3)\ When $H \in P_{0}^{\{ \alpha_{1}, \alpha_{2}\} }=\{tH_{\alpha_{1}}+(1-t)H_{\alpha_{2}} \mid 0<t<1\}$, 
we have $\Sigma^{+}_{H}=\emptyset, W^{+}_{H}=\{\tilde{\alpha}=2\alpha_{1}+2\alpha_{2}=2e_{1}\}$.
We set $\vartheta=\langle 2\alpha_{1}, H\rangle $.
Then $\langle 2\alpha_{2}, H\rangle =(\pi/2)- \vartheta$.
By Theorem \ref{Theorem3.2}, 
we have 
\begin{align*}
\tau_{H}=&
\{
(4m_{2}+2n_{3})\cot \vartheta
-(2m_{1}+2m_{3})\tan((\pi/2)-\vartheta)
\}\alpha_{2}.
\end{align*}
Hence we have 
$\tau_{H}=0$ if and only if 
$$
(\cot \vartheta)^{2}
=\frac{m_{1}+m_{3}}{2m_{2}+n_{3}}.
$$
By Theorem \ref{thm: Bih. orbits of commutative Hermann actions}, the orbit $K_{2}\cdot \pi_{1}(x)$ is biharmonic if and only if 
\begin{align*}
0=&
4\langle \tau_{H}, \alpha_{2}\rangle
\{
(2m_{2}+n_{3}) (1-(\cot \vartheta  )^{2})\\ 
&+(m_{1}+m_{3}) (1-(\tan \vartheta )^{2})  
-(2m_{2}+m_{1})
\} \alpha_{2}
.
\end{align*}
Therefore, 
$K_{2}\cdot \pi_{1}(x)$ is biharmonic if and only if 
$\tau_{H}=0$ or 
\begin{align*} 
0=(2m_{2}+n_{3}) (1-(\cot \vartheta  )^{2})
+(m_{1}+m_{3}) (1-(\tan \vartheta )^{2})  
-(2m_{2}+m_{1})
\end{align*}
holds.
The last equation is equivalent to 
\begin{align*} 
\{ (2m_{2}+n_{3})(\cot\vartheta )^{2}- (m_{1}+m_{3})\} 
((\cot \vartheta )^{2}-1)=-(m_{1}+2m_{2})(\cot \vartheta )^{2}.
\end{align*}
Since $m_{1}+2m_{2}>0$, the solutions of the equation are not harmonic.
\begin{itemize}
\item {\em When 
$
(m_{3}+n_{3})^{2}-4(2m_{2}+n_{3})(m_{1}+m_{3})>0,
$
the orbit $K_{2}\cdot \pi_{1}(x)$ is proper biharmonic if and only if
\begin{align*}
(\cot \vartheta)^{2}
=
\frac{(m_{3}+n_{3}) \pm \sqrt{(m_{3}+n_{3})^{2}-4(2m_{2}+n_{3})(m_{1}+m_{3}) }  }{2(2m_{2}+n_{3})}
\end{align*}
holds.
In this case, there exist exactly
two proper biharmonic orbits.}

\item {\em When $(m_{3}+n_{3})^{2}-4(2m_{2}+n_{3})(m_{1}+m_{3})<0$,
if the orbit $K_{2}\cdot \pi_{1}(x)$ is biharmonic, then it is harmonic.}

\item {\em When $(m_{3}+n_{3})^{2}-4(2m_{2}+n_{3})(m_{1}+m_{3})=0$,
there exists unique proper biharmonic orbit.}\\
\end{itemize}
\subsection{Type $\rm{III\mathchar`-G}_{2}$}
We set
\begin{align*}
\Sigma^{+}&=W^{+}=
\{\alpha_{1}, \alpha_{2}, \alpha_{1}+\alpha_{2}, 2\alpha_{1}+\alpha_{2}, 3\alpha_{1}+\alpha_{2}, 3\alpha_{1}+2\alpha_{2} \},\\
&\langle \alpha_{1}, \alpha_{1} \rangle=1, \ \langle \alpha_{1}, \alpha_{2} \rangle =-\frac{3}{2}, \ \langle \alpha_{2}, \alpha_{2} \rangle=3,\
\tilde{\alpha}=3\alpha_{1}+2\alpha_{2},
\end{align*}
and 
$$
m_{1}=m(\alpha_{1} ),\ m_{2}=m(\alpha_{2}).
$$
(1)\ When $H \in P_{0}^{\{ \alpha_{1}, \tilde{\alpha}\} }=\{tH_{\alpha_{1}} \mid 0<t<1\}$, 
we have $\Sigma^{+}_{H}=\{ \alpha_{2}\}, W^{+}_{H}=\emptyset$.
By Theorem \ref{Theorem3.2}, 
we have 
\begin{align*}
\tau_{H}=&
2\big[
-m_{1}\{
\cot\langle 2\alpha_{1}, H \rangle
+\cot\langle 4\alpha_{1}, H \rangle
\}
-3m_{2}
\cot\langle 6\alpha_{1}, H \rangle 
\big] (2\alpha_{1}+\alpha_{2}).
\end{align*}
Hence we have 
$\tau_{H}=0$ if and only if 
$$
-m_{1}\{
\cot\langle 2\alpha_{1}, H \rangle
+\cot\langle 4\alpha_{1}, H \rangle
\}
-3m_{2}
\cot\langle 6\alpha_{1}, H \rangle=0.
$$
By Theorem \ref{thm: Bih. orbits of commutative Hermann actions}, the orbit $K_{2}\cdot \pi_{1}(x)$ is biharmonic if and only if 
\begin{align*}
0=&
-4\langle \tau_{H}, \alpha_{1}\rangle
\{
m_{1}(\cot\langle 2\alpha_{1}, H \rangle )^{2}
+2m_{1}(\cot\langle 4\alpha_{1}, H \rangle )^{2}\\
&+9m_{2}(\cot\langle 6\alpha_{1}, H \rangle )^{2}
\}(2\alpha_{1}+\alpha_{2}).
\end{align*}
Therefore, 
$K_{2}\cdot \pi_{1}(x)$ is biharmonic if and only if 
$\tau_{H}=0$ or 
\begin{align*} 
0=m_{1}(\cot\langle 2\alpha_{1}, H \rangle )^{2}
+2m_{1}(\cot\langle 4\alpha_{1}, H \rangle )^{2}
+9m_{2}(\cot\langle 6\alpha_{1}, H \rangle )^{2}  
\end{align*}
holds.
Clearly, 
\begin{align*} 
m_{1}(\cot\langle 2\alpha_{1}, H \rangle )^{2}
+2m_{1}(\cot\langle 4\alpha_{1}, H \rangle )^{2}
+9m_{2}(\cot\langle 6\alpha_{1}, H \rangle )^{2} >0 
\end{align*}
for $0<t<1$.
{\em Therefore, if the orbit $K_{2}\cdot \pi_{1}(x)$ is biharmonic, 
then it is harmonic.}\\

(2)\ When $H \in P_{0}^{\{ \alpha_{2}, \tilde{\alpha}\} }=\{tH_{\alpha_{2}} \mid 0<t<1\}$,
we have $\Sigma^{+}_{H}=\{ \alpha_{1}\}, W^{+}_{H}=\emptyset$.
By Theorem \ref{Theorem3.2}, 
we have 
\begin{align*}
\tau_{H}=&
-2\big[
(m_{1}+m_{2})
\cot\langle 2\alpha_{1}, H \rangle
+m_{2}
\cot\langle 4\alpha_{1}, H \rangle 
\big] (3\alpha_{1}+2\alpha_{2}).
\end{align*}
Hence we have 
$\tau_{H}=0$ if and only if 
$$
(m_{1}+m_{2})
\cot\langle 2\alpha_{1}, H \rangle
+m_{2}
\cot\langle 4\alpha_{1}, H \rangle 
=0.
$$
By Theorem \ref{thm: Bih. orbits of commutative Hermann actions}, the orbit $K_{2}\cdot \pi_{1}(x)$ is biharmonic if and only if 
\begin{align*}
0=&
-\langle \tau_{H}, \alpha_{2}\rangle
[
(m_{1}+m_{2})(\cot\langle \alpha_{2}, H \rangle -\tan\langle \alpha_{2}, H \rangle )^{2}\\
&+2m_{2}(\cot\langle 2\alpha_{1}, H \rangle -\tan\langle 2\alpha_{1}, H \rangle )^{2}
]
(2\alpha_{1}+\alpha_{2}).
\end{align*}
Therefore, 
$K_{2}\cdot \pi_{1}(x)$ is biharmonic if and only if 
$\tau_{H}=0$ or 
\begin{align*} 
0=(m_{1}+m_{2})(\cot\langle 2\alpha_{2}, H \rangle )^{2}
+2m_{2}(\cot\langle 4\alpha_{1}, H \rangle )^{2} 
\end{align*}
holds.
Clearly, 
\begin{align*} 
(m_{1}+m_{2})(\cot\langle 2\alpha_{2}, H \rangle )^{2}
+2m_{2}(\cot\langle 4\alpha_{1}, H \rangle )^{2} >0 
\end{align*}
for $0<t<1$.
{\em Therefore, if the orbit $K_{2}\cdot \pi_{1}(x)$ is biharmonic, 
then it is harmonic.}\\

(3)\ When $H \in P_{0}^{\{ \alpha_{1}, \alpha_{2}\} }=\{tH_{\alpha_{1}}+(1-t)H_{\alpha_{2}} \mid 0<t<1\}$, 
we have $\Sigma^{+}_{H}= \emptyset, W^{+}_{H}=\{3\alpha_{1}+2\alpha_{2}\}$.
We set $\vartheta=\langle \alpha_{1}, H\rangle $.
Then $\langle 2\alpha_{2}, H\rangle =(\pi/2)-3\vartheta$ 
and 
$0<\vartheta<(\pi/6)$.
By Theorem \ref{Theorem3.2}, 
we have 
\begin{align*}
\tau_{H}=&
-2\big\{
m_{1}\cot(2\vartheta)\alpha_{1}
-m_{2}\tan(3\vartheta)(3\alpha_{1})
-m_{1}\tan\vartheta \alpha_{1}
\big\}.
\end{align*}
Hence we have 
$\tau_{H}=0$ if and only if 
$$
m_{1}\cot(2\vartheta)
-3m_{2}\tan(3\vartheta)
-m_{1}\tan\vartheta=0.
$$
By Theorem \ref{thm: Bih. orbits of commutative Hermann actions}, the orbit $K_{2}\cdot \pi_{1}(x)$ is biharmonic if and only if 
\begin{align*}
0=&
-\langle \tau_{H}, \alpha_{1}\rangle 
\left\{
m_{1}(\cot (2\vartheta))^{2}
+ \frac{9}{2} m_{2}(\tan(3\vartheta))^{2} 
+ \frac{1}{2}m_{1}(\tan\vartheta)^{2}
\right\}\alpha_{1}.
\end{align*}
Therefore, 
$K_{2}\cdot \pi_{1}(x)$ is biharmonic if and only if 
$\tau_{H}=0$ or 
\begin{align*} 
0=m_{1}(\cot (2\vartheta))^{2}
+\frac{9}{2}m_{2}(\tan(3\vartheta))^{2} 
+\frac{1}{2}m_{1}(\tan\vartheta)^{2} 
\end{align*}
holds.
Clearly, 
\begin{align*} 
m_{1}(\cot (2\vartheta))^{2}
+\frac{9}{2}m_{2}(\tan(3\vartheta))^{2} 
+\frac{1}{2}m_{1}(\tan\vartheta)^{2}>0 
\end{align*}
for $0<t<1$.
{\em Therefore, if the orbit $K_{2}\cdot \pi_{1}(x)$ is biharmonic, 
then it is harmonic.}

\subsection{Tables of proper biharmonic orbits}
\label{Tables_of_proper_biharmonic_orbits_of_commutative_Hermann_actions}
By the above arguments, 
we classify all the proper biharmonic submanifolds which are singular orbits of commutative Hermann actions whose cohomogeneity is two.
The co-dimensions of these submanifolds are greater than two, since we consider singular orbits of cohomogeneity two actions. 
\begin{theorem}\label{Theorem:cohom2}
Let $(G, K_{1}, K_{2})$ be a compact symmetric triad
which satisfies $\theta_{1}=\theta_{2}$ or one of the conditions 
(A), (B) and (C) in Theorem \ref{Theorem:I2 AB}.
Assume that the $K_{2}$-action on $N_{1}=G/K_{1}$ is cohomogeneity two.
Then, 
all singular orbit types which are one parameter families in the orbit space are 
divided into one of the following three cases:
\begin{enumerate}
\item[(i)] There exists a unique proper biharmonic orbit.
\item[(ii)] There exist exactly two distinct proper biharmonic orbits. 
\item[(iii)] Any biharmonic orbit is harmonic.
\end{enumerate}
\end{theorem}

We list below all the results of the above computations.
In the following tables, 
the first column shows compact symmetric triads which induce Hermann actions;
the second column shows multiplicities of symmetric triads which are induced by compact symmetric triads in the first column; 
the third column shows a subset $\Delta$ in $\Pi \cup \{\delta \}$ or $\Pi \cup \{\tilde{\alpha} \}$ 
where $P_{0}^{\Delta}$ gives a singular orbit types which are one parameter families in the orbit space;
the fourth column represents the result (i), (ii) or (iii) in Theorem \ref{Theorem:cohom2} for the orbit type $P_{0}^{\Delta}$;
the fifth column shows the co-dimension of orbits $K_{2}\cdot \pi_{1}(x)$ in $N_{1}$ and $K_{1}\cdot \pi_{2}(x)$ in $N_{2}$. \vspace{12pt} \\
\noindent{\bf Isotropy actions\ $(\theta_{1}=\theta_{2})$}

{\bf Type $\rm{A}_{2}$}\\
\begin{tabular}{|c|c|c|c|c|} 
\hline
$(G,K_{1})$                                            &$m(\alpha)$ & $\Delta $                                    &Theorem~\ref{Theorem:cohom2}&$\mathrm{codim}$ \\ \hline \hline
$(\mathrm{SU}(3), \mathrm{SO}(3))$                     & $1$        & $\{\alpha_{1}, \delta \}$                    &(ii)   &$3$          \\ \cline{3-5}
                                                       &            & $\{\alpha_{2}, \delta \}$                    &(ii)   &$3$          \\ \cline{3-5}
                                                       &            & $\{\alpha_{1}, \alpha_{2} \}$                &(ii)   &$3$          \\ \hline
$(\mathrm{SU}(3)\times \mathrm{SU}(3), \mathrm{SU}(3))$& $2$        & $\{\alpha_{1}, \delta \}$                    &(ii)   &$4$          \\ \cline{3-5}
                                                       &            & $\{\alpha_{2}, \delta \}$                    &(ii)   &$4$          \\ \cline{3-5}
                                                       &            & $\{\alpha_{1}, \alpha_{2} \}$                &(ii)   &$4$          \\ \hline
$(\mathrm{SU}(6), \mathrm{Sp}(3))$                     & $4$        & $\{\alpha_{1}, \delta \}$                    &(ii)   &$6$          \\ \cline{3-5}
                                                       &            & $\{\alpha_{2}, \delta \}$                    &(ii)   &$6$          \\ \cline{3-5}
                                                       &            & $\{\alpha_{1}, \alpha_{2} \}$                &(ii)   &$6$          \\ \hline
$(E_{6}, F_{4})$                                       & $8$        & $\{\alpha_{1}, \delta \}$                    &(ii)   &$10$         \\ \cline{3-5}
                                                       &            & $\{\alpha_{2}, \delta \}$                    &(ii)   &$10$         \\ \cline{3-5}
                                                       &            & $\{\alpha_{1}, \alpha_{2} \}$                &(ii)   &$10$         \\ \hline
\end{tabular}
\\[12pt]

{\bf Type $\rm{B}_{2}$}\\
\begin{tabular}{|c|c|c|c|c|} 
\hline
$(G,K_{1})$                                      &$(m_{1},m_{2})$& $\Delta              $                    &Theorem~\ref{Theorem:cohom2}&$\mathrm{codim}$ \\ \hline \hline
$(\mathrm{SO}(3)\times \mathrm{SO}(3), \mathrm{SO}(3))$& $(2,2)$       & $\{\alpha_{1}, \delta \}$                    &(ii)   &$4$          \\ \cline{3-5}
                                                       &               & $\{\alpha_{2}, \delta \}$                    &(ii)   &$4$          \\ \cline{3-5}
                                                       &               & $\{\alpha_{1}, \alpha_{2} \}$                &(ii)   &$4$          \\ \hline
$(\mathrm{SO}(4+n), \mathrm{SO}(2)\times \mathrm{SO}(2+n))$& $(n,1)$   & $\{\alpha_{1}, \delta \}$                    &(ii)   &$n+2$          \\ \cline{3-5}
                                                       &               & $\{\alpha_{2}, \delta \}$                    &(ii)   &$3 $      \\ \cline{3-5}
                                                       &               & $\{\alpha_{1}, \alpha_{2} \}$                &(ii)   &$3$       \\ \hline
\end{tabular}
\\[12pt]

{\bf Type $\rm{C}_{2}$}\\
\begin{tabular}{|c|c|c|c|c|} 
\hline
$(G,K_{1})$                                      &$(m_{1},m_{2})$      & $\Delta              $                    &Theorem~\ref{Theorem:cohom2}&$\mathrm{codim}$ \\ \hline \hline
$(\mathrm{Sp}(2), \mathrm{U}(2))$                      & $(1,1)$       & $\{\alpha_{1}, \delta \}$                    &(ii)   &$3$          \\ \cline{3-5}
                                                       &               & $\{\alpha_{2}, \delta \}$                    &(ii)   &$3$          \\ \cline{3-5}
                                                       &               & $\{\alpha_{1}, \alpha_{2} \}$                &(ii)   &$3$          \\ \hline
$(\mathrm{Sp}(2)\times \mathrm{Sp}(2), \mathrm{Sp}(2))$& $(2,2)$       & $\{\alpha_{1}, \delta \}$                    &(ii)   &$4$          \\ \cline{3-5}
                                                       &               & $\{\alpha_{2}, \delta \}$                    &(ii)   &$4$         \\ \cline{3-5}
                                                       &               & $\{\alpha_{1}, \alpha_{2} \}$                &(ii)   &$4$         \\ \hline
$(\mathrm{Sp}(4), \mathrm{Sp}(2)\times \mathrm{Sp}(2))$& $(4,3)$       & $\{\alpha_{1}, \delta \}$                    &(ii)   &$5$          \\ \cline{3-5}
                                                       &               & $\{\alpha_{2}, \delta \}$                    &(ii)   &$6$         \\ \cline{3-5}
                                                       &               & $\{\alpha_{1}, \alpha_{2} \}$                &(ii)   &$5$         \\ \hline
$(\mathrm{SU}(4), \mathrm{S}(\mathrm{U}(2)\times\mathrm{U}(2)))$& $(2,1)$& $\{\alpha_{1}, \delta \}$                  &(ii)   &$3$          \\ \cline{3-5}
                                                       &               & $\{\alpha_{2}, \delta \}$                    &(ii)   &$4$         \\ \cline{3-5}
                                                       &               & $\{\alpha_{1}, \alpha_{2} \}$                &(ii)   &$3$         \\ \hline
$(\mathrm{SO}(8), \mathrm{U}(4))$                      & $(4,1)$       & $\{\alpha_{1}, \delta \}$                    &(ii)   &$3$          \\ \cline{3-5}
                                                       &               & $\{\alpha_{2}, \delta \}$                    &(ii)   &$6$         \\ \cline{3-5}
                                                       &               & $\{\alpha_{1}, \alpha_{2} \}$                &(ii)   &$3$         \\ \hline
\end{tabular}
\\[12pt]

{\bf Type $\rm{BC}_{2}$}\\
\begin{tabular}{|c|c|c|c|c|} 
\hline
$(G,K_{1})$                                      &$(m_{1},m_{2}, m_{3})$& $\Delta$                          &Theorem~\ref{Theorem:cohom2}&$\mathrm{codim}$ \\ \hline \hline
$(\mathrm{SU}(4+n), \mathrm{S}(\mathrm{U}(2)\times \mathrm{U}(2+n)))$& $(2n,2,1)$& $\{\alpha_{1}, \delta \}$          &(ii)   &$5$          \\ \cline{3-5}
                                                       &               & $\{\alpha_{2}, \delta \}$                    &(ii)   &$2n+2$          \\ \cline{3-5}
                                                       &               & $\{\alpha_{1}, \alpha_{2} \}$                &(ii)   &$3$          \\ \hline
$(\mathrm{SO}(10), \mathrm{U}(5))$                     & $(4,4,1)$     & $\{\alpha_{1}, \delta \}$                    &(ii)   &$7$          \\ \cline{3-5}
                                                       &               & $\{\alpha_{2}, \delta \}$                    &(ii)   &$6$         \\ \cline{3-5}
                                                       &               & $\{\alpha_{1}, \alpha_{2} \}$                &(ii)   &$3$         \\ \hline
$(\mathrm{Sp}(4+n), \mathrm{Sp}(2)\times \mathrm{Sp}(2+n))$&$(4n,4,3)$ & $\{\alpha_{1}, \delta \}$                    &(ii)   &$9$          \\ \cline{3-5}
                                                       &               & $\{\alpha_{1}, \delta \}$                    &(ii)   &$4n+2$         \\ \cline{3-5}
                                                       &               & $\{\alpha_{1}, \alpha_{2} \}$                &(ii)   &$5$         \\ \hline
$(E_{6}, \mathrm{T}^{1}\cdot \mathrm{Spin}(10))$       &$(8,6,1)$      & $\{\alpha_{1}, \delta \}$                    &(ii)   &$9$          \\ \cline{3-5}
                                                       &               & $\{\alpha_{1}, \delta \}$                    &(ii)   &$10$        \\ \cline{3-5}
                                                       &               & $\{\alpha_{1}, \alpha_{2} \}$                &(ii)   &$3$         \\ \hline
\end{tabular}
\\[12pt]

{\bf Type $\rm{G}_{2}$}\\
\begin{tabular}{|c|c|c|c|c|} 
\hline
$(G,K_{1})$                                      &$(m_{1},m_{2})$& $\Delta$                                  &Theorem~\ref{Theorem:cohom2}&$\mathrm{codim}$ \\ \hline \hline
$(G_{2}, \mathrm{SO}(4))$                               &$(1,1)$       & $\{\alpha_{1}, \delta \}$                    &(ii)   &$3$          \\ \cline{3-5}
                                                       &               & $\{\alpha_{2}, \delta \}$                    &(ii)   &$3$          \\ \cline{3-5}
                                                       &               & $\{\alpha_{1}, \alpha_{2} \}$                &(iii)   &$3$          \\ \hline
$(G_{2}\times G_{2}, G_{2})$                           &$(2,2)$        & $\{\alpha_{1}, \delta \}$                    &(ii)   &$4$          \\ \cline{3-5}
                                                       &               & $\{\alpha_{2}, \delta \}$                    &(ii)   &$4$          \\ \cline{3-5}
                                                       &               & $\{\alpha_{1}, \alpha_{2} \}$                &(iii)   &$4$          \\ \hline
\end{tabular}
\\[24pt]

{\bf When $(\theta_{1}\not\sim \theta_{2})$}

{\bf Type $\rm{I\mathchar`-B}_{2}$}\\
\begin{tabular}{|c|c|c|c|c|} 
\hline
$(G,K_{1}, K_{2})$                             &$(m_{1},m_{2}, n_{1})$& $\Delta$                                  &Theorem~\ref{Theorem:cohom2}&$\mathrm{codim}$ \\ \hline \hline
$(\mathrm{SO}(2+a+b), $                        &$(b-2,1,a)$           & $\{\alpha_{1}, \tilde{\alpha} \}$                    &(i)   &$3$          \\ \cline{3-5}
$\mathrm{SO}(2+a)\times \mathrm{SO}(b),$       &                      & $\{\alpha_{2}, \tilde{\alpha} \}$                    &(ii)   &$b$          \\ \cline{3-5}
$\mathrm{SO}(2)\times \mathrm{SO}(a+b))$       &                      & $\{\alpha_{1}, \alpha_{2} \}$                        &(ii)   &$a+2$          \\ \hline
$(\mathrm{SO}(6)\times \mathrm{SO}(6),\Delta(\mathrm{SO}(6)\times \mathrm{SO}(6)), $      &$(2,2,2)$ &$\{\alpha_{1}, \tilde{\alpha} \}$                    &(i)   &$4$          \\ \cline{3-5}
$, K_{2})\ \text{with}\ (G_{\sigma})_{0}\cong \mathrm{SO}(3)\times \mathrm{SO}(3)$                                          &               & $\{\alpha_{2}, \tilde{\alpha} \}$                    &(ii)   &$4$          \\ \cline{3-5}
(C)                                                       &               & $\{\alpha_{1}, \alpha_{2} \}$                       &(ii)   &$4$          \\ \hline
\end{tabular}

 Here $(2<b, 1\leq a)$.
\\[12pt]

{\bf Type $\rm{I\mathchar`-C}_{2}$}\\
\begin{tabular}{|c|c|c|c|c|} 
\hline
$(G,K_{1}, K_{2})$                             &$(m_{1},m_{2}, n_{1})$& $\Delta$                                    &Theorem~\ref{Theorem:cohom2}&$\mathrm{codim}$ \\ \hline \hline
$(\mathrm{SO}(8), $                            &$(2,1,2)$& $\{\alpha_{1}, \tilde{\alpha} \}$                        &(i)   &$3$          \\ \cline{3-5}
$\mathrm{SO}(4)\times \mathrm{SO}(4), \mathrm{U}(4))$&                      & $\{\alpha_{2}, \tilde{\alpha} \}$     &(ii)   &$4$          \\ \cline{3-5}
                                                          &                      & $\{\alpha_{1}, \alpha_{2} \}$    &(ii)   &$4$          \\ \hline
$(\mathrm{SU}(4), \mathrm{SO}(4), $                       &$(1,1,1)$             & $\{\alpha_{1}, \tilde{\alpha} \}$&(i)   &$3$          \\ \cline{3-5}
$\mathrm{S}(\mathrm{U}(2)\times \mathrm{U}(2)) )$         &                      & $\{\alpha_{2}, \tilde{\alpha} \}$&(ii)   &$3$          \\ \cline{3-5}
                                                       &                         & $\{\alpha_{1}, \alpha_{2} \}$    &(ii)   &$3$          \\ \hline
$(\mathrm{SU}(4)\times \mathrm{SU}(4),\Delta(\mathrm{SU}(4)\times \mathrm{SU}(4) ), $                        &$(2,2,2)$             & $\{\alpha_{1}, \tilde{\alpha} \}$&(i)   &$4$          \\ \cline{3-5}
$K_{2})\ \text{with}\ (G_{\sigma})_{0}\cong \mathrm{SO}(4)$\ (C)                                              &                      & $\{\alpha_{2}, \tilde{\alpha} \}$&(ii)   &$4$          \\ \cline{3-5}
                                                       &                         & $\{\alpha_{1}, \alpha_{2} \}$    &(ii)   &$4$          \\ \hline
$(\mathrm{SU}(4)\times \mathrm{SU}(4),\Delta(\mathrm{SU}(4)\times \mathrm{SU}(4)), $                        &$(2,2,2)$             & $\{\alpha_{1}, \tilde{\alpha} \}$&(i)   &$4$          \\ \cline{3-5}
$K_{2})\ \text{with}\ (G_{\sigma})_{0}\cong  \mathrm{Sp}(2)$\ (C)                                           &                      & $\{\alpha_{2}, \tilde{\alpha} \}$&(ii)   &$4$          \\ \cline{3-5}
                                                       &                         & $\{\alpha_{1}, \alpha_{2} \}$    &(ii)   &$4$          \\ \hline
\end{tabular}
 \\[12pt]

 {\bf Type $\rm{I\mathchar`-BC}_{2}\rm{\mathchar`-A}_{1}^{2}$}\\
\begin{tabular}{|c|c|c|c|c|} 
\hline
$(G,K_{1}, K_{2})$                                  &$(m_{1},m_{2},m_{3},n_{1})$& $\Delta$                             &Theorem~\ref{Theorem:cohom2}&$\mathrm{codim}$ \\ \hline \hline
$(\mathrm{SU}(2+a+b), $                             &$(2(b-2),2,1,2a)$          & $\{\alpha_{1}, \tilde{\alpha} \}$    &(ii)   &$2b-1$         \\ \cline{3-5}
$\mathrm{S}(\mathrm{U}(2+a)\times \mathrm{U}(b)),$  &                           & $\{\alpha_{2}, \tilde{\alpha} \}$    &(ii)   &$4$          \\ \cline{3-5}
$ \mathrm{S}(\mathrm{U}(2)\times \mathrm{U}(a+b)))$ &                           & $\{\alpha_{1}, \alpha_{2} \}$        &(ii)   &$2(a+1)$          \\ \hline
$(\mathrm{Sp}(2+a+b), $                             &$(4(b-1),4,3,4a)$          & $\{\alpha_{1}, \tilde{\alpha} \}$    &(ii)   &$4b+1$          \\ \cline{3-5}
$  \mathrm{Sp}(2+s)\times \mathrm{Sp}(t), $         &                           & $\{\alpha_{2}, \tilde{\alpha} \}$    &(ii)   &$6$          \\ \cline{3-5}
$\mathrm{Sp}(2)\times \mathrm{Sp}(s+t)  )$          &                           & $\{\alpha_{1}, \alpha_{2} \}$        &(ii)   &$4a+2$          \\ \hline
$(\mathrm{SO}(12), \mathrm{U}(6), \mathrm{U}(6)')$  &$(4,4,1,4)$                & $\{\alpha_{1}, \tilde{\alpha} \}$    &(ii)   &$7$          \\ \cline{3-5}
                                                          &                     & $\{\alpha_{2}, \tilde{\alpha} \}$    &(ii)   &$6$          \\ \cline{3-5}
                                                       &                        & $\{\alpha_{1}, \alpha_{2} \}$        &(ii)   &$6$          \\ \hline
\end{tabular}

Here $2<b, 1\leq a$.
\vspace{12pt}

{\bf Type $\rm{I\mathchar`-BC}_{2}\rm{\mathchar`-B}_{2}$}\\
\begin{tabular}{|c|c|c|c|c|} 
\hline
$(G,K_{1}, K_{2})$                                  &$(m_{1},m_{2},m_{3},n_{2})$& $\Delta$                             &Theorem~\ref{Theorem:cohom2}&$\mathrm{codim}$ \\ \hline \hline
$(\mathrm{SO}(4+2a),$                               &$(2(a-2),2,1,2)$           & $\{\alpha_{1}, \tilde{\alpha} \}$    &(ii)   &$2a-1$         \\ \cline{3-5}
$ \mathrm{SO}(4)\times \mathrm{SO}(2a),$            &                           & $\{\alpha_{2}, \tilde{\alpha} \}$    &(ii)   &$4$          \\ \cline{3-5}
$\mathrm{U}(2+a))$                                  &                           & $\{\alpha_{1}, \alpha_{2} \}$        &(iii)   &$4$          \\ \hline
$(E_{6},$                                           &$(4,4,1,2)$          & $\{\alpha_{1}, \tilde{\alpha} \}$          &(ii)   &$7$          \\ \cline{3-5}
$\mathrm{SU}(6)\cdot \mathrm{SU}(2), $              &                           & $\{\alpha_{2}, \tilde{\alpha} \}$    &(iii)   &$6$          \\ \cline{3-5}
$\mathrm{SO}(10)\cdot \mathrm{U}(1))$               &                           & $\{\alpha_{1}, \alpha_{2} \}$        &(iii)   &$4$          \\ \hline
$(E_{7},$                                           &$(8,6,1,2)$                & $\{\alpha_{1}, \tilde{\alpha} \}$    &(ii)   &$11$          \\ \cline{3-5}
$\mathrm{SO}(12)\cdot \mathrm{SU}(2),$              &                           & $\{\alpha_{2}, \tilde{\alpha} \}$    &(iii)   &$8$          \\ \cline{3-5}
$E_{6}\cdot \mathrm{U}(1))$                         &                        & $\{\alpha_{1}, \alpha_{2} \}$           &(iii)   &$4$          \\ \hline
\end{tabular}

Here $2\leq a$.
\vspace{12pt}

 {\bf Type $\rm{II\mathchar`-BC}_{2}$}\\
\begin{tabular}{|c|c|c|c|c|} 
\hline
$(G,K_{1}, K_{2})$                                  &$(m_{1},m_{2},m_{3},n_{3})$& $\Delta$                             &Theorem~\ref{Theorem:cohom2}&$\mathrm{codim}$ \\ \hline \hline
$(\mathrm{SU}(2+a),$                                &$(a-2,1,1)$                & $\{\alpha_{1}, \tilde{\alpha} \}$    &(iii)   &$3$         \\ \cline{3-5}
$ \mathrm{SO}(2+a),$                                &                           & $\{\alpha_{2}, \tilde{\alpha} \}$    &(iii)   &$a$          \\ \cline{3-5}
$ \mathrm{S}(\mathrm{U}(2)\times \mathrm{U}(a)))$   &                           & $\{\alpha_{1}, \alpha_{2} \}$        &(iii)   &$3$          \\ \hline
$(\mathrm{SO}(10), $                                &$(2,2,1)$                  & $\{\alpha_{1}, \tilde{\alpha} \}$    &(iii)   &$4$          \\ \cline{3-5}
$\mathrm{SO}(5)\times \mathrm{SO}(5),$              &                           & $\{\alpha_{2}, \tilde{\alpha} \}$    &(iii)   &$4$          \\ \cline{3-5}
$ \mathrm{U}(5))$                                   &                           & $\{\alpha_{1}, \alpha_{2} \}$        &(iii)   &$3$          \\ \hline
$(E_{6}, $                                          &$(4,3,1)$                  & $\{\alpha_{1}, \tilde{\alpha} \}$    &(iii)   &$5$          \\ \cline{3-5}
$\mathrm{Sp}(4),$                                   &                           & $\{\alpha_{2}, \tilde{\alpha} \}$    &(iii)   &$6$          \\ \cline{3-5}
$\mathrm{SO}(10)\cdot \mathrm{U}(1))$               &                           & $\{\alpha_{1}, \alpha_{2} \}$        &(iii)   &$3$          \\ \hline
\end{tabular}

 Here $2\leq a$.
 \vspace{12pt}

{\bf Type $\rm{III\mathchar`-A}_{2}$}\\
\begin{tabular}{|c|c|c|c|c|} 
\hline
$(G,K_{1}, K_{2})$                                  &$(m_{1},n_{1})$& $\Delta$                                       &Theorem~\ref{Theorem:cohom2}&$\mathrm{codim}$ \\ \hline \hline
$(\mathrm{SU}(6), \mathrm{Sp}(3), \mathrm{SO}(6))$  &$(2,2)$                    & $\{\alpha_{1}, \tilde{\alpha} \}$  &(iii)   &$4$         \\ \cline{3-5}
                                                    &                           & $\{\alpha_{2}, \tilde{\alpha} \}$  &(iii)   &$4$          \\ \cline{3-5}
                                                    &                           & $\{\alpha_{1}, \alpha_{2} \}$      &(iii)   &$4$          \\ \hline
$(E_{6}, \mathrm{Sp}(4),F_{4})$                     &$(4,4)$                    & $\{\alpha_{1}, \tilde{\alpha} \}$  &(iii)   &$6$          \\ \cline{3-5}
                                                    &                           & $\{\alpha_{2}, \tilde{\alpha} \}$  &(iii)   &$6$          \\ \cline{3-5}
                                                    &                           & $\{\alpha_{1}, \alpha_{2} \}$      &(iii)   &$6$          \\ \hline
$(U\times U,$                                       &$(m,m)$                    & $\{\alpha_{1}, \tilde{\alpha} \}$  &(iii)   &$m+2$          \\ \cline{3-5}
$\Delta(U\times U),$                                &                           & $\{\alpha_{2}, \tilde{\alpha} \}$  &(iii)   &$m+2$          \\ \cline{3-5}
$\overline{K}\times \overline{K})$, (B)             &                           & $\{\alpha_{1}, \alpha_{2} \}$      &(iii)   &$m+2$          \\ \hline
\end{tabular}\\
Here $m$ is the multiplicity of the root system of the symmetric pair $(U, \overline{K})$ of type $\mathrm{A}_{2}$.
\vspace{12pt}

{\bf Type $\rm{III\mathchar`-B}_{2}$}\\
\begin{tabular}{|c|c|c|c|c|} 
\hline
$(G,K_{1}, K_{2})$                                  &$(m_{1},m_{1}, n_{2})$& $\Delta$                                  &Theorem~\ref{Theorem:cohom2}&$\mathrm{codim}$ \\ \hline \hline
$(U\times U,$                                       &$(m,n,n)$                   & $\{\alpha_{1}, \tilde{\alpha} \}$   &(iii)   &$ m+2$          \\ \cline{3-5}
$\Delta(U\times U),$                                &                           & $\{\alpha_{2}, \tilde{\alpha} \}$    &(iii)   &$n+2$          \\ \cline{3-5}
$\overline{K}\times \overline{K})$, (B)             &                           & $\{\alpha_{1}, \alpha_{2} \}$        &(iii)   &$n+2$          \\ \hline
\end{tabular}\\
Here $(m ,n)$ is the multiplicity of the root system of the symmetric pair $(U, \overline{K})$ of type $\mathrm{B}_{2}$.
\vspace{12pt}

{\bf Type $\rm{III\mathchar`-C}_{2}$}\\
\begin{tabular}{|c|c|c|c|c|} 
\hline
$(G,K_{1}, K_{2})$                                  &$(m_{1},m_{1}, n_{2})$& $\Delta$                               &Theorem~\ref{Theorem:cohom2}&$\mathrm{codim}$ \\ \hline \hline
$(\mathrm{SU}(8),$                                  &$(4,3,1)$                  & $\{\alpha_{1}, \tilde{\alpha} \}$ &(i)   &$6$          \\ \cline{3-5}
$\mathrm{S}(\mathrm{U}(4)\times \mathrm{U}(4)),$    &                           & $\{\alpha_{2}, \tilde{\alpha} \}$ &(iii)   &$5$          \\ \cline{3-5}
$\mathrm{Sp}(4))$                                   &                           & $\{\alpha_{1}, \alpha_{2} \}$     &(iii)   &$3$          \\ \hline
$(\mathrm{Sp}(4),$                                  &$(2,1,2)$                  & $\{\alpha_{1}, \tilde{\alpha} \}$ &(i)   &$4$          \\ \cline{3-5}
$ \mathrm{U}(4), $                                  &                           & $\{\alpha_{2}, \tilde{\alpha} \}$ &(iii)   &$3$          \\ \cline{3-5}
$\mathrm{Sp}(2)\times \mathrm{Sp}(2))$              &                           & $\{\alpha_{1}, \alpha_{2} \}$     &(iii)   &$4$          \\ \hline
$(U\times U,$                                       &$(m,n,n)$                  & $\{\alpha_{1}, \tilde{\alpha} \}$ &(iii)   &$m+2$          \\ \cline{3-5}
$\Delta(U\times U),$                                &                           & $\{\alpha_{2}, \tilde{\alpha} \}$ &(iii)   &$n+2$          \\ \cline{3-5}
$\overline{K}\times \overline{K})$, (B)             &                           & $\{\alpha_{1}, \alpha_{2} \}$     &(iii)   &$n+2$          \\ \hline
\end{tabular}\\
Here $(m ,n)$ is the multiplicity of the root system of the symmetric pair $(U, \overline{K})$ of type $\mathrm{C}_{2}$.
\vspace{12pt}

{\bf Type $\rm{III\mathchar`-BC}_{2}$}\\
\begin{tabular}{|c|c|c|c|c|} 
\hline
$(G,K_{1}, K_{2})$                                  &$(m_{1},m_{1}, m_{3}, n_{3})$& $\Delta$                        &Theorem~\ref{Theorem:cohom2}&$\mathrm{codim}$ \\ \hline \hline
$(\mathrm{SU}(4+2s), $                              &$(4(s-2),4,3,1)$           & $\{\alpha_{1}, \tilde{\alpha} \}$ &(iii)   &$6$          \\ \cline{3-5}
$\mathrm{S}(\mathrm{U}(4)\times \mathrm{U}(2s)),$   &                           & $\{\alpha_{2}, \tilde{\alpha} \}$ &(iii)   &$4s-3$          \\ \cline{3-5}
$ \mathrm{Sp}(2+s))$                                &                           & $\{\alpha_{1}, \alpha_{2} \}$     &(iii)   &$3$          \\ \hline
$(\mathrm{SU}(10),$                                 &$(4,4,1,3)$                & $\{\alpha_{1}, \tilde{\alpha} \}$ &(iii)   &$6$          \\ \cline{3-5}
$  \mathrm{S}(\mathrm{U}(5)\times \mathrm{U}(5)) ,$ &                           & $\{\alpha_{2}, \tilde{\alpha} \}$ &(iii)   &$7$          \\ \cline{3-5}
$ \mathrm{Sp}(5))$                                  &                           & $\{\alpha_{1}, \alpha_{2} \}$     &(iii)   &$5$          \\ \hline
$(U\times U,$                                       &$(m,n,l,l)$                & $\{\alpha_{1}, \tilde{\alpha} \}$ &(iii)   &$n+2$          \\ \cline{3-5}
$\Delta(U\times U),$                                &                           & $\{\alpha_{2}, \tilde{\alpha} \}$ &(iii)   &$m+l+2$          \\ \cline{3-5}
$\overline{K}\times \overline{K})$, (B)             &                           & $\{\alpha_{1}, \alpha_{2} \}$     &(iii)   &$l+2$          \\ \hline
\end{tabular}\\
Here $(m ,n, l)$ is the multiplicity of the root system of the symmetric pair $(U, \overline{K})$ of type $\mathrm{BC}_{2}$.
\vspace{12pt}

{\bf Type $\rm{III\mathchar`-G}_{2}$}\\
\begin{tabular}{|c|c|c|c|c|} 
\hline
$(G,K_{1}, K_{2})$                                  &$(m_{1},m_{1}, n_{1}, n_{2})$& $\Delta$                          &Theorem~\ref{Theorem:cohom2}&$\mathrm{codim}$ \\ \hline \hline
$(U\times U,$                                       &$(m,n,m,n)$                  & $\{\alpha_{1}, \tilde{\alpha} \}$ &(iii)   &$n+2$          \\ \cline{3-5}
$\Delta(U\times U),$                                &                           & $\{\alpha_{2}, \tilde{\alpha} \}$   &(iii)   &$m+2$          \\ \cline{3-5}
$\overline{K}\times \overline{K})$, (B)             &                           & $\{\alpha_{1}, \alpha_{2} \}$       &(iii)   &$n+2$          \\ \hline
\end{tabular}\\
Here $(m ,n)$ is the multiplicity of the root system of the symmetric pair $(U, \overline{K})$ of type $\mathrm{G}_{2}$.

\section{Biharmonic homogeneous hypersurfaces in compact Lie groups}
\label{sect:Biharmonic homogeneous hypersurfaces in compact Lie groups}

In this section, applying Corollary~\ref{cohom1assoc}
we will study biharmonic regular orbits of cohomogeneity one $(K_{2}\times K_{1})$-actions on compact Lie groups.

Let $(G,K_{1},K_{2})$ be a commutative compact symmetric triad
where $G$ is semisimple.
Hereafter we assume that $\dim \mathfrak{a} = 1$. 
Then the orbit space of $(K_{2}\times K_{1})$-action is homeomorphic to a closed interval.
A point in the interior of the orbit space corresponds to a regular orbit,
and there exists a unique minimal (harmonic) orbit among regular orbits.
On the other hand, two endpoints of the orbit space correspond to singular orbits.
These singular orbits are minimal (harmonic), moreover these are weakly reflective (\cite{IST2}).
For $H \in \mathfrak{a}$, we set $x=\exp (H)$ and consider the orbit $(K_{2}\times K_{1})\cdot x$ 
of the $(K_{2}\times K_{1})$-action on $G$ through $x$.
For simplicity, we denote the tension field $dL_{x}^{-1}(\tau_{H})_{x}$ of the orbit $(K_{2}\times K_{1})\cdot x$ in $G$ by $\tau_{H}$ for $H \in \mathfrak{a}$.
\vskip0.6cm\par

{\bf Cases of $\theta_{1}=\theta_{2}$.}
First, we consider the cases of $\theta_{1}=\theta_{2}$. 
Then $(G, K_{1})=(G, K_{2})$ is a compact symmetric pair of rank one.
The restricted root type system of $(G, K_{1})$ is of type $\rm{B}_{1}$ or $\rm{BC}_{1}$.
Let $\vartheta := \langle \delta , H \rangle$ for $H \in \mathfrak{a}$, where $\delta$ is the highest root of $\tilde{\Sigma}$.
Then, by (\ref{eq:cell_of_isotropy_action}), $P_0 = \{ H \in \mathfrak{a} \mid 0 < \vartheta < \pi \}$ is a cell
in these types.

\subsection{Type $\rm{B}_{1}$}
We set $\Sigma^{+}=\{\alpha \}, W^{+}=\emptyset$ and $m=m(\alpha)$.
In this case, $\vartheta$ satisfies $0<\vartheta <\pi$.
By Corollary~\ref{cor:2} we have 
$$
\tau_{H}=-m\cot \vartheta \alpha. 
$$
Hence $\tau_{H}=0$ if and only if $\vartheta = \pi /2$ holds.
By Corollary~\ref{cohom1assoc}, $(K_{2}\times K_{1})\cdot x$ in $G$ is biharmonic if and only if 
$$
0=m\langle \tau_{H}, \alpha \rangle \left(\frac{3}{2}- (\cot \vartheta )^{2} \right)
$$
holds. 
Therefore, 
the orbit $(K_{2}\times K_{1})\cdot x$ is proper biharmonic if and only if 
$$
(\cot \vartheta)^{2} = \frac{3}{2}  
$$
holds.
{\em In this case, there exist exactly two proper biharmonic hypersurfaces
which are regular orbits of the $(K_{2}\times K_{1})$-action on $G$.}

\subsection{Type $\rm{BC}_{1}$}
We set $\Sigma^{+}=\{\alpha , 2\alpha \}, W^{+}=\emptyset$ and $m_{1}=m(\alpha), m_{2}=m(2\alpha)$.
$\vartheta$ satisfies $0<\vartheta <\pi$.
By Corollary~\ref{cor:2} we have 
\begin{align*}
\tau_{H}
&=-m_{1}\cot \left( \frac{\vartheta}{2}\right) \alpha 
-m_{2}(\cot \vartheta ) 2\alpha \\
&=-\left\{ (m_{1}+m_{2})\cot \left( \frac{\vartheta}{2}\right) -m_{2}\cot \left( \frac{\vartheta}{2}\right) \right\}\alpha
. 
\end{align*}
Hence $\tau_{H}=0$ if and only if 
$$
\left( \cot \frac{\vartheta}{2} \right)^{2} = \frac{m_{2}}{m_{1}+m_{2}}
$$ 
holds.
By Corollary~\ref{cohom1assoc}, $(K_{2}\times K_{1})\cdot x$ in $G$ is biharmonic if and only if 
$$
0=\langle \tau_{H}, \alpha \rangle \left\{ 
m_{1}\left(\frac{3}{2}- \left( \cot \frac{\vartheta}{2} \right)^{2} \right)
+4m_{2}\left(\frac{3}{2}- (\cot \vartheta)^{2} \right)\right\}
$$
holds. 
Therefore, 
the orbit $(K_{2}\times K_{1})\cdot x$ is biharmonic if and only if $\tau_{H}=0$ or 
$$
0=
m_{1}\left(\frac{3}{2}- \left( \cot \frac{\vartheta}{2} \right)^{2} \right)
+4m_{2}\left(\frac{3}{2}- (\cot \vartheta)^{2} \right)
$$
holds.
The last equation is equivalent to 
\begin{align*}
\left( (m_{1}+m_{2})\left( \cot \frac{\vartheta}{2} \right)^{2} -m_{2}\right) \left( \left( \cot \frac{\vartheta}{2} \right)^{2} -1\right) =\left( \frac{m_{1}}{2} + 6m_{2}\right) \left( \cot \frac{\vartheta}{2} \right)^{2}.
\end{align*}
Since $(m_{1}/2)+6m_{2}>0$, the solutions of the equation are not harmonic.
Hence the orbit $(K_{2}\times K_{1})\cdot x$ is proper biharmonic if and only if 
$$
\left( \cot \left( \frac{\vartheta}{2}\right) \right)^{2} = \frac{3m_{1}+16m_{2}\pm \sqrt{(3m_{1}+16m_{2})^{2}-16(m_{1}+m_{2})m_{2}}}{4(m_{1}+m_{2})}  
$$
holds.
{\em In this case, there exist exactly two proper biharmonic hypersurfaces which are regular orbits of the $(K_{2}\times K_{1})$-action on $G$.}
\vspace{12pt}

{\bf Cases of $\theta_{1}\not\sim \theta_{2}$.}
If $G$ is simple and $\theta_1 \not\sim \theta_2$,
then for a commutative compact symmetric triad $(G, K_1, K_2)$ the triple $(\tilde{\Sigma}, \Sigma, W)$ is a symmetric triad
with multiplicities $m(\lambda)$ and $n(\alpha)$ (cf. Theorem~\ref{Theorem:I2 AB}).

All the symmetric triads with $\dim \mathfrak{a}=1$ are classified into the following four types (\cite{I1}):
\begin{center}
\begin{tabular}{|l|c|c|c|} 
\hline                            & $\Sigma^{+}$           & $W^{+}$               & $\tilde{\alpha}$ \\ \hline \hline
$\rm{III\mathchar`-B}_{1}$  & $\{\alpha \}$          & $\{\alpha \}$         & $\alpha$        \\ \hline
$\rm{I\mathchar`-BC}_{1}$   & $\{\alpha, 2\alpha \}$ & $\{\alpha \}$         & $\alpha$        \\ \hline
$\rm{II\mathchar`-BC}_{1}$  & $\{\alpha \}$          & $\{\alpha, 2\alpha \}$& $2\alpha$       \\ \hline
$\rm{III\mathchar`-BC}_{1}$ & $\{\alpha, 2\alpha \}$ & $\{\alpha, 2\alpha \}$& $2\alpha$       \\ 
\hline
\end{tabular}
\end{center}

\medskip\par
Let $\vartheta := \langle \tilde{\alpha}, H \rangle$ for $H \in \mathfrak{a}$.
Then, by (\ref{eq:cell_of_Hermann_action}), $P_0 = \{ H \in \mathfrak{a} \mid 0 < \vartheta < \pi/2\}$ is a cell
in these types.
If $G$ is simply connected and $K_{1}$ and $K_{2}$ are connected, then the orbit space of the $(K_2\times K_{1})$-action on $G$ is identified with 
$\overline{P_0} = \{ H \in \mathfrak{a} \mid 0 \leq \vartheta \leq \pi/2\}$,
more precisely, each orbit meets $\exp \overline{P_0}$ at one point.
For each orbit $(K_{2}\times K_{1})\cdot x$, $dL_{x}^{-1}(\tau_{H})_{x} \in \mathfrak{a}$ holds.
\subsection{Type $\rm{III\mathchar`-B}_{1}$}
We set $\Sigma^{+}=\{\alpha \}, W^{+}=\{\alpha\}$ and $m=m(\alpha), n=n(\alpha)$.
By Corollary~\ref{cor:2} we have 
\begin{align*}
\tau_{H}
&=-m\cot \vartheta \alpha 
+n\tan \vartheta \alpha . 
\end{align*}
Hence $\tau_{H}=0$ if and only if 
$$
(\cot \vartheta)^{2} = \frac{n}{m}
$$ 
holds.
By Corollary~\ref{cohom1assoc}, $(K_{2}\times K_{1})\cdot x$ in $G$ is biharmonic if and only if 
$$
0=\langle \tau_{H}, \alpha \rangle \left\{ 
m\left(\frac{3}{2}- (\cot \vartheta )^{2} \right)
+n\left(\frac{3}{2}- (\tan \vartheta)^{2} \right)\right\}
$$
holds. 
Therefore, 
the orbit $(K_{2}\times K_{1})\cdot x$ is biharmonic if and only if $\tau_{H}=0$ or 
$$
0= 
m\left(\frac{3}{2}- (\cot \vartheta )^{2} \right)
+n\left(\frac{3}{2}- (\tan \vartheta)^{2} \right)
$$
holds.
The last equation is equivalent to 
\begin{align*}
0=&
m(\cot \vartheta)^{4}-\frac{3}{2}(m+n)(\cot \vartheta)^{2}+n\\
=&(m(\cot \vartheta)^{2}-n)(m(\cot \vartheta)^{2}-1)-\left( \frac{m+n}{2}\right)(\cot \vartheta)^{2} .
\end{align*}
Since $(m+n)/2>0$, the solutions of the equation are not harmonic.
Hence the orbit $(K_{2}\times K_{1})\cdot x$ is proper biharmonic if and only if 
$$
(\cot \vartheta)^{2} = \frac{3(m+n)\pm \sqrt{9(m+n)^{2}-16mn}}{4m}  
$$
holds.
{\em In this case, there exist exactly two proper biharmonic hypersurfaces which are regular orbits of the $(K_{2}\times K_{1})$-action on $G$.}

\subsection{Type $\rm{I\mathchar`-BC}_{1}$}
We set $\Sigma^{+}=\{\alpha, 2\alpha \}, W^{+}=\{\alpha\}$ and $m_{1}=m(\alpha), m_{2}=m(2\alpha), n=n(\alpha)$.
By Corollary~\ref{cor:2} we have 
\begin{align*}
\tau_{H}
&=-m_{1}\cot \vartheta \alpha 
-m_{2}\cot (2\vartheta) 2\alpha 
+n\tan \vartheta \alpha \\
&=\{ -(m_{1}+m_{2}) \cot \vartheta
+(m_{2}+n)\tan \vartheta\} \alpha.
\end{align*}
Hence $\tau_{H}=0$ if and only if 
$$
(\cot \vartheta)^{2} = \frac{m_{2}+n}{m_{1}+m_{2}}
$$ 
holds.
By Corollary~\ref{cohom1assoc}, $(K_{2}\times K_{1})\cdot x$ in $G$ is biharmonic if and only if 
$$
0=\langle \tau_{H}, \alpha \rangle \left\{ 
m_{1}\left(\frac{3}{2}- (\cot \vartheta )^{2} \right)
+4m_{2}\left(\frac{3}{2}- (\cot (2\vartheta) )^{2} \right)
+n\left(\frac{3}{2}- (\tan \vartheta)^{2} \right)\right\}
$$
holds. 
Therefore, 
the orbit $(K_{2}\times K_{1})\cdot x$ is biharmonic if and only if $\tau_{H}=0$ or 
$$
0=
m_{1}\left(\frac{3}{2}- (\cot \vartheta )^{2} \right)
+4m_{2}\left(\frac{3}{2}- (\cot (2\vartheta) )^{2} \right)
+n\left(\frac{3}{2}- (\tan \vartheta)^{2} \right)
$$
holds.
The last equation is equivalent to 
\begin{align*}
0=&
(m_{1}+m_{2})(\cot \vartheta)^{4}-\left( \frac{3}{2}(m_{1}+n)+8m_{2}\right)(\cot \vartheta)^{2}+(m_{2}+n)\\
=&((m_{2}+m_{2})(\cot \vartheta)^{2}-(m_{2}+n))((\cot \vartheta)^{2}-1)-\left( \frac{m_{1}+n}{2}+6m_{2}\right)(\cot \vartheta)^{2}.
\end{align*}
Since $(m_{1}+n)/2+6m_{2}>0$, the solutions of the equation are not harmonic.
Hence the orbit $(K_{2}\times K_{1})\cdot x$ is proper biharmonic if and only if 

$$
(\cot \vartheta)^{2} = \frac{3(m+n)+16m_{2}\pm \sqrt{(3(m+n)+16m_{2})^2-16(m_{1}+m_{2})(m_{2}+n)}}{4(m_{1}+m_{2})}  
$$
holds.
{\em In this case, there exist exactly two proper biharmonic hypersurfaces
which are regular orbits of the $(K_{2}\times K_{1})$-action on $G$.}

\subsection{Type $\rm{II\mathchar`-BC}_{1}$}
We set $\Sigma^{+}=\{\alpha\}, W^{+}=\{\alpha, 2\alpha \}$ and $m=m(\alpha)=n(\alpha), n=n(2\alpha)$.
By Corollary~\ref{cor:2} we have 
\begin{align*}
\tau_{H}
&=-m\cot \left( \frac{\vartheta}{2}\right) \alpha 
+m\tan \left( \frac{\vartheta}{2}\right) \alpha
+n\tan (\vartheta) 2\alpha \\
&=2\{ -m \cot \vartheta
+n\tan \vartheta\} \alpha.
\end{align*}
Hence $\tau_{H}=0$ if and only if 
$$
(\cot \vartheta)^{2} = \frac{n}{m}
$$ 
holds.
By Corollary~\ref{cohom1assoc}, $(K_{2}\times K_{1})\cdot x$ in $G$ is biharmonic if and only if 
$$
0=\langle \tau_{H}, \alpha \rangle 
\left\{ 
m\left(\frac{3}{2}- \left( \cot \frac{\vartheta}{2} \right)^{2} \right)
+m\left(\frac{3}{2}- \left( \tan \frac{\vartheta}{2} \right)^{2} \right)
+4n\left(\frac{3}{2}- (\tan \vartheta )^{2} \right)
\right\}
$$
holds. 
Therefore, 
the orbit $(K_{2}\times K_{1})\cdot x$ is biharmonic if and only if $\tau_{H}=0$ or 
$$
0= 
m\left(\frac{3}{2}- \left( \cot \frac{\vartheta}{2} \right)^{2} \right)
+m\left(\frac{3}{2}- \left( \tan \frac{\vartheta}{2} \right)^{2} \right)
+4n\left(\frac{3}{2}- (\tan \vartheta )^{2} \right)
$$
holds.
The last equation is equivalent to 
\begin{align*}
0=&
4m(\cot \vartheta)^{4}-\left( m+6n \right)(\cot \vartheta)^{2}+4n\\
=&4(m(\cot \vartheta)^{2}-n)((\cot \vartheta)^{2}-1)-(-3m+2n)(\cot \vartheta)^{2}.
\end{align*}
If $-3m+2n=0$, then the orbit $(K_{2}\times K_{1})\cdot x$ is proper biharmonic if and only if
$(\cot \vartheta)^{2}=1$ holds.
When $-3m+2n\neq 0$,
the solutions of the equation are not harmonic.
\begin{itemize}
\item 
{\em When $(m+6n)^2-64mn>0$, then 
the orbit $(K_{2}\times K_{1})\cdot x$ is proper biharmonic if and only if 
$$
(\cot \vartheta)^{2} = \frac{m+6n\pm \sqrt{(m+6n)^2-64mn}}{8m}  
$$
holds.
The condition $(m+6n)^2-64mn>0$ is equivalent to $m<(26-8\sqrt{10})n$ or $(26+8\sqrt{10})n<m$.
In this case, 
if $(m+6n)^2-64mn>0$, then 
there exist exactly
two proper biharmonic hypersurfaces which are regular orbits of the $(K_{2}\times K_{1})$-action on $G$.}
\item {\em When $(m+6n)^2-64mn<0$, biharmonic regular orbits of the $(K_{2}\times K_{1})$-action
on $G$ is harmonic.}
\end{itemize}
\begin{remark}\rm
By the classification of compact symmetric triads, we can see that there is no compact symmetric triad which satisfies $(m+6n)^2-64mn=0$.
\end{remark}
\subsection{Type $\rm{III\mathchar`-BC}_{1}$}
We set $\Sigma^{+}=\{\alpha, 2\alpha \}, W^{+}=\{\alpha, 2\alpha \}$ and $m_{1}=m(\alpha)=n(\alpha), m_{2}=m(2\alpha), n=n(2\alpha)$.
By Corollary~\ref{cor:2} we have 
\begin{align*}
\tau_{H}
&=-m_{1}\cot \left( \frac{\vartheta}{2}\right) \alpha 
+m_{1}\tan \left( \frac{\vartheta}{2}\right) \alpha
+n\tan \vartheta (2\alpha)
-m_{2}\cot \vartheta (2\alpha) \\
&=2\{ -(m_{1}+m_{2}) \cot \vartheta
+n\tan \vartheta\} \alpha.
\end{align*}
Hence $\tau_{H}=0$ if and only if 
$$
(\cot \vartheta)^{2} = \frac{n}{m_{1}+m_{2}}
$$ 
holds.
By Corollary~\ref{cohom1assoc}, $(K_{2}\times K_{1})\cdot x$ in $G$ is biharmonic if and only if 
\begin{align*}
0=&\langle \tau_{H}, \alpha \rangle 
\left\{ 
m_{1}\left(\frac{3}{2}- \left(\cot \frac{\vartheta}{2} \right)^{2} \right)
+m_{1}\left(\frac{3}{2}- \left( \tan \frac{\vartheta}{2} \right)^{2} \right)\right. \\
&\left. +4n\left(\frac{3}{2}- (\tan \vartheta )^{2} \right)
+4m_{2}\left(\frac{3}{2}- (\cot \vartheta )^{2} \right)
\right\}
\end{align*}
holds. 
Therefore, 
the orbit $(K_{2}\times K_{1})\cdot x$ is biharmonic if and only if $\tau_{H}=0$ or 
\begin{align*}
0=& 
m_{1}\left(\frac{3}{2}- \left(\cot \frac{\vartheta}{2} \right)^{2} \right)
+m_{1}\left(\frac{3}{2}- \left( \tan \frac{\vartheta}{2} \right)^{2} \right) \\
& +4n\left(\frac{3}{2}- (\tan \vartheta )^{2} \right)
+4m_{2}\left(\frac{3}{2}- (\cot \vartheta )^{2} \right)
\end{align*}
holds.
The last equation is equivalent to 
\begin{align*}
0=&
4(m_{1}+m_{2})(\cot \vartheta)^{4}-\left( m_{1}+6m_{2}+6n \right)(\cot \vartheta)^{2}+4n\\
=&4((m_{1}+m_{2})(\cot \vartheta)^{2}-n)((\cot \vartheta)^{2}-1)-(-3m_{1}+m_{2}+2n)(\cot \vartheta)^{2}.
\end{align*}
If $-3m_{1}+m_{2}+2n=0$, then the orbit $(K_{2}\times K_{1})\cdot x$ is proper biharmonic if and only if
$(\cot \vartheta)^{2}=1$ holds.
When $-3m_{1}+m_{2}+2n \neq 0$,
the solutions of the equation are not harmonic.
\begin{itemize}
\item {\em When $(m_{1}+6m_{2}+6n)^2-64(m_{1}+m_{2})n>0$, then 
the orbit $(K_{2}\times K_{1})\cdot x$ is proper biharmonic if and only if 
$$
(\cot \vartheta)^{2} = \frac{m_{1}+6m_{2}+6n\pm \sqrt{(m_{1}+6m_{2}+6n)^2-64(m_{1}+m_{2})n}}{8(m_{1}+m_{2})}  
$$
holds.
In this case, 
if $(m_{1}+6m_{2}+6n)^2-64(m_{1}+m_{2})n>0$, 
then there exist exactly
two proper biharmonic hypersurfaces which are regular orbits of
the $(K_{2}\times K_{1})$-action on $G$.}
\item 
{\em When $(m_{1}+6m_{2}+6n)^2-64(m_{1}+m_{2})n<0$, 
biharmonic regular orbits of
the $(K_{2}\times K_{1})$-action on $G$ is harmonic.}
\end{itemize}

\begin{remark}\rm
By the classification of compact symmetric triads, we can see that there is no compact symmetric triad which satisfies $(m_{1}+6m_{2}+6n)^2-64(m_{1}+m_{2})n=0$.
\end{remark}
Let $b>0$, $c>1$ and $q>1$.
Each commutative compact symmetric triad $(G, K_{1}, K_{2})$
where $G$ is simple, $\theta_1 \not\sim \theta_2$ and $\dim \mathfrak{a}=1$
is one of the following (see \cite{I2}):
\vspace{12pt}

\noindent
{\bf Type $\rm{III\mathchar`-B}_{1}$}\\
\begin{tabular}{|c|c|} 
\hline
$(G,K_{1},K_{2})$                                                                   &$(m(\alpha),n(\alpha))$  \\ \hline \hline
$(\mathrm{SO}(1+b+c), \mathrm{SO}(1+b) \times \mathrm{SO}(c), \mathrm{SO}(b+c))$    &$(c-1,b)$\\ \hline
$(\mathrm{SU}(4), \mathrm{Sp}(2), \mathrm{SO}(4))$                                  &$(2,2)$  \\ \hline
$(\mathrm{SU}(4), \mathrm{S}(\mathrm{U}(2) \times \mathrm{U}(2)), \mathrm{Sp}(2))$ &$(3,1)$  \\ \hline
$(\mathrm{Sp}(2), \mathrm{U}(2), \mathrm{Sp}(1) \times \mathrm{Sp}(1))$             &$(1,2)$  \\ \hline
\end{tabular}
\vspace{12pt}

\noindent
{\bf Type $\rm{I\mathchar`-BC}_{1}$}\\
\begin{tabular}{|c|c|} 
\hline
$(G,K_{1},K_{2})$                                                                           &$(m(\alpha),m(2\alpha),n(\alpha))$  \\ \hline \hline
$(\mathrm{SO}(2+2q), \mathrm{SO}(2) \times \mathrm{SO}(2q), \mathrm{U}(1+q))$               &$(2(q-1),1,2(q-1))$\\ \hline
$(\mathrm{SU}(1+b+c), \mathrm{S}(\mathrm{U}(1+b) \times \mathrm{U}(c)), \mathrm{S}(\mathrm{U}(1) \times \mathrm{U}(b+c))$&$(2(c-1),1,2b)$\\ \hline
$(\mathrm{Sp}(1+b+c), \mathrm{Sp}(1+b) \times \mathrm{Sp}(c), \mathrm{Sp}(1) \times \mathrm{Sp}(b+c))$&$(4(c-1),3,4b)$\\ \hline
$(\mathrm{SO}(8), \mathrm{U}(4), \mathrm{U}(4)')$                                             &$(4,1,1)$ \\ \hline
\end{tabular}
\vspace{12pt}

\noindent
{\bf Type $\rm{II\mathchar`-BC}_{1}$}\\
\begin{tabular}{|c|c|} 
\hline
$(G,K_{1},K_{2})$                                                                           &$(m(\alpha),n(\alpha),n(2\alpha))$  \\ \hline \hline
$(\mathrm{SO}(6), \mathrm{U}(3), \mathrm{SO}(3) \times \mathrm{SO}(3))$                         &$(2,2,1)$ \\ \hline
$(\mathrm{SU}(1+q), \mathrm{SO}(1+q), \mathrm{S}(\mathrm{U}(1) \times \mathrm{U}(q)))$   &$(q-1,q-1,1)$\\ \hline
\end{tabular}
\vspace{12pt}

\noindent
{\bf Type $\rm{III\mathchar`-BC}_{1}$}\\
\begin{tabular}{|c|c|} 
\hline
$(G,K_{1},K_{2})$                                                                           &$(m(\alpha),m(2\alpha),n(\alpha),n(2\alpha))$  \\ \hline \hline
$(\mathrm{SU}(2+2q), \mathrm{S}(\mathrm{U}(2) \times \mathrm{U}(2q)), \mathrm{Sp}(1+q))$     &$(4(q-1),3,4(q-1),1)$\\ \hline
$(\mathrm{Sp}(1+q), \mathrm{U}(1+q), \mathrm{Sp}(1) \times \mathrm{Sp}(q)) $          &$(2(q-1),1,2(q-1),2)$\\ \hline
$(\mathrm{E}_{6}, \mathrm{SU}(6) \cdot \mathrm{SU}(2), \mathrm{F}_{4})$                         &$(8,3,8,5)$ \\ \hline
$(\mathrm{E}_{6}, \mathrm{SO}(10) \cdot \mathrm{U}(1), \mathrm{F}_{4})$                         &$(8,7,8,1)$ \\ \hline
$(\mathrm{F}_{4}, \mathrm{Sp}(3) \cdot \mathrm{Sp}(1), \mathrm{Spin}(9))$                         &$(4,3,4,4)$\\ \hline
\end{tabular}

\medskip
Here, we define 
$\mathrm{U}(4)'=\{g\in \mathrm{SO}(8) \mid JgJ^{-1}=g \}$, 
where
$$
J=\left[
\begin{array}{cc|cc}
& &I_{3} &\\
& & &-1\\ \hline
-I_{3}& & &\\
& 1& &
\end{array}
\right]
$$ 
and $I_l$ denotes the identity matrix of $l \times l$.
 
Summing up the above, 
we obtain the following theorem.
\begin{theorem}\label{thm:list_of_biharmonic_orbits_in_Lie_group}
Let $(G,K_1,K_2)$ be a commutative compact symmetric triad where $G$ is simple,
and suppose that the $(K_{2}\times K_{1})$-action on $G$ is of cohomogeneity one.
Then all the proper biharmonic hypersurfaces which are regular orbits of the $(K_{2}\times K_{1})$-action 
in the compact Lie group $G$
are classified into the following lists:
\begin{enumerate}
\item When $(G,K_1,K_2)$ is one of the following cases,
there exist exactly two distinct proper biharmonic hypersurfaces which are regular orbits of
the $(K_{2}\times K_{1})$-action on $G$.
\begin{enumerate}
\item[(1-1) ] $(\mathrm{SO}(1+b+c),\ \mathrm{SO}(1+b) \times \mathrm{SO}(c),\ \mathrm{SO}(b+c)) \quad (b>0,\ c>1)$%
\item[(1-2) ] $(\mathrm{SU}(4),\ \mathrm{Sp}(2),\ \mathrm{SO}(4))$
\item[(1-3) ] $(\mathrm{SU}(4),\ \mathrm{S}(\mathrm{U}(2) \times \mathrm{U}(2)),\ \mathrm{Sp}(2))$
\item[(1-4) ] $(\mathrm{Sp}(2),\ \mathrm{U}(2),\ \mathrm{Sp}(1) \times \mathrm{Sp}(1))$ 
\item[(1-5) ] $(\mathrm{SO}(2+2q),\ \mathrm{SO}(2) \times \mathrm{SO}(2q),\ \mathrm{U}(1+q)) \quad (q>1)$
\item[(1-6) ] $(\mathrm{SU}(1+b+c),\ \mathrm{S}(\mathrm{U}(1+b) \times \mathrm{U}(c)),\ \mathrm{S}(\mathrm{U}(1) \times \mathrm{U}(b+c)) \quad (b \geq 0,\ c>1)$
\item[(1-7) ] $(\mathrm{Sp}(1+b+c),\ \mathrm{Sp}(1+b) \times \mathrm{Sp}(c),\ \mathrm{Sp}(1) \times \mathrm{Sp}(b+c)) \quad (b \geq 0,\ c>1)$
\item[(1-8) ] $(\mathrm{SO}(8),\ \mathrm{U}(4),\ \mathrm{U}(4)')$
\item[(1-9) ] $(\mathrm{SU}(1+q),\ \mathrm{SO}(1+q),\ \mathrm{S}(\mathrm{U}(1) \times \mathrm{U}(q))) \quad (q>52)$
\item[(1-10) ] $(\mathrm{SU}(2+2q),\ \mathrm{S}(\mathrm{U}(2) \times \mathrm{U}(2q)),\ \mathrm{Sp}(1+q)) \quad (q>1)$
\item[(1-11) ] $(\mathrm{Sp}(1+q),\ \mathrm{U}(1+q),\ \mathrm{Sp}(1) \times \mathrm{Sp}(q)) \quad (q=2\  \text{or}\ q>45)$
\item[(1-12) ] $(\mathrm{E}_{6},\ \mathrm{SO}(10) \cdot \mathrm{U}(1),\ \mathrm{F}_{4})$
\item[(1-13) ] $(\mathrm{F}_{4},\ \mathrm{Sp}(3) \cdot \mathrm{Sp}(1),\ \mathrm{Spin}(9))$
\item[(1-14) ] $(\mathrm{SO}(1+q),\ \mathrm{SO}(q),\ \mathrm{SO}(q)) \quad (q>1)$
\item[(1-15) ] $(\mathrm{F}_4,\ \mathrm{Spin}(9),\ \mathrm{Spin}(9))$
\end{enumerate}

\item When $(G,K_1,K_2)$ is one of the following cases,
any biharmonic regular orbit of the $(K_{2}\times K_{1})$-action on $G$ is harmonic.
\begin{enumerate}
\item[(2-1) ] $(\mathrm{SO}(6),\ \mathrm{U}(3),\ \mathrm{SO}(3) \times \mathrm{SO}(3))$
\item[(2-2) ] $(\mathrm{SU}(1+q),\ \mathrm{SO}(1+q),\ \mathrm{S}(\mathrm{U}(1) \times \mathrm{U}(q))) \quad (52\geq q>1)$
\item[(2-3) ] $(\mathrm{Sp}(1+q),\ \mathrm{U}(1+q),\ \mathrm{Sp}(1) \times \mathrm{Sp}(q)) \quad (45\geq q>2)$
\item[(2-4) ] $(\mathrm{E}_{6},\ \mathrm{SU}(6) \cdot \mathrm{SU}(2),\ \mathrm{F}_{4})$
\end{enumerate}
\end{enumerate}
\end{theorem}

\end{document}